\documentclass[reqno]{amsart}
\usepackage{graphicx,amsmath,amssymb,cite}
\usepackage{enumitem}
\usepackage[dvips]{epsfig}
\usepackage{bm}
\numberwithin{figure}{section}
\usepackage{tikz}

\newtheorem{theorem}{Theorem}[section]
\newtheorem{lemma}{Lemma}[section]
\newtheorem{proposition}[lemma]{Proposition}

\newtheorem{definition}[lemma]{Definition}
\newtheorem{remark}[lemma]{Remark}
\newtheorem{problem}{Problem}[section]
\newtheorem*{convention}{Convention}

\numberwithin{equation}{section}

\begin{document}

\title[Contact discontinuities with nonzero swirl]{Contact discontinuities for 3-D axisymmetric flows with nonzero swirl}

\author{Myoungjean Bae}
\address{Department of Mathematical Sciences, KAIST, 291 Daehak-Ro, Yuseong-Gu, Daejeon 34141, Republic of Korea}
\email{mjbae@kaist.ac.kr}

\author{Jong-Seo Yoon}
\address{Department of Mathematical Sciences, KAIST, 291 Daehak-Ro, Yuseong-Gu, Daejeon 34141, Republic of Korea}
\email{newyjsk@kaist.ac.kr}



\date{\today}

\begin{abstract}
In this paper, we prove the existence and far-field behavior of an axisymmetric solution $(\rho,\mathbf{u},p)$ to the steady Euler system in an infinitely long cylindrical nozzle with a perturbed wall $r=R(x_1)$, containing a contact discontinuity $r=g_D(x_1)$. The background flow carries a nonzero swirl component $\beta_0\mathbf{e}_\theta$, where $\beta_0$ is not necessarily small. Consequently, the background entropy, Bernoulli function, and angular momentum density depend on the radial variable, and their radial derivatives are not necessarily small.
\end{abstract}
\maketitle

\section{Introduction}
Steady inviscid compressible flow of ideal polytropic gas in $\mathbb{R}^3$ is governed by the {\emph{steady Euler system}}
\cite{courant1999supersonic}:
\begin{equation}\label{Euler-system-B}
\left\{
\begin{aligned}
&\operatorname{div}\left(\rho\mathbf{u}\right)=0,\\
&\operatorname{div}\left(\rho\mathbf{u}\otimes\mathbf{u}\right)+\nabla p=\mathbf{0},\\
&\operatorname{div}\left(\rho\mathbf{u}B\right)=0, 
\end{aligned}
\right.
\end{equation}
where $\rho=\rho(\mathbf{x})$, $\mathbf{u}=(u_1, u_2,u_3)(\mathbf{x})$, $p=p(\mathbf{x})$, and $B=B(\mathbf{x})$ denote the density, velocity, pressure, and the Bernoulli function of the flow, respectively, at $\mathbf{x}=(x_1, x_2, x_3)\in \mathbb{R}^3$. For a constant $\gamma>1$ called the \emph{adiabatic exponent}, $B$ is defined by
\begin{equation}\label{Ber-inv}
B=\frac{1}{2}|\mathbf{u}|^2+\frac{\gamma p}{(\gamma-1)\rho}=\frac{1}{2}|\mathbf{u}|^2+\frac{\gamma}{\gamma-1}S\rho^{\gamma-1}.
\end{equation}
Here, $S=p/\rho^{\gamma}$ denotes the entropy.

Let $\Omega\subset\mathbb{R}^3$ be an open connected set, and let $\Gamma$ be a non-self-intersecting $C^1$ surface dividing $\Omega$ into two disjoint open subsets $\Omega^{\pm}$ such that $\Omega=\Omega^-\cup\Gamma\cup\Omega^+$.
\begin{definition} \label{def-wsol}
We define $\mathbf{U}=(\rho,\mathbf{u},p)\in[L^{\infty}_{\mathrm{loc}}(\Omega)\cap C^1_{\mathrm{loc}}(\Omega^{\pm})\cap C^0_{\mathrm{loc}}(\Omega^{\pm}\cup\Gamma)]^5$ to be a \emph{weak solution} to the Euler system \eqref{Euler-system-B} in $\Omega$ if the following properties are satisfied:
For any test function $\xi\in C_0^{\infty}(\Omega)$ and $j=1,2,3$,
\begin{equation*}
\int_{\Omega}\rho \mathbf{u}\cdot\nabla\xi\, d\mathbf{x}=\int_{\Omega}(\rho u_j\mathbf{u}+p\mathbf{e}_j)\cdot\nabla\xi\, d\mathbf{x}=\int_{\Omega}\rho \mathbf{u} B\cdot\nabla\xi\, d\mathbf{x}=0,
\end{equation*}
where each $\mathbf{e}_j$ is the unit vector in the $x_j$-direction.
\end{definition}

By integration by parts, one can check that $\mathbf{U}$ is a weak solution to \eqref{Euler-system-B} in $\Omega$ if and only if $\mathbf{U}$ is a classical solution to \eqref{Euler-system-B} in $\Omega^\pm$ and $\mathbf{U}$ satisfies the Rankine--Hugoniot conditions:
\begin{align}
[\rho\mathbf{u}\cdot\mathbf{n}]_{\Gamma}=[\rho\mathbf{u}\cdot\mathbf{n}B]_{\Gamma}&=0,\nonumber\\
\label{R-H-2}[\rho(\mathbf{u}\cdot\mathbf{n})\mathbf{u}+p\mathbf{n}]_{\Gamma}&=\mathbf{0},
\end{align}
for a unit normal vector field $\mathbf{n}$ on $\Gamma$, where $[F]_{\Gamma}$ is defined by
$$[F(\mathbf{x})]_{\Gamma}:=\left.F(\mathbf{x})\right|_{\overline{\Omega^-}}-\left.F(\mathbf{x})\right|_{\overline{\Omega^+}}\quad\mbox{for}\quad \mathbf{x}\in\Gamma.$$

Let $\boldsymbol{\tau}_1$ and $\boldsymbol{\tau}_2$ be tangent vector fields on $\Gamma$ such that they are linearly independent at each point on $\Gamma$.
Taking the inner product of \eqref{R-H-2} with $\mathbf{n}$ and $\boldsymbol{\tau}_k$, respectively, yields
\begin{equation}\label{R-H-3}
[\rho(\mathbf{u}\cdot\mathbf{n})^2+p]_{\Gamma}=0,\quad\rho(\mathbf{u}\cdot\mathbf{n})[\mathbf{u}\cdot\boldsymbol{\tau}_k]_{\Gamma}=0\quad\text{for $k=1,2$.}
\end{equation}

Assume that $\rho>0$ in $\Omega$. Then, the second condition in \eqref{R-H-3} holds if either $\mathbf{u}\cdot\mathbf{n}=0$ holds on $\Gamma$, or $[\mathbf{u}\cdot\boldsymbol{\tau}_k]_{\Gamma}=0$ for all $k=1,2$. If $\mathbf{u}\cdot\mathbf{n}\ne 0$ and $[\mathbf{u}\cdot\boldsymbol{\tau}_k]_{\Gamma}=0$ hold on $\Gamma$, then the surface $\Gamma$ is called a shock. If $\mathbf{u}\cdot \mathbf{n}=0$ and $[\mathbf{u}\cdot\boldsymbol{\tau}_k]_{\Gamma}\ne 0$, then $\Gamma$ is called a contact discontinuity. For shocks, see \cite{park2020shock,park2025shock} and the references therein. For contact discontinuities, $\mathbf{u}\cdot \mathbf{n}=0$ and the first equation in \eqref{R-H-3} give $[p]_\Gamma=0$. Then we get the Rankine--Hugoniot conditions corresponding to a contact discontinuity as follows: $$\mathbf{u}\cdot\mathbf{n}=0 \quad \text{on } \Gamma, \quad [p]_\Gamma=0.$$

\begin{definition}
\label{definition-wsol}
We define $\mathbf{U}=(\rho,\mathbf{u},p)\in[L^{\infty}_{\mathrm{loc}}(\Omega)\cap C^1_{\mathrm{loc}}(\Omega^{\pm})\cap C^0_{\mathrm{loc}}(\Omega^{\pm}\cup\Gamma)]^5$ to be a weak solution to the Euler system \eqref{Euler-system-B} in $\Omega$ with a {\emph{contact discontinuity $\Gamma$}} if the following properties hold:
\begin{itemize}
\item[(i)] $\Gamma$ is a non-self-intersecting $C^1$-surface dividing $\Omega$ into two open subsets $\Omega^{\pm}$ such that $\Omega=\Omega^+\cup\Gamma\cup \Omega^-$;

\item[(ii)] $\mathbf{U}$ is a classical solution to \eqref{Euler-system-B} in $\Omega^\pm$;

\item[(iii)] $\rho>0$ in $\overline{\Omega}$;

\item[(iv)] $\left(\mathbf{u}|_{\overline{\Omega^-}\cap \Gamma}-\mathbf{u}|_{\overline{\Omega^+}\cap \Gamma}\right)(\mathbf{x})\neq \mathbf{0}$ holds for all $\mathbf{x}\in \Gamma$;

\item[(v)] $\mathbf{u}\cdot\mathbf{n}|_{\overline{\Omega^-}\cap\Gamma}=\mathbf{u}\cdot\mathbf{n}|_{\overline{\Omega^+}\cap\Gamma}=0$, where $\mathbf{n}$ is a unit normal vector field on $\Gamma$;

\item[(vi)] $[p]_\Gamma=0$.

\end{itemize}
\end{definition}

Contact discontinuities are among the fundamental characteristic discontinuities arising in multidimensional systems of conservation laws. Across a contact discontinuity, the pressure remains continuous, whereas the density, entropy, and tangential velocity may undergo jumps. In contrast to a shock, the normal velocity vanishes on a contact discontinuity. Moreover, since the contact discontinuity is itself a streamline and the flow states adjacent to it are generally unknown, its stability analysis naturally leads to a nonlinear free boundary problem. A rigorous theory for contact discontinuities is therefore an essential ingredient in the mathematical analysis of vortex sheets, jet boundaries, Mach reflection and refraction, and interactions among elementary waves.

Contact discontinuities have been studied from several complementary viewpoints. In \cite{bae2013stability}, the structural stability of a flat contact discontinuity separating two uniform subsonic flows was established under perturbations of an infinitely long duct, together with the determination of the asymptotic states at the far fields. In \cite{bae2018contact}, subsonic contact discontinuities with nonzero vorticity were constructed in two-dimensional infinitely long nozzles by means of a Helmholtz decomposition. This approach was subsequently extended to the three-dimensional axisymmetric setting in \cite{bae2019contact}, where contact discontinuities with nonzero vorticity and nonzero swirl were constructed in infinitely long circular cylinders. In \cite{chen2017steady}, steady Euler flows with large vorticity and characteristic discontinuities in arbitrary infinitely long nozzles were obtained through compensated compactness. More recently, the finitely long curved-nozzle problem was considered in \cite{weng2025contactcurved}, where a Lagrangian transformation was used to straighten the contact discontinuity, and the free boundary was determined through weighted H\"older estimates and the implicit function theorem. For a finitely long axisymmetric cylinder, Weng and Zhang \cite{weng2025contact} established the existence and uniqueness of subsonic flows with a contact discontinuity by introducing a modified Lagrangian transformation to handle the singularity near the symmetry axis and employing a deformation-curl decomposition together with the implicit function theorem. Contact discontinuities in supersonic and transonic regimes have also been investigated in \cite{chen2013stability,chen2013well,gao2025contact,huang2019contact,huang2021contact,wang2015structural,weng2025contactsupersonic}, while related discontinuous configurations arising in Mach reflection, shock interaction, subsonic flows past airfoils with vortex lines, jet flow, and piston problems were studied in \cite{chen2006stability,chen2008mach,chen2008stability,chen2022subsonic,pei2022shock,zhang2025contact}.

A key precedent for the use of a nontrivial background flow generated by nonzero swirl is the work of Fang, Gao, Xiang, and Zhao \cite{fang2024transonic}. They studied three-dimensional axisymmetric transonic shocks in a finite cylindrical nozzle and constructed special nontrivial background shock solutions generated by nonzero swirl functions. In particular, their background solutions may carry large vorticity, and the nonzero swirl plays an essential role in the mechanism determining the location of the shock front. The free boundary considered in \cite{fang2024transonic} is a transonic shock separating different flow regimes, whereas the free boundary in the present paper is a contact discontinuity in an infinitely long nozzle with a perturbed wall. Nevertheless, the two problems share an important structural feature: the swirl is already present in the background state and produces a genuinely nonconstant background flow.

The present work builds on the three-dimensional axisymmetric framework developed in \cite{bae2019contact}, but differs from it in two essential respects. First, we consider an infinitely long nozzle with a perturbed wall rather than a straight circular cylinder. Second, and more importantly, nonzero swirl is incorporated into the background state itself, whereas the background flow considered in \cite{bae2019contact} has zero swirl. This distinction substantially changes the structure of the background solution. Indeed, for an axisymmetric flow with constant density and velocity, the radial momentum balance implies that the pressure cannot remain constant in the radial direction. This computation is carried out in Proposition~\ref{background}.

The radial dependence of the background pressure creates a further difficulty in the transport part of the Euler system. In the chosen background state, not only the angular momentum density but also the entropy and the Bernoulli function depend on $r$. Consequently, the background values of these transported quantities have nonvanishing radial derivatives. When their transport equations are solved along the perturbed streamlines, the displacement of the streamlines acts on these background gradients and generates additional error terms. Therefore, the transport estimates are not uniform with respect to arbitrarily large background swirl, and the iteration scheme cannot be closed without imposing an upper bound on $|\beta_0|$. This can be seen from the estimate established in Lemma~\ref{Pro-trans}.

The remainder of the paper is organized as follows. In Section~\ref{3D-sec-Main}, we formulate the free boundary problem and state the main existence and far-field results in Theorem~\ref{3D-MainThm}. In Section~\ref{3D-sec-Hel}, we reformulate the problem through a Helmholtz decomposition and state the solvability of the reformulated problem as Theorem~\ref{3D-Thm-HD}. In Section~\ref{3D-sec-Cut}, we solve the corresponding free boundary problems in cut-off domains and derive estimates that are uniform with respect to the cut-off length. These results are used to prove Theorem~\ref{3D-Thm-HD} and hence Theorem~\ref{3D-MainThm}(a). Finally, in Section~\ref{3D-sec-ex}, we analyze the far-field behavior and prove Theorem~\ref{3D-MainThm}(b).

\section{Main theorems}\label{3D-sec-Main}
Let $(x_1,r,\theta)$ be the cylindrical coordinates of $(x_1,x_2, x_3)\in\mathbb{R}^3$, that is,
$$(x_1,x_2,x_3)=(x_1,r\cos\theta,r\sin\theta),\quad r\ge0,\quad \theta\in\mathbb{T},$$
where $\mathbb{T}$ denotes the one-dimensional torus of period $2\pi$.
Accordingly, any scalar function $f(\mathbf{x})$ can be represented as $f(\mathbf{x})=f(x_1,r,\theta)$, and any vector field $\mathbf{F}(\mathbf{x})$ can be represented as
$$\mathbf{F}(\mathbf{x})=F_{x_1}(x_1,r,\theta)\mathbf{e}_{x_1}+F_r(x_1,r,\theta)\mathbf{e}_r+F_{\theta}(x_1,r,\theta)\mathbf{e}_{\theta}$$ for orthonormal vectors
$$\mathbf{e}_{x_1}=(1,0,0),\quad\mathbf{e}_r=(0,\cos\theta,\sin\theta),\quad\mathbf{e}_{\theta}=(0,-\sin\theta,\cos\theta).$$
\begin{definition}
\begin{itemize}
\item[(i)]
A scalar function $f(\mathbf{x})$ is said to be axisymmetric if it is independent of $\theta$.
\item[(ii)] A vector field $\mathbf{F}(\mathbf{x})$ is said to be axisymmetric if $\mathbf{F}=F_{x_1}\mathbf{e}_{x_1}+F_r\mathbf{e}_r+F_\theta\mathbf{e}_\theta$ for axisymmetric functions $F_{x_1}$, $F_r$, and $F_{\theta}$.
\end{itemize}
\end{definition}
Define
\begin{equation*}
\mathcal{N}_0:=\left\{\mathbf{x}\in\mathbb{R}^3:\ x_1>0,\ 0\le r<1\right\}.
\end{equation*}

We first construct an axisymmetric background state with a contact discontinuity along $r=\frac{1}{2}$:
\begin{itemize}
\item[(i)] In the inner layer $\mathcal{N}_0\cap\{r<\frac{1}{2}\}$, fix constants $\rho_0^->0$ and $p_0^->0$, and set the density, velocity, and pressure to be $\rho_0^-$, $\mathbf{0}$, and $p_0^-$, respectively.
\item[(ii)] In the outer layer $\mathcal{N}_0\cap\{r>\frac{1}{2}\}$, fix constants $\rho_0^+>0$, $u_0>0$, and $\beta_0\in\mathbb{R}$, and set the density, velocity, and pressure to be $\rho_0^+$, $u_0\mathbf{e}_{x_1}+\beta_0\mathbf{e}_\theta$, and $p_0^+(r):=\rho_0^+\beta_0^2\ln(2r)+p_0^-$, respectively.
\end{itemize}
The above construction is justified by the following proposition.

\begin{proposition}\label{background}
The piecewise smooth vector field
\begin{equation*}
\mathbf{U}_0(\mathbf{x}):=\left\{
\begin{aligned}
(\rho_0^+, u_0\mathbf{e}_{x_1}+\beta_0\mathbf{e}_\theta,p_0^+(r))\quad&\mbox{for}\quad r>\frac{1}{2},\\
(\rho_0^-,\mathbf{0},p_0^-)\quad&\mbox{for}\quad r <\frac{1}{2}
\end{aligned}\right.
\end{equation*}
is a weak solution of the steady Euler system \eqref{Euler-system-B} in $\mathcal{N}_0$ with a contact discontinuity $\mathcal{N}_0\cap \{r=\frac{1}{2}\}$ in the sense of Definition~\ref{definition-wsol}.
\end{proposition}

\begin{proof}
We verify that $\mathbf{U}_0$ satisfies the conditions in Definition~\ref{definition-wsol}.
It is clear that all equations in \eqref{Euler-system-B} are satisfied in the inner region $\mathcal{N}_0^-:=\mathcal{N}_0\cap\{r<\frac12\}$.
In $\mathcal N_0^+:=\mathcal{N}_0\cap\{r>\frac12\}$, we have $$(\rho,\mathbf u,p)=\bigl(\rho_0^+,\,u_0\mathbf e_{x_1}+\beta_0\mathbf e_\theta,\,p_0^+(r)\bigr),$$
where $$p_0^+(r)=\rho_0^+\beta_0^2\ln(2r)+p_0^-.$$
Since $\rho_0^+$, $u_0$ and $\beta_0$ are constants, $\operatorname{div} \left(\rho\mathbf{u}\right)=0$.
For an axisymmetric flow, the radial component of the momentum equation is $$ \rho\left(u_{x_1}\partial_{x_1}+u_r\partial_r\right)u_r-\frac{\rho u_\theta^2}{r}+\partial_r p=0 $$ where $(u_{x_1}, u_r, u_\theta)=(\mathbf{u}\cdot\mathbf{e}_{x_1},\mathbf{u}\cdot\mathbf{e}_r,\mathbf{u}\cdot\mathbf{e}_\theta)$.
For the background state, this equation reduces to $$-\frac{\rho_0^+\beta_0^2}{r}
+\partial_r p_0^+(r)=0,$$
which holds because $$\partial_r p_0^+(r)=\frac{\rho_0^+\beta_0^2}{r}.$$
The other momentum equations are immediate.
Since $B$ depends only on $r$ and $\mathbf{u}\cdot\mathbf{e}_r=0$, we also have $$\operatorname{div}\left(\rho\mathbf u B\right)=\rho\mathbf{u}\cdot\nabla B+B\operatorname{div}\left(\rho\mathbf{u}\right)=0.$$
Thus $\mathbf{U}_0$ is a classical solution to \eqref{Euler-system-B} in each of $\mathcal{N}_0^\pm$.

Next, let $\mathbf n=\mathbf e_r$ be the unit normal vector on $\mathcal{N}_0\cap\{r=\frac12\}$. Then
$$\mathbf u|_{\mathcal N_0^-}\cdot \mathbf n=0,
\qquad
\mathbf u|_{\mathcal N_0^+}\cdot \mathbf n
=(u_0\mathbf e_{x_1}+\beta_0\mathbf e_\theta)\cdot \mathbf e_r=0.$$
Moreover, $$p_0^+\left(\frac12\right)=\rho_0^+\beta_0^2\ln 1+p_0^-=p_0^-,$$
hence $[p]_{\mathcal{N}_0\cap\{r=\frac12\}}=0.$
Finally, $\mathbf u|_{\mathcal N_0^+}-\mathbf u|_{\mathcal N_0^-}
=u_0\mathbf e_{x_1}+\beta_0\mathbf e_\theta\neq 0$ on $\mathcal{N}_0\cap\{r=\frac12\}$.
Therefore, $\mathbf{U}_0$ satisfies all the conditions in Definition~\ref{definition-wsol}, so $\mathbf{U}_0$ is a weak solution of \eqref{Euler-system-B} in $\mathcal{N}_0$ with a contact discontinuity $\mathcal{N}_0\cap\{r=\frac12\}$.
\end{proof}

\begin{remark}
For an axisymmetric flow of the above form with $\rho>0$,
$$ p\equiv\mathrm{constant}\quad\Longleftrightarrow\quad u_\theta\equiv0.$$
Indeed, this follows immediately from
$$\partial_rp=\frac{\rho u_\theta^2}{r}, \qquad r>0.$$
Hence a nonzero swirl necessarily generates a nontrivial radial pressure gradient.
\end{remark}

In this case, the entropy $S_0$ and Bernoulli function $B_0$ are piecewise smooth functions with
\begin{equation}\label{def-S0-B0}
\begin{aligned}
&S_0(\mathbf{x})=\left\{\begin{aligned}
	\frac{p_0^+(r)}{(\rho_0^+)^{\gamma}}=:S_0^+(r)\quad&\mbox{for}\quad\frac{1}{2}<r<1,\\
	\frac{p_0^-}{(\rho_0^-)^{\gamma}}=:S_0^-\quad&\mbox{for}\quad0\le r <\frac{1}{2},\end{aligned}\right.\\
& B_0(\mathbf{x})=\left\{\begin{aligned}
	\frac{1}{2}(u_0^2+\beta_0^2)+\frac{\gamma p_0^+(r)}{(\gamma-1)\rho_0^+}=:B_0^+(r)\quad&\mbox{for}\quad\frac{1}{2}<r<1,\\
	\frac{\gamma p_0^-}{(\gamma-1)\rho_0^-}=:B_0^-\quad&\mbox{for}\quad0\le r<\frac{1}{2}.\end{aligned}\right.
\end{aligned}
\end{equation}

\begin{figure}[ht]
\begin{center}
\begin{tikzpicture}[scale=1]
\draw[domain=pi:2*pi,smooth,dashed,variable=\x,red]
plot ({1/2*\x+1/6*sin(\x r)-6},{2*cos(\x r)});
\draw[domain=3*pi:4*pi,smooth,dashed,variable=\x,red]
plot ({1/2*\x+1/6*sin(\x r)-6},{2*cos(\x r)});
\draw[domain=5*pi:6*pi,smooth,dashed,variable=\x,red]
plot ({1/2*\x+1/6*sin(\x r)-6},{2*cos(\x r)});
\draw [thick,->] (-8,0) -- (3*pi-4,0);
\draw [thick,->] (-6,-3) -- (-6,3);
\draw [thick,blue] (-6,0) ellipse (1/3 and 1);
\draw [thick,blue] (-6,1) -- (3*pi-6,1);
\draw [thick,blue] (-6,-1) -- (3*pi-6,-1);
\draw [thick] (-6,0) ellipse (2/3 and 2);
\draw [thick] (-6,2) -- (3*pi-6,2);
\draw [thick] (-6,-2) -- (3*pi-6,-2);
\draw [dashed,blue] (3*pi-6,0) ellipse (1/3 and 1);
\draw [dashed] (3*pi-6,0) ellipse (2/3 and 2);
\draw[domain=0:3*pi/5,smooth,thick,variable=\x,red,->]
plot ({1/2*\x+1/6*sin(\x r)-6},{2*cos(\x r)});
\draw[domain=3*pi/5:pi,smooth,thick,variable=\x,red]
plot ({1/2*\x+1/6*sin(\x r)-6},{2*cos(\x r)});
\draw[domain=2*pi:13*pi/5,smooth,thick,variable=\x,red,->]
plot ({1/2*\x+1/6*sin(\x r)-6},{2*cos(\x r)});
\draw[domain=13*pi/5:3*pi,smooth,thick,variable=\x,red]
plot ({1/2*\x+1/6*sin(\x r)-6},{2*cos(\x r)});
\draw[domain=4*pi:23*pi/5,smooth,thick,variable=\x,red,->]
plot ({1/2*\x+1/6*sin(\x r)-6},{2*cos(\x r)});
\draw[domain=23*pi/5:5*pi,smooth,thick,variable=\x,red]
plot ({1/2*\x+1/6*sin(\x r)-6},{2*cos(\x r)});
\node at (-6.1,0.2) {0};
\node at (-6.2,1.2) {$\frac{1}{2}$};
\node at (-6.2,2.2) {1};
\node at (1.5*pi-6,0.5) {$\rho_0^-,\mathbf{0},p_0^-$};
\node at (1.5*pi-6,1.5) {$\rho_0^+,u_0\mathbf{e}_{x_1}+\beta_0\mathbf{e}_\theta,p_0^+$};
\node[below right] at (3*pi-4,0) {$x_1$};
\node[above,blue] at ({3*pi-6},1.0) {contact discontinuity $\left(r=\frac12\right)$};

\end{tikzpicture}
\end{center}
\caption{Background state}
\label{Fig-back}
\end{figure}
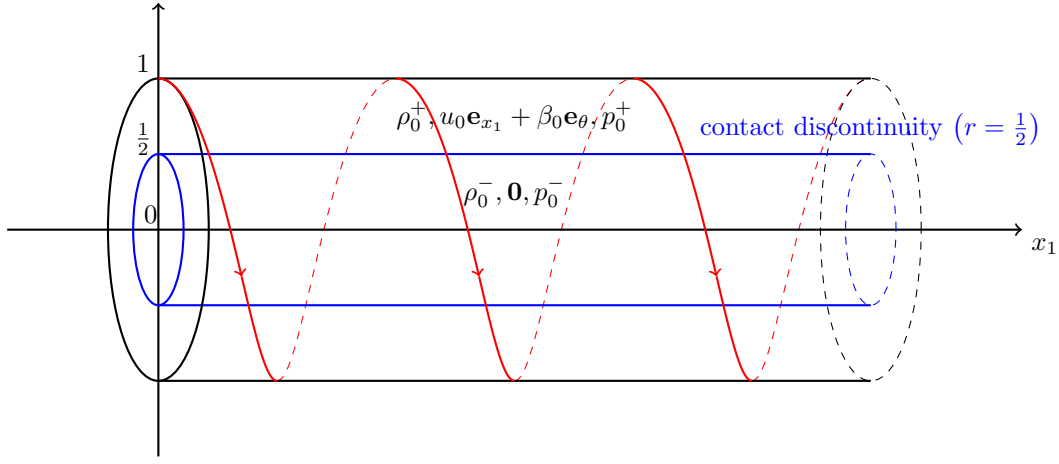

In the present nonzero swirl setting, the relevant condition is not the standard subsonic condition $|\mathbf{u}|<c$, but rather
\begin{equation}\label{def-pseudo-subsonic}
    \inf_{\mathcal{N}_0}\left(1-\frac{u_0^2}{c_0^2}-\frac{u_0^2\beta_0^2}{4c_0^4}\right)>0\quad\text{for the sound speed}\quad c=\sqrt{\frac{\gamma p}{\rho}}.
\end{equation}
We shall refer to this condition as the \emph{pseudo-subsonic} condition. Throughout the paper, we assume that the background state satisfies \eqref{def-pseudo-subsonic}. The reason for introducing this notion is that it is exactly the condition that guarantees the uniform ellipticity of the principal part of the linearized operator around the background state; see Lemma~\ref{lem-pseudosubsonic}. We emphasize that pseudo-subsonicity is weaker than the usual subsonic condition. Indeed, from \eqref{def-pseudo-subsonic}, one sees that the admissible range of $\beta_0$ becomes larger as $u_0$ becomes smaller. Therefore, a flow can be pseudo-subsonic while remaining supersonic in the physical sense, provided that the swirl component $\beta_0$ is sufficiently large.

Our goal is to establish the stability of the background state under perturbation of the nozzle wall and the entrance data. More precisely, we consider the following problem.

\begin{problem}\label{3D-Prob1}
Fix $\epsilon\in(0,1/10)$, $\alpha\in(0,1)$ and $L_0>0$.
Let $R\in C^{2, \alpha}([0, \infty))$ satisfy 
\begin{equation}
\label{condition-R}
R(0)=1, \quad R'(0)=0, \quad R(x_1)=R_\infty \quad\mbox{for all}\quad x_1>L_0.
\end{equation}
Define \begin{equation}
\label{definition-domain}
\begin{aligned}
    \mathcal{N}:=\left\{\mathbf{x}\in \mathbb{R}^3: x_1>0, r<R(x_1)\right\}, \\ \Gamma_\mathrm{w}:=\partial\mathcal{N}\cap\{r=R(x_1)\},\quad
\Gamma_\mathrm{en}:=\partial\mathcal{N}\cap\{x_1=0\}, \quad \Gamma_\mathrm{en}^+:=\Gamma_\mathrm{en}\cap\left\{r\ge \frac{1}{2}\right\}.
\end{aligned}
\end{equation}
For given radial functions
$S_\mathrm{en}(r)$, $\beta_\mathrm{en}(r)$, and $u_r^\mathrm{en}(r)$, define
\begin{equation}
\label{definition-sigma}
\sigma(R, S_\mathrm{en}, \beta_\mathrm{en}, u_r^\mathrm{en}):=\|R-1\|_{2, \alpha, [0, \infty)}+\|S_\mathrm{en}-S_0\|_{2,\alpha,\Gamma_\mathrm{en}^+}+\|\beta_\mathrm{en}-\beta_0\|_{2,\alpha,\Gamma_\mathrm{en}^+}+\|u_r^\mathrm{en}\|_{1,\alpha,\Gamma_\mathrm{en}^+},
\end{equation}

Assume that
\begin{equation}\label{def-entrance-ep}
\begin{aligned}
(S_\mathrm{en},\beta_\mathrm{en})\equiv(S_0^-,0)\quad&\mbox{on}\quad \Gamma_\mathrm{en}\setminus\Gamma_\mathrm{en}^+,\\
u_r^\mathrm{en}\equiv0\quad&\mbox{on}\quad \Gamma_\mathrm{en}
\setminus\left\{ \frac{1}{2}+\epsilon<r<1-\epsilon\right\},
\end{aligned}
\end{equation}
and
\begin{equation*}
\sigma(R, S_\mathrm{en}, \beta_\mathrm{en}, u_r^\mathrm{en})\le \sigma_0
\end{equation*}
for sufficiently small $\sigma_0>0$ to be specified later.

Find a weak solution $\mathbf{U}=(\rho, \mathbf{u}, p)$ to the Euler system \eqref{Euler-system-B} with a contact discontinuity
$$\Gamma_{g_D}:=\mathcal{N}\cap\{r=g_D(x_1)\}$$
in the sense of Definition~\ref{definition-wsol} in $\mathcal{N}$ such that the following properties hold:
\begin{itemize}
\item[(a)] $\rho, \mathbf{u}$, and $p$ are axisymmetric.
\item[(b)] $g_D(0)=\frac{1}{2}$.
\item[(c)] Pseudo-subsonicity: 
\begin{equation}\label{pseudo-subsonicity}\inf_{\mathcal{N}}\left(1-\frac{u_{x_1}^2+u_r^2}{c^2}-\frac{(u_{x_1}^2+u_r^2)u_\theta^2}{4c^4}\right)>0 \quad \mbox{for the sound speed}\quad c=\sqrt{\frac{\gamma p}{\rho}}\end{equation}
\item[(d)] $\rho>0$ {in} $\overline{\mathcal{N}}.$
\item[(e)] $\mathbf{U}$ satisfies the boundary conditions:
\begin{equation*}\label{3D-BC-ent1}
\left\{\begin{aligned}
\frac{1}{2}|\mathbf{u}|^2+\frac{\gamma p}{(\gamma-1)\rho}=B_0, \quad \frac{p}{\rho^{\gamma}}=S_\mathrm{en},\quad \mathbf{u}\cdot\mathbf{e}_{\theta}=\beta_\mathrm{en},\quad\mathbf{u}\cdot\mathbf{e}_r=u_r^\mathrm{en}&\quad\mbox{on}\quad \Gamma_\mathrm{en},\\\mathbf{u}\cdot\mathbf{n}_R=0&\quad\mbox{on}\quad\Gamma_\mathrm{w},\\ [p]_{\Gamma_{g_D}}=0,\quad\mathbf{u}\cdot \mathbf{n}_{g_D}=0&\quad\mbox{on}\quad \Gamma_{g_D},
\end{aligned}\right.
\end{equation*} where $\mathbf{n}_R$, $\mathbf{n}_{g_D}$ denote unit normal vector fields on $\Gamma_\mathrm{w}$, $\Gamma_{g_D}$, respectively.
\end{itemize}
\end{problem}

The conditions $R(0)=1$ and $R'(0)=0$ mean that the perturbed nozzle agrees with the reference cylinder at the entrance up to first order. In particular, the wall meets the entrance orthogonally, which is consistent with the axisymmetric nozzle geometry and avoids an artificial corner-type incompatibility at $\Gamma_\mathrm{w}\cap\Gamma_{\mathrm{en}}^+$.

The constant state $(\rho_0^-,\mathbf{0},p_0^-)$ clearly satisfies the following properties:
\begin{itemize}
\item[(i)] $(\rho_0^-,\mathbf{0},p_0^-)$ satisfies \eqref{pseudo-subsonicity} in $\mathcal{N}\cap\{r<g_D(x_1)\}$.
\item[(ii)] $\rho_0^->0$;
\item[(iii)] By \eqref{def-S0-B0},
$$\frac{p_0^-}{(\rho_0^-)^{\gamma}}=S_0^-,\quad \frac{\gamma p_0^-}{(\gamma-1)\rho_0^-}=B_0^-;$$
\item[(iv)] $\mathbf{u}\cdot\mathbf{v}=0$ for any vector $\mathbf{v}\in\mathbb{R}^3$.
\end{itemize}
Motivated by this observation, we fix $(\rho,\mathbf u,p)=(\rho_0^-,\mathbf 0,p_0^-)$ in the inner layer and reduce Problem 2.1 to the following free boundary problem in the outer layer.

\begin{problem}\label{3D-Problem2}
Under the same assumptions as in Problem~\ref{3D-Prob1},
find a function $g_D:\mathbb{R}^+\to (0,1)$ and a $C^1$ solution $\mathbf{U}=(\rho, \mathbf{u}, p)$ to \eqref{Euler-system-B} in $\mathcal{N}_{g_D}^+:=\mathcal{N}\cap\{r>g_D(x_1)\}$ such that
\begin{itemize}
\item[(a)] $\rho, \mathbf{u}$, and $p$ are axisymmetric.
\item[(b)] \begin{equation*}
g_D(0)=\frac{1}{2}.
\end{equation*}
\item[(c)] $(\rho,\mathbf{u},p)$ satisfies \eqref{pseudo-subsonicity} in $\mathcal{N}\cap\{r>g_D(x_1)\}$.
\item[(d)] $\rho>0$ {in} $\overline{\mathcal{N}^+_{g_D}}.$
\item[(e)] $\mathbf{U}$ satisfies the boundary conditions:
\begin{equation}\label{Prob2-BC-ent}
\left\{\begin{aligned}
\frac{1}{2}|\mathbf{u}|^2+\frac{\gamma p}{(\gamma-1)\rho}=B_0, \quad \frac{p}{\rho^{\gamma}}=S_\mathrm{en},\quad \mathbf{u}\cdot\mathbf{e}_{\theta}=\beta_\mathrm{en},\quad\mathbf{u}\cdot\mathbf{e}_r=u_r^\mathrm{en}&\quad\mbox{on}\quad \Gamma_\mathrm{en}^+,\\\mathbf{u}\cdot \mathbf{n}_R=0 &\quad \mbox{on} \quad \Gamma_\mathrm{w},\\ p=p_0^-,\quad\mathbf{u}\cdot \mathbf{n}_{g_D}=0&\quad\mbox{on}\quad \Gamma_{g_D},
\end{aligned}\right.
\end{equation}where $\mathbf{n}_R$, $\mathbf{n}_{g_D}$ denote unit normal vector fields on $\Gamma_\mathrm{w}$ and $\Gamma_{g_D}$, respectively.
\end{itemize}
\end{problem}

\begin{convention}
Before stating the main theorem, we introduce the following convention. A constant $C$ is said to depend on the data if it depends on $(L_0,R_\infty,\rho_0^+,u_0, p_0^-, \gamma, \alpha)$.
Now we state the main results in this paper.
\end{convention}

\begin{theorem}\label{3D-MainThm}
For a given function $R(x_1)$ on $[0, \infty)$ and given radial functions
$S_\mathrm{en}(r)$, $\beta_\mathrm{en}(r)$, $u_r^\mathrm{en}(r)$ on $\Gamma_\mathrm{en}$, assume that they satisfy \eqref{condition-R}, \eqref{def-entrance-ep}, and let $\sigma(R, S_\mathrm{en}, \beta_\mathrm{en}, u_r^\mathrm{en})$ be given by \eqref{definition-sigma}. For simplicity of notation, let $\sigma$ denote $\sigma(R, S_\mathrm{en}, \beta_\mathrm{en}, u_r^\mathrm{en})$.
\begin{itemize}
\item[(a)](Existence)
For any fixed $\alpha\in(0,1)$, there exist a positive constant $\bar\beta>0$ and a small constant $\sigma_1>0$ depending only on the data such that if
\begin{itemize}
\item $|\beta_0|\le\bar\beta$;
\item $\sigma\le\sigma_1$,
\end{itemize}
 then there exists an axisymmetric solution $\mathbf{U}=(\rho, \mathbf{u},p)$ of Problem~\ref{3D-Problem2} with a contact discontinuity $r=g_D(x_1)$ satisfying \begin{equation}\label{Thm2.1-uniq-est}
\left\|g_D-\frac{1}{2}\right\|_{2,\alpha,\mathbb{R}^+}+\|(\rho, \mathbf{u},p)-(\rho_0^+,\mathbf{u}_0,p_0^+)\|_{1,\alpha,\mathcal{N}^+_{g_D}}\le C\sigma\quad\mbox{for}\,\,\mathbf{u}_0:=u_0\mathbf{e}_{x_1}+\beta_0\mathbf{e}_\theta,
\end{equation}
where the constant $C>0$ depends only on the data.
\item[(b)](Asymptotic state)
There exist constants $\bar\beta_2\in(0,\bar\beta],\,\sigma_2\in(0,\sigma_1]$ depending only on the data such that if
\begin{itemize}
\item $|\beta_0|\le\bar\beta_2$;
\item $\sigma\le\sigma_2$,
\end{itemize}
then the solution $\mathbf{U}=(\rho, \mathbf{u},p)$ in (a) satisfies
\begin{equation*}
\begin{aligned}
&\lim_{L\rightarrow\infty}\left\|g_D'\right\|_{C^1(\{x_1\ge L\})}=0,\\
&\lim_{L\rightarrow\infty}\|\mathbf{u}\cdot\mathbf{e}_r\|_{C^1(\overline{\mathcal{N}^+_{g_D}\cap\{x_1>L\}})}=0,\\
&\lim_{L\rightarrow\infty}\left\|\partial_rp-\frac{\rho (\mathbf{u}\cdot\mathbf{e}_\theta)^2}{r}\right\|_{C^0(\overline{\mathcal{N}_{g_D}^+\cap\{x_1>L\}})}=0.
\end{aligned}
\end{equation*}
\end{itemize}
\end{theorem}

\section{Reformulation of Problem~\ref{3D-Problem2} via Helmholtz decomposition}\label{3D-sec-Hel}

Since $p=S\rho^{\gamma}$, we may regard Problem~\ref{3D-Problem2} as a problem for $(\rho, \mathbf{u}, S)$.
Assume that a smooth solution $(\rho, \mathbf{u}, S)$ of \eqref{Euler-system-B} in $\mathcal{N}_{g_D}^+$ is axisymmetric, so that
$$\rho=\rho(x_1,r),\quad\mathbf{u}=u_{x_1}(x_1,r)\mathbf{e}_{x_1}+u_r(x_1,r)\mathbf{e}_r+u_{\theta}(x_1,r)\mathbf{e}_{\theta},\quad S=S(x_1,r).$$
Define $\Lambda$ by
\begin{equation*}
\Lambda(x_1,r):=ru_{\theta}(x_1,r),
\end{equation*} which represents the angular momentum density.
Then the steady Euler system \eqref{Euler-system-B} can be rewritten in axisymmetric form as:
\begin{equation}\label{3D-ang}
\left\{\begin{aligned}
&\partial_{x_1}(\rho u_{x_1})+\partial_r(\rho u_r)+\frac{\rho u_r}{r}=0,\\
&\rho(u_{x_1}\partial_{x_1}+u_r\partial_r)u_r-\frac{\rho \Lambda^2}{r^3}+\partial_r (S\rho^\gamma)=0,\\
&\rho(u_{x_1}\partial_{x_1}+u_r\partial_r)S=0,\\
&\rho(u_{x_1}\partial_{x_1}+u_r\partial_r)B=0,\\
&\rho(u_{x_1}\partial_{x_1}+u_r\partial_r)\Lambda=0.
\end{aligned}
\right.
\end{equation}

Let $g_D:\mathbb R^+\to(0,1)$ be determined together with $(\rho,\mathbf u,p)$, and represent the velocity field in $\mathcal N_{g_D}^+$ as
\begin{equation*}\label{u-HD}
\mathbf{u}(\mathbf{x})=\nabla\varphi(\mathbf{x})+\operatorname{curl}\mathbf{V}(\mathbf{x})\quad\operatorname{in}\quad \mathcal{N}_{g_D}^+
\end{equation*}
for axisymmetric functions
$$\varphi(\mathbf{x})=\varphi(x_1,r),\quad\mathbf{V}(\mathbf{x})=h(x_1,r)\mathbf{e}_r+\psi(x_1,r)\mathbf{e}_{\theta}.$$
If $(\varphi, \mathbf{V})$ are $C^2$ in $\mathcal{N}_{g_D}^+$,
then a direct computation yields
\begin{equation}\label{3D-u}\begin{aligned}
\mathbf{u}&=\left(\partial_{x_1}\varphi+\frac{1}{r}\partial_r(r\psi)\right)\mathbf{e}_{x_1}+(\partial_r\varphi-\partial_{x_1}\psi)\mathbf{e}_r+\left(\frac{\Lambda}{r}\right)\mathbf{e}_{\theta}\\&=:\mathbf{q}(r, \psi, D\psi, D\varphi, \Lambda)\quad\mbox{for}\quad D=(\partial_{x_1}, \partial_r).
\end{aligned}\end{equation}
For later convenience, we denote by
\begin{equation}\label{def-T}
\mathbf{t}(r,\psi,D\psi,\Lambda):=\mathbf{q}(r,\psi,D\psi,D\varphi,\Lambda)-\nabla\varphi\left(=\operatorname{curl}\mathbf{V}\right)
\end{equation} the non-potential part of the velocity field.

Then, as in \cite{bae20183}, we can rewrite the system \eqref{3D-ang} as a system for $(\varphi,\psi,S,B,\Lambda)$:\begin{equation}\label{3D-H}
\left\{\begin{aligned}
&\operatorname{div}\left(\varrho(S, B,\mathbf{q})\mathbf{q}\right)=0,\\
&-\Delta(\psi\mathbf{e}_{\theta})=G(S, B, \Lambda,\partial_r S,\partial_rB,\partial_r\Lambda,\mathbf{t},\nabla\varphi)\mathbf{e}_{\theta},\\
&\varrho(S, B,\mathbf{q})\mathbf{q}\cdot\nabla S=0,\\
&\varrho(S, B,\mathbf{q})\mathbf{q}\cdot\nabla B=0,\\
&\varrho(S, B,\mathbf{q})\mathbf{q}\cdot\nabla \Lambda=0,
\end{aligned}
\right.
\end{equation}
with
\begin{equation*}
\mathbf{q}=\mathbf{q}(r,\psi,D\psi,D\varphi,\Lambda),\quad\text{and}\quad \mathbf{t}=\mathbf{t}(r,\psi,D\psi,\Lambda),
\end{equation*}
for $(\varrho, G)$ defined by
\begin{equation}\label{def-H-G}
\left.\begin{aligned}
&\varrho(\eta,\zeta,\mathbf{q}):=\left[\frac{\gamma-1}{\gamma\eta}\left(\zeta-\frac{1}{2}|\mathbf{q}|^2\right)\right]^{1/(\gamma-1)},\\
&G(\eta_1,\eta_2,\eta_3,\eta_4,\eta_5,\eta_6,\mathbf{t},\mathbf{v}):=\frac{1}{(\mathbf{t}+\mathbf{v})\cdot\mathbf{e}_{x_1}}\left(-\eta_5+\frac{\varrho^{\gamma-1}(\eta_1, \eta_2, \mathbf{t}+\mathbf{v})}{\gamma-1}\eta_4+\frac{\eta_3}{r^2}\eta_6\right),\\
\end{aligned}
\right.
\end{equation}
for $\eta,\zeta\in\mathbb{R}$, $\mathbf{q}\in\mathbb{R}^3$, $\eta_k\in\mathbb{R}$ for $k=1, ..., 6$, and $\mathbf{t},\mathbf{v}\in\mathbb{R}^3$.
\smallskip

Next, we derive boundary conditions for $(g_D,S, B, \Lambda,\varphi,\psi)$.

(i) On $\Gamma_\mathrm{en}^+$ we prescribe
\begin{equation}\label{def-varphi-en}
\begin{cases}
(S, B, \Lambda)(0,r)=(S_\mathrm{en},B_0^+,r\beta_\mathrm{en})\\
\varphi(0,r)=\displaystyle\int_{\frac{1}{2}}^ru_r^\mathrm{en}(t)dt=:\varphi_\mathrm{en}(r)\\
\partial_{x_1}\psi(0,r)=0
\end{cases}\quad\text{on $\Gamma_\mathrm{en}^+$}
\end{equation}
so that the boundary conditions given in \eqref{Prob2-BC-ent} hold on $\Gamma_\mathrm{en}^+$.

(ii) On $\Gamma_\mathrm{w}$ we prescribe $$\nabla \varphi\cdot\mathbf{n}_R=0, \quad \frac{1}{r}\nabla(r\psi)\cdot \boldsymbol{\tau}_R=0 \quad \text{on } \Gamma_\mathrm{w}$$ where $$\mathbf{n}_R:=\frac{-R'(x_1)\mathbf{e}_{x_1}+\mathbf{e}_r}{\sqrt{1+|R'(x_1)|^2}}, \quad \boldsymbol{\tau}_R:=\frac{\mathbf{e}_{x_1}+R'(x_1)\mathbf{e}_r}{\sqrt{1+|R'(x_1)|^2}},$$ so that the boundary conditions given in \eqref{Prob2-BC-ent} hold on $\Gamma_\mathrm{w}$.

(iii) Free boundary conditions for $\Gamma_{g_D}$: If a contact discontinuity $\Gamma_{g_D}$ is represented as $$\Gamma_{g_D}=\{\mathbf{x}\in \mathcal{N}: r=g_D(x_1)\},$$ then the unit normal $\mathbf{n}_{g_D}$ of $\Gamma_{g_D}$ pointing toward $\{r>g_D(x_1)\}$ is given by
\begin{equation*}
\mathbf{n}_{g_D}=\frac{-g_D'(x_1)\mathbf{e}_{x_1}+\mathbf{e}_r}{\sqrt{1+|g_D'(x_1)|^2}}.
\end{equation*}
Therefore, if $g_D:\mathbb{R}^+\longrightarrow (0,1)$ solves the initial value problem
\begin{equation}\label{g-free-cond}
\left\{\begin{aligned}
&g_D'(x_1)=\frac{\mathbf{q}(r,\psi,D\psi,D\varphi,\Lambda)\cdot\mathbf{e}_r}{\mathbf{q}(r,\psi,D\psi,D\varphi,\Lambda)\cdot\mathbf{e}_{x_1}}(x_1,g_D(x_1),0)\quad\mbox{for}\quad x_1>0,\\
&g_D(0)=\frac{1}{2},
\end{aligned}\right.
\end{equation}
then $\mathbf{u}\cdot\mathbf{n}_{g_D}=0$ on $\Gamma_{g_D}$ for $\mathbf{u}$. We use \eqref{g-free-cond} to determine the location of the contact discontinuity $r=g_D(x_1)$.

An orthonormal basis for the tangent plane $\Gamma_{g_D}$ at any point on $r=g_D(x_1)$ is given by
$$\left\{{\boldsymbol{\tau}}_{g_D}, \mathbf{e}_{\theta}\right\}\quad\mbox{for}\quad {\boldsymbol{\tau}}_{g_D}:=\frac{\mathbf{e}_{x_1}+g_D'(x_1)\mathbf{e}_r}{\sqrt{1+|g_D'(x_1)|^2}}.$$ Then, it follows from the condition $\mathbf{u}\cdot\mathbf{n}_{g_D}=0$ on $\Gamma_{g_D}$ that
\begin{equation}\label{ab-u}
|\mathbf{u}|^2=|\mathbf{u}\cdot\boldsymbol{\tau}_{g_D}|^2+|\mathbf{u}\cdot\mathbf{e}_{\theta}|^2\quad\mbox{on}\quad\Gamma_{g_D}.
\end{equation}
By substituting the expression \eqref{3D-u} into \eqref{ab-u}, we get
\begin{equation}\label{3D-u2}
|\mathbf{u}|^2=\left|\left[\left(\partial_{x_1}\varphi+\frac{1}{r}\partial_r(r\psi)\right)\mathbf{e}_{x_1}+(\partial_r\varphi-\partial_{x_1}\psi)\mathbf{e}_r\right]\cdot\boldsymbol{\tau}_{g_D}\right|^2+\left|\frac{\Lambda}{r}\right|^2\mbox{ on }\Gamma_{g_D}.
\end{equation}
Since the contact discontinuity is a streamline issuing from the point $r=\frac12$ at the entrance, the transport equation for $B$ yields
$$B=B_0^+\!\left(\frac12\right)
\quad\text{on}\quad\Gamma_{g_D}.$$
Thus $\mathbf{u}$ should satisfy
\begin{equation}\label{3D-u22}
\begin{aligned}
|\mathbf{u}|^2
&=2\left(B_0^+(\tfrac{1}{2})-\frac{\gamma p^{1-1/\gamma}S^{1/\gamma}}{\gamma-1}\right)\quad\mbox{on}\quad\Gamma_{g_D}.
\end{aligned}
\end{equation}
Therefore, if $(\varphi, \psi)$ satisfies
\begin{equation*}
\nabla\varphi\cdot\boldsymbol{\tau}_{g_D}=\nabla\varphi_0\cdot\boldsymbol{\tau}_{g_D}\quad\mbox{and}\quad\frac{1}{r}\nabla(r\psi)\cdot\mathbf{n}_{g_D}=\mathcal{A}(g_D,g_D',S, B, \Lambda)
\end{equation*}
for $\varphi_0$ and $\mathcal{A}$ defined by
\begin{equation}\label{def-varphi0-B}
\left.\begin{aligned}
&\varphi_0(\mathbf{x}):=u_0x_1\quad\mbox{for}\quad\mathbf{x}=(x_1,x_2,x_3)\in\overline{\mathcal{N}_{g_D}^+},\\
&\mathcal{A}(g_D,g_D',S, B, \Lambda):=\sqrt{2\left(B_0^+(\tfrac12)-\frac{\gamma (p_0^-)^{1-1/\gamma}S^{1/\gamma}}{\gamma-1}\right)-\left(\frac{\Lambda}{g_D}\right)^2}-\nabla\varphi_0\cdot\boldsymbol{\tau}_{g_D},
\end{aligned}\right.
\end{equation}
then one can directly check from \eqref{ab-u}--\eqref{3D-u22} that the Rankine--Hugoniot condition $p=p_0^-$ on $\Gamma_{g_D}$ given in \eqref{Prob2-BC-ent} holds for $p=S\varrho^\gamma(S,B,\mathbf{u})$.

We collect all the boundary conditions for $(g_D,S, B, \Lambda,\varphi,\psi)$ with \eqref{g-free-cond} as follows:
\begin{equation}\label{3D-BC-C}
\left\{
\begin{aligned}
(S, B, \Lambda)=(S_\mathrm{en},B_0^+,r\beta_\mathrm{en}),\quad\varphi=\varphi_\mathrm{en},\quad\partial_{x_1}\psi=0\quad&\mbox{on}\quad\Gamma_\mathrm{en}^+,\\
\nabla \varphi\cdot\mathbf{n}_R=0, \quad \frac{1}{r}\nabla (r\psi)\cdot\boldsymbol{\tau}_R=0\quad&\mbox{on}\quad\Gamma_\mathrm{w},\\
\nabla\varphi\cdot\boldsymbol{\tau}_{g_D}=\nabla\varphi_0\cdot\boldsymbol{\tau}_{g_D},\quad \frac{1}{r}\nabla(r\psi)\cdot\mathbf{n}_{g_D}=\mathcal{A}(g_D,g_D',S, B, \Lambda)\quad&\mbox{on}\quad\Gamma_{g_D}.
\end{aligned}\right.
\end{equation}
We are thus led to the following reformulated free boundary problem for $(g_D,S,B,\Lambda,\varphi,\psi)$, which will be the main formulation in the rest of the paper.

\begin{theorem}\label{3D-Thm-HD}
Let $\alpha\in(0,1)$ be fixed. Under the same assumptions as in Theorem~\ref{3D-MainThm}, there exist a positive constant $\beta_1>0$ and a small constant $\sigma_3>0$ depending only on the data such that if
\begin{itemize}[label=--]
\item $|\beta_0|\le\beta_1$;
\item $\sigma\le\sigma_3$,
\end{itemize}
then the free boundary problem \eqref{3D-H} with boundary conditions \eqref{g-free-cond} and \eqref{3D-BC-C} has a solution $(g_D, S, B, \Lambda, \varphi, \psi)$ that satisfies
\begin{equation}\label{Thm-HD-est}
\begin{aligned}
\left\|g_D-\frac{1}{2}\right\|_{2,\alpha,\mathbb{R}^+}&\le C\sigma,\\
\|\varphi-\varphi_0\|_{2,\alpha,\mathcal{N}^+_{g_D}}+\|\psi\|_{2,\alpha,\mathcal{N}^+_{g_D}}+\|(S, B, \Lambda)-(S_0^+, B_0^+,\beta_0r)\|_{1,\alpha,\mathcal{N}^+_{g_D}}&\le C\sigma,
\end{aligned}
\end{equation}
where the constant $C>0$ depends only on the data.
\end{theorem}

We first prove Theorem~\ref{3D-Thm-HD} then apply it to prove Theorem~\ref{3D-MainThm}.
We prove Theorem~\ref{3D-Thm-HD} by a limiting argument. To this end, we introduce a free boundary problem in a cut-off domain of finite length $L$, and solve it by the method of iteration in Section~\ref{3D-sec-Cut}.
Uniform estimates of the solutions to the free boundary problems in cut-off domains are established independently of the length $L$. In Section~\ref{5-1}, we prove Theorem~\ref{3D-Thm-HD} by taking a sequence of the solutions to the free boundary problems in cut-off domains then passing to the limit $L\rightarrow\infty$.
The limit yields a solution to the free boundary problem \eqref{3D-H} with boundary conditions \eqref{g-free-cond} and \eqref{3D-BC-C}. Then we can prove that $(g_D, \rho, \mathbf{u}, p)$ yields a solution to Problem~\ref{3D-Problem2}. This proves Theorem~\ref{3D-MainThm}(a).
Finally, Theorem~\ref{3D-MainThm}(b) is proved by using the stream function formulation and energy estimates.

\section{Free boundary problems in cut-off domains}\label{3D-sec-Cut}

\subsection{Iteration framework}
Let $\mathcal{N}$ be given by \eqref{definition-domain}.
For $L>L_0+10$, define $\mathcal{N}_L$ by
\begin{equation*}
\mathcal{N}_L:=\mathcal{N}\cap\{0<x_1<L\}.
\end{equation*}
For functions $R:[0, L]\to (\frac12, \frac32),\,f:[0,L]\rightarrow(0,1)$, set
\begin{equation*}
\begin{aligned}
&\mathcal{N}_{L,f}^+:=\mathcal{N}_L\cap\{r>f(x_1)\},\quad \Gamma_\mathrm{w}^{L}:=\partial \mathcal{N}_{L}\cap \{ r=R(x_1)\},\\
&\Gamma^{L,f}_\mathrm{ex}:=\partial\mathcal{N}_{L,f}^+\cap\{x_1=L\},\quad
\Gamma_\mathrm{cd}^{L,f}:=\partial\mathcal{N}_{L,f}^+\cap\{r=f(x_1)\}.
\end{aligned}
\end{equation*}

\begin{problem}\label{Prob3-Cut} Find a solution $(f,S, B, \Lambda,\varphi,\psi)$ of the following free boundary problem:
\begin{equation*}\label{3D-Cut-Eq}
\eqref{3D-H}\quad\mbox{in}\quad\mathcal{N}_{L,f}^+
\end{equation*}
 with boundary conditions
\begin{equation*}
\left\{\begin{aligned}
(S, B, \Lambda)=(S_\mathrm{en},B_0^+,r\beta_\mathrm{en}),\quad\varphi=\varphi_\mathrm{en},\quad\partial_{x_1}\psi=0\quad&\mbox{on}\quad\Gamma_\mathrm{en}^+,\\
\nabla \varphi\cdot\mathbf{n}_R=0, \quad \psi=0\quad&\mbox{on}\quad\Gamma_\mathrm{w}^L,\\
\varphi=\varphi_0(L,\cdot),\quad\partial_{x_1}\psi=0\quad&\mbox{on}\quad\Gamma^{L,f}_\mathrm{ex},\\
\nabla\varphi\cdot\boldsymbol{\tau}_f=\nabla\varphi_0\cdot\boldsymbol{\tau}_f,\quad\frac{1}{r}\nabla(r\psi)\cdot\mathbf{n}_{f}=\mathcal{A}(f,f',S, B, \Lambda)\quad&\mbox{on}\quad\Gamma_\mathrm{cd}^{L,f},
\end{aligned}\right.
\end{equation*}
and
\begin{equation}\label{g-free-cut}
\left\{\begin{aligned}
&f'(x_1)=\frac{\mathbf{q}(r,\psi,D\psi,D\varphi,\Lambda)\cdot\mathbf{e}_r}{\mathbf{q}(r,\psi,D\psi,D\varphi,\Lambda)\cdot\mathbf{e}_{x_1}}(x_1,f(x_1),0)\quad\mbox{for}\quad x_1>0,\\
&f(0)=\frac{1}{2},
\end{aligned}\right.
\end{equation}
where
$$\boldsymbol{\tau}_f:=\frac{\mathbf{e}_{x_1}+f'(x_1)\mathbf{e}_r}{\sqrt{1+|f'(x_1)|^2}},\quad \mathbf{n}_f:=\frac{-f'(x_1)\mathbf{e}_{x_1}+\mathbf{e}_r}{\sqrt{1+|f'(x_1)|^2}},$$ where $\mathcal{A}$ and $\varphi_0$ are given in \eqref{def-varphi0-B}.
\end{problem}

\begin{proposition}\label{3D-Prop4.1}
Let $\alpha\in(0,1)$ be fixed. Under the same assumptions as in Theorem~\ref{3D-MainThm}, there exist a positive constant $\beta_2>0$ and a small constant $\sigma_4>0$ depending only on the data such that if
\begin{itemize}[label=--]
\item $|\beta_0|\le\beta_2$;
\item $\sigma\le\sigma_4$,
\end{itemize}
then Problem~\ref{Prob3-Cut} has a unique solution $(f,S, B, \Lambda,\varphi,\psi)$ that satisfies
\begin{equation}\label{3D-Prop-est}
\begin{aligned}
\left\|f-\frac{1}{2}\right\|_{2,\alpha,(0,L)}&\le C\sigma,\\
\|\varphi-\varphi_0\|_{2,\alpha,\mathcal{N}_{L,f}^+}+\|\psi\|_{2,\alpha,\mathcal{N}_{L,f}^+}+\|(S, B, \Lambda)-(S_0^+,B_0^+,\beta_0r)\|_{1,\alpha,\mathcal{N}_{L,f}^+}&\le C\sigma,
\end{aligned}
\end{equation}
where the constant $C>0$ depends only on the data but not on $L$.
\end{proposition}

To solve the transport equations for $(S,B,\Lambda)$, we first require $(f,\mathbf q)$ to satisfy \eqref{g-free-cut}. Furthermore, the vector field $\varrho(S, B,\mathbf{q})\mathbf{q}$ must be divergence free; See \cite[Proposition 3.5]{bae20183}.
Therefore, we need to solve a free boundary problem for $(f, \varphi, \psi)$ by fixing approximated entropy, Bernoulli function, and angular momentum density $(\tilde{S},\tilde{B}, \tilde{\Lambda})$, then solve $\varrho(\tilde S,\tilde{B},\mathbf{\tilde q})\mathbf{\tilde q}\cdot\nabla (S, B, \Lambda)=0$ in $\mathcal{N}_{L,f}^+$ to update $(S, B, \Lambda)$, where $\mathbf{\tilde q}$ is given by $\mathbf{\tilde q}=\mathbf{q}(r, \psi, D\psi, D\varphi, \tilde{\Lambda})$.

For fixed constants $\epsilon\in(0,1/10)$ and $\alpha\in(0,1)$, we define iteration sets $\mathcal{U}_r(\zeta_0)$ and $\mathcal{V}_r(\eta_0)$ by
\begin{equation}\label{Ent-Ang-set}
\left.
\begin{aligned}
&\mathcal{U}_r(\zeta_0
):=\left\{\zeta\in C^{1,\alpha}(\overline{\mathcal{N}_{L,1/4}^+})\left|\,
\begin{aligned}
&\zeta\mbox{ is axisymmetric,}\\
&\|\zeta-\zeta_0\|_{1,\alpha,\mathcal{N}_{L,1/4}^+}\le r,\\
&\partial_{x_1}\zeta\equiv0\mbox{ on }\Gamma_\mathrm{en}^\epsilon\cup\Gamma^{L,1/4}_\mathrm{ex}
\end{aligned}\right.\right\},\\
&\mathcal{V}_r(\eta_0
):=\left\{r\eta\in C^{1,\alpha}(\overline{\mathcal{N}_{L,1/4}^+})\left|\,
\begin{aligned}
&\eta\mbox{ is axisymmetric,}\\
&\|\eta-\eta_0\|_{1,\alpha,\mathcal{N}_{L,1/4}^+}\le r,\\
&\partial_{x_1}\eta\equiv0\mbox{ on }\Gamma_\mathrm{en}^\epsilon\cup\Gamma^{L,1/4}_\mathrm{ex}
\end{aligned}\right.\right\},\\
\end{aligned}\right.
\end{equation}
where $\Gamma_\mathrm{en}^\epsilon:=\Gamma_\mathrm{en}^+\setminus\left\{\frac{1}{2}+\epsilon<r<1-\epsilon\right\}$. Then define
$$\mathcal{P}_r:=\mathcal{U}_r(S_0^+)\times\mathcal{U}_r(B_0^+)\times\mathcal{V}_r(\beta_0).$$

\begin{problem}\label{Prob4-Fix-S}
For a constant $\delta_1>0$ to be determined later, let $\mathcal{W}_{\ast}:=(S_{\ast},B_{\ast},\Lambda_{\ast})\in\mathcal{P}_{\delta_1\sigma}$ and set
\begin{equation*}
\mathbf{q}_*:=\mathbf{q}(r,\psi,D\psi,D\varphi,\Lambda_{\ast}),\quad
\mathbf{t}_*:=\mathbf{t}(r,\psi,D\psi,\Lambda_{\ast})
\end{equation*}
for $(\mathbf{q}, \mathbf{t})$ given by \eqref{3D-u} and \eqref{def-T}.
Then, find $(f,\varphi,\psi)$ satisfying \eqref{g-free-cut} and
\begin{equation}\label{S-Free-BP}
\left\{\begin{aligned}
	\left.\begin{aligned}	&\operatorname{div}\left(\varrho(S_{\ast},B_{\ast},\mathbf{q}_*)\mathbf{q}_*\right)=0\\
	&-\Delta(\psi\mathbf{e}_{\theta})=G(S_{\ast},B_{\ast},\Lambda_{\ast},\partial_r S_{\ast},\partial_rB_\ast,\partial_r\Lambda_{\ast},\mathbf{t}_*,D\varphi)\mathbf{e}_{\theta}
	\end{aligned}\right.\quad&\mbox{in}\quad\mathcal{N}_{L,f}^+,\\
\varphi=\varphi_\mathrm{en},\quad\partial_{x_1}\psi=0\quad&\mbox{on}\quad\Gamma_\mathrm{en}^+,\\
\nabla \varphi\cdot\mathbf{n}_R=0, \quad \psi=0\quad&\mbox{on}\quad\Gamma_\mathrm{w}^L,\\
\varphi=\varphi_0(L,\cdot),\quad\partial_{x_1}\psi=0\quad&\mbox{on}\quad\Gamma^{L,f}_\mathrm{ex},\\
\nabla\varphi\cdot\boldsymbol{\tau}_f=\nabla\varphi_0\cdot\boldsymbol{\tau}_f,\quad\frac{1}{r}\nabla(r\psi)\cdot\mathbf{n}_{f}=\mathcal{A}(f,f',S_{\ast},B_{\ast},\Lambda_{\ast})\quad&\mbox{on}\quad\Gamma_\mathrm{cd}^{L,f},
\end{aligned}\right.
\end{equation}
where $\varrho$, $G$, $\varphi_0$, and $\mathcal{A}$ are given by \eqref{def-H-G} and \eqref{def-varphi0-B}.
\end{problem}
\begin{lemma}\label{Lem-S-free}
Under the same assumptions on $(R,S_\mathrm{en},\beta_\mathrm{en},u_r^\mathrm{en})$ as in Proposition~\ref{3D-Prop4.1},
there exist a positive constant $\beta^\star>0$ and a small constant $\sigma_5>0$ depending only on the data and $\delta_1$ so that if
\begin{itemize}[label=--]
\item $|\beta_0|\le\beta^\star$;
\item $\sigma\le\sigma_5$,
\end{itemize}
then, for each $\mathcal{W}_{\ast}\in\mathcal{P}_{\delta_1\sigma}$, Problem~\ref{Prob4-Fix-S} has a unique solution $(f,\varphi,\psi)$ satisfying
\begin{equation}\label{3D-pps-est}
\left\|f-\frac{1}{2}\right\|_{2,\alpha,(0,L)}+\|\varphi-\varphi_0\|_{2,\alpha,\mathcal{N}_{L,f}^+}+\|\psi\|_{2,\alpha,\mathcal{N}_{L,f}^+}\le C\left(\delta_1+1\right)\sigma,
\end{equation}
where the constant $C>0$ depends only on the data but not on $L$.
\end{lemma}

Assuming Lemma~\ref{Lem-S-free}, we prove Proposition~\ref{3D-Prop4.1} as follows.
For each $\mathcal{W}_{\ast}=(S_{\ast},B_{\ast},\Lambda_{\ast})\in\mathcal{P}_{\delta_1\sigma}$, let $(f,\varphi,\psi)$ be the unique solution of Problem~\ref{Prob4-Fix-S}.
Using the divergence-free property of $\varrho(S_\ast,B_\ast,\mathbf{q}_\ast)\mathbf{q}_\ast$, we solve:
\begin{equation}\label{Ite-3D2}
\left\{\begin{aligned}
	\left.\begin{aligned}
	&\varrho(S_{\ast},B_{\ast},\mathbf{q}_\ast)\mathbf{q}_\ast\cdot\nabla S=0\\
    &\varrho(S_{\ast},B_{\ast},\mathbf{q}_\ast)\mathbf{q}_\ast\cdot\nabla B=0\\&\varrho(S_{\ast},B_{\ast},\mathbf{q}_\ast)\mathbf{q}_\ast\cdot\nabla \Lambda=0
	\end{aligned}\right.\quad&\mbox{in}\quad\mathcal{N}_{L,f}^+,\\
(S, B, \Lambda)=(S_\mathrm{en},B_0^+,r\beta_\mathrm{en})\quad&\mbox{on}\quad\Gamma_\mathrm{en}^+.
\end{aligned}\right.
\end{equation}
We take suitable extensions $\mathcal{E}_f(S, B, \Lambda)\in [C^{1,\alpha/2}(\overline{\mathcal{N}_{L,1/4}^+})]^3$ of $(S, B, \Lambda)$.
For such $\mathcal{E}_f(S, B, \Lambda)$, we define an iteration mapping $\mathcal{J}:\mathcal{P}_{\delta_1\sigma}\rightarrow [C^{1,\alpha/2}(\overline{\mathcal{N}_{L,1/4}^+})]^3$ by
\begin{equation*}
\mathcal{J}(S_{\ast},B_{\ast},\Lambda_{\ast})=\mathcal{E}_f(S, B, \Lambda).
\end{equation*}

We then choose $\delta_1$ and $\sigma$ so that the mapping $\mathcal{J}$ maps $\mathcal{P}_{\delta_1\sigma}$ into itself, allowing us to find a fixed point $(S_{\sharp},B_{\sharp},\Lambda_{\sharp})\in\mathcal{P}_{\delta_1\sigma}$ of $\mathcal{J}$. This will prove Proposition~\ref{3D-Prop4.1}. A detailed proof is given in Section~\ref{subsection_4_3}.

\subsection{Proof of Lemma~\ref{Lem-S-free} }
\label{subsection_4_2}

We prove the solvability of Problem~\ref{Prob4-Fix-S} for each fixed $$\mathcal{W}_*=(S_*,B_*,\Lambda_*)\in \mathcal{P}_{\delta_1\sigma}.$$

We define an iteration set
\begin{equation}\label{F-set}
\mathcal{F}_r\left(\frac{1}{2}\right):=\left\{f\in C^{2,\alpha}([0,L])\left|\,
\begin{aligned}
&\left\|f-\frac{1}{2}\right\|_{2,\alpha,(0,L)}\le r, \\
&f(0)=\frac{1}{2},\,f'(0)=f'(L)=0
\end{aligned}\right.\right\}.
\end{equation}
For a constant $\delta_2>0$ to be determined later, we fix $f_{\ast}\in\mathcal{F}_{\delta_2\sigma}(\frac{1}{2})$. If $$\sigma\le\frac{1}{32\delta_2},$$ then $|f_\ast(x_1)-\frac{1}{2}|\le\frac{1}{32}$ for $x_1\in[0,L]$. For such a function $f_\ast$, we solve the following boundary value problem in $\mathcal{N}_{L,f_{\ast}}^+$:
\begin{equation}\label{Fixed-BVP}
\left\{\begin{aligned}
	\left.\begin{aligned}
	&\operatorname{div}\left(\varrho(S_{\ast},B_{\ast},\mathbf{q}_\ast)\mathbf{q}_\ast\right)=0\\
	&-\Delta(\psi\mathbf{e}_{\theta})=G(S_{\ast},B_{\ast},\Lambda_{\ast},\partial_r S_{\ast},\partial_rB_\ast,\partial_r\Lambda_{\ast},\mathbf{t}_\ast,D\varphi)\mathbf{e}_{\theta}
	\end{aligned}\right.\quad&\mbox{in}\quad\mathcal{N}_{L,f_{\ast}}^+,\\
\varphi=\varphi_\mathrm{en},\quad\partial_{x_1}\psi=0\quad&\mbox{on}\quad\Gamma_\mathrm{en}^+,\\
\nabla \varphi\cdot\mathbf{n}_R=0, \quad \psi=0\quad&\mbox{on}\quad\Gamma_\mathrm{w}^L,\\
\varphi=\varphi_0(L,\cdot),\quad\partial_{x_1}\psi=0\quad&\mbox{on}\quad\Gamma^{L,f_{\ast}}_\mathrm{ex},\\
\nabla\varphi\cdot\boldsymbol{\tau}_{f_{\ast}}=\nabla\varphi_0\cdot\boldsymbol{\tau}_{f_\ast},\quad\frac{1}{r}\nabla(r\psi)\cdot\mathbf{n}_{f_\ast}=\mathcal{A}(f_{\ast},f'_{\ast},S_{\ast},B_{\ast},\Lambda_{\ast})\quad&\mbox{on}\quad\Gamma_\mathrm{cd}^{L,f_{\ast}},
\end{aligned}\right.
\end{equation}
where $\varrho$, $G$, $\varphi_0$, and $\mathcal{A}$ are given by \eqref{def-H-G} and \eqref{def-varphi0-B}.

\begin{lemma} \label{Pro-fix-S}
Under the same assumptions on $(R,S_\mathrm{en},\beta_\mathrm{en},u_r^\mathrm{en})$ as in Proposition~\ref{3D-Prop4.1},
there exist a positive constant $\beta^\ast>0$ and a small constant $\sigma_6>0$ depending only on the data and $(\delta_1, \delta_2)$ so that if
\begin{itemize}[label=--]
\item $|\beta_0|\le\beta^\ast$;
\item $\sigma\le\sigma_6$,
\end{itemize}
then the nonlinear boundary value problem \eqref{Fixed-BVP} has a unique solution $(\varphi,\psi)$ satisfying
\begin{equation}\label{fix-est}
\|\varphi-\varphi_0\|_{2,\alpha,\mathcal{N}_{L,f_{\ast}}^+}+\|\psi\|_{2,\alpha,\mathcal{N}_{L,f_{\ast}}^+}\le C\left(\delta_1+1\right)\sigma,
\end{equation}
where the constant $C>0$ depends only on the data but not on $L$.
\end{lemma}

Hereafter, unless otherwise specified, $C$ denotes a positive constant
depending only on the data and independent of $L$; its value may vary
from line to line.

\begin{proof}
\textbf{1.} Let us define two iteration sets
\begin{equation*}
\left.
\begin{aligned}
&\mathcal{K}^{f_\ast,1}_r:=\left\{\phi\in C^{2,\alpha}(\overline{\mathcal{N}_{L,f_\ast}^+})\left|\,
\begin{aligned}
&\phi\mbox{ is axisymmetric,}\\
&\|\phi\|_{2,\alpha,\mathcal{N}_{L,f_\ast}^+}\le r,\\
&\phi=\varphi_\mathrm{en}\mbox{ on }\Gamma_\mathrm{en}^+,\\
&\phi\equiv0\mbox{ on }\Gamma_\mathrm{ex}^{L,f_\ast},\\
&\partial_{x_1}^2\phi\equiv0\mbox{ on }\Gamma_\mathrm{en}^\epsilon\cup\Gamma^{L,f_{\ast}}_\mathrm{ex}
\end{aligned}\right.\right\},\\
&\mathcal{K}^{f_\ast,2}_r:=\left\{\psi\in C^{2,\alpha}(\overline{\mathcal{N}_{L,f_\ast}^+})\left|\,
\begin{aligned}
&\psi\mbox{ is axisymmetric,}\\
& \|\psi\|_{2,\alpha,\mathcal{N}_{L,f_{\ast}}^+}\le r,\\
&\partial_{x_1}\psi\equiv0\mbox{ on }\Gamma_\mathrm{en}^+\cup\Gamma^{L,f_{\ast}}_\mathrm{ex}
\end{aligned}\right.\right\}.
\end{aligned}\right.
\end{equation*}
Then define 
\begin{equation}\label{ell-set}
\mathcal{K}^{f_\ast}_{r_1,r_2}:=\mathcal{K}^{f_\ast,1}_{r_1}\times\mathcal{K}^{f_\ast,2}_{r_2}.
\end{equation}

\textbf{2.} For two constants $\delta_3,\delta_4>0$ to be determined later, let $(\tilde{\phi},\tilde{\psi})\in\mathcal{K}^{f_\ast}_{\delta_3\sigma,\delta_4\sigma}$ and set
 \begin{equation}\label{def-GB}
 \tilde{G}:=G(\mathcal{W}_{\ast},\partial_r\mathcal{W}_{\ast},\mathbf{t}(r,\tilde{\psi},D\tilde{\psi},\Lambda_{\ast}),D\tilde{\phi}+D\varphi_0),\quad \tilde{\mathcal{A}}:=\mathcal{A}(f_{\ast},f_{\ast}',\mathcal{W}_{\ast}),
 \end{equation}
 where $G$ and $\mathcal{A}$ are given by \eqref{def-H-G} and \eqref{def-varphi0-B}, respectively.

The second equation in \eqref{Fixed-BVP} is equivalent to $$-\partial_{x_1 x_1}\psi-\frac{1}{r}\partial_r(r\partial_r\psi)+\frac{\psi}{r^2}=\tilde{G}.$$
By the standard elliptic theory, the following boundary value problem
\begin{equation}\label{3D-psi-BC}\left\{
\begin{aligned}
-\left(\partial_{x_1 x_1}+\frac{1}{r}\partial_r(r\partial_r)-\frac{1}{r^2}\right)\psi=\tilde{G}\quad&\mbox{in}\quad\mathcal{N}_{L,f_{\ast}}^+\\
-\partial_{x_1}\psi=0\quad&\mbox{on}\quad\Gamma_\mathrm{en}^+,\\
\psi=0\quad&\mbox{on}\quad\Gamma_\mathrm{w}^L,\\
\partial_{x_1}\psi=0\quad&\mbox{on}\quad\Gamma^{L,f_{\ast}}_\mathrm{ex},\\
\frac{1}{r}\nabla(r\psi)\cdot\mathbf{n}_{f_{\ast}}=\tilde{\mathcal{A}}\quad&\mbox{on}\quad\Gamma_\mathrm{cd}^{L,f_\ast}
\end{aligned}\right.
\end{equation}
has a unique solution $\psi\in C^{1,\alpha}(\overline{\mathcal{N}_{L,f_\ast}^+})\cap C^{2,\alpha}(\mathcal{N}_{L,f_\ast}^+)$. To obtain a uniform \(C^0\)-estimate independent of \(L\), it is more convenient to
estimate \(Z:=r\psi\). From \eqref{3D-psi-BC}, one directly computes that \(Z\) satisfies
\begin{equation*}\left\{
\begin{aligned}
\displaystyle\left(\partial_{x_1 x_1}+\partial_{rr}-\frac1r\partial_r\right)Z=-r\tilde G
\quad&\text{in}\quad\mathcal{N}_{L,f_\ast}^+,\\
\partial_{x_1}Z=0\quad&\text{on}\quad\Gamma_\mathrm{en}^+\cup\Gamma_\mathrm{ex}^{L,f_\ast},\\
Z=0\quad&\text{on}\quad\Gamma_\mathrm{w}^L,\\
\nabla Z\cdot \mathbf{n}_{f_\ast}=f_\ast\tilde{\mathcal A}\quad&\text{on}\quad\Gamma_\mathrm{cd}^{L,f_\ast}.
\end{aligned}\right.
\end{equation*}
Define $\mathfrak{N}:\overline{\mathcal{N}_{L,f_\ast}^+}\to\mathbb{R}^+$ by $$\mathfrak{N}(x_1,r):=2\mathfrak{n}(8-r^3)\quad\mbox{for}\quad\mathfrak{n}:=\|\tilde{G}\|_{0,\alpha,\mathcal{N}_{L,f_\ast}^+}+\|\tilde{\mathcal{A}}\|_{1,\alpha,\Gamma_{\mathrm{cd}}^{L,f_\ast}}.$$
Since $$\partial_r\mathfrak{N}=-6\mathfrak{n}r^2,\quad\partial_{rr}\mathfrak{N}=-12\mathfrak{n}r,$$ we have $$\left(\partial_{x_1 x_1}+\partial_{rr}-\frac{1}{r}\partial_r\right)(\mathfrak{N}\pm Z)=-6\mathfrak{n}r\mp r\tilde{G}\le 0.$$ Moreover, on $\Gamma_{\mathrm{cd}}^{L,f_\ast}$, \begin{equation}\label{N-est}\nabla\mathfrak{N}\cdot\mathbf{n}_{f_\ast}=-\frac{6\mathfrak{n}(f_\ast(x_1))^2}{\sqrt{1+|f_\ast'(x_1)|^2}}\le -f_\ast(x_1)|\tilde{\mathcal{A}}|.\end{equation} Then, a direct computation with \eqref{N-est} yields $$\left\{\begin{aligned}
    \displaystyle\left(\partial_{x_1 x_1}+\partial_{rr}-\frac{1}{r}\partial_r\right)(\mathfrak{N}\pm Z)\le 0\quad&\mbox{in}\quad \mathcal{N}_{L,f_\ast}^+,\\
    \partial_{x_1}(\mathfrak{N}\pm Z)=0\quad&\mbox{on}\quad \Gamma_{\mathrm{en}}^+\cup\Gamma_{\mathrm{ex}}^{L,f_\ast},\\
    \mathfrak{N}\pm Z\ge 0\quad&\mbox{on}\quad \Gamma_{\mathrm{w}}^L,\\
    \nabla(\mathfrak{N}\pm Z)\cdot\mathbf{n}_{f_\ast}\le 0\quad&\mbox{on}\quad\Gamma_{\mathrm{cd}}^{L,f_\ast}.
\end{aligned}\right.$$
By the comparison principle and Hopf's lemma, we have $$-\mathfrak{N}\le Z\le \mathfrak{N}\quad\mbox{in}\quad\mathcal{N}_{L,f_\ast}^+.$$
Therefore, we get the estimate \begin{equation*}
    \|Z\|_{0,\mathcal{N}_{L,f_\ast}^+}\le C\left(\|r\tilde{G}\|_{0,\alpha,\mathcal{N}_{L,f_\ast}^+}+\|\tilde{\mathcal{A}}\|_{1,\alpha,\Gamma_{\mathrm{cd}}^{L,f_\ast}}\right).
\end{equation*}
Since $Z=r\psi$ and $r\ge 1/4$, it follows that \begin{equation}\label{psi-estimate-C0}
    \|\psi\|_{0,\mathcal{N}_{L,f_\ast}^+}\le C\left(\|\tilde{G}\|_{0,\alpha,\mathcal{N}_{L,f_\ast}^+}+\|\tilde{\mathcal{A}}\|_{1,\alpha,\Gamma_{\mathrm{cd}}^{L,f_\ast}}\right).
\end{equation}

To obtain a $C^{2,\alpha}$-estimate of $\psi$ up to the boundary,
we use the method of reflection.
Define extensions of $R\in C^{2, \alpha}([0, L])$, $f_{\ast}\in\mathcal{F}_{\delta_2\sigma}(\frac{1}{2})$ into $-1\le x_1\le L+1$ by
\begin{equation*}
R^e(x_1):=\left\{\begin{aligned}
R(-x_1)\quad&\mbox{for }-1\le x_1<0,\\
R(x_1)\quad&\mbox{for }0\le x_1\le L,\\
R(2L-x_1)\quad&\mbox{for }L< x_1\le L+1,
\end{aligned}\right.\qquad
f_{\ast}^e(x_1):=\left\{\begin{aligned}
f_{\ast}(-x_1)\quad&\mbox{for }-1\le x_1<0,\\
f_{\ast}(x_1)\quad&\mbox{for }0\le x_1\le L,\\
f_{\ast}(2L-x_1)\quad&\mbox{for }L< x_1\le L+1.
\end{aligned}\right.
\end{equation*}
Since $R'(0)=R'(L)=f_{\ast}'(0)=f_{\ast}'(L)=0$, we have the estimate
$$\|R^e\|_{2,\alpha,(-1,L+1)}\le C\|R\|_{2,\alpha,(0,L)}\quad\|f_{\ast}^e\|_{2,\alpha,(-1,L+1)}\le C\|f_{\ast}\|_{2,\alpha,(0,L)}.$$
We define an extended domain
\begin{equation*}\label{3D-N-ext}
\mathcal{N}_\mathrm{ext}:=\left\{\mathbf{x}\in\mathbb{R}^3: -1<x_1<L+1,\, f^e_{\ast}(x_1)<r<R^e(x_1)\right\}
\end{equation*}
and
$$\Gamma_\mathrm{ext}:=\partial\mathcal{N}_\mathrm{ext}\cap\left\{r=f_{\ast}^e(x_1)\right\}.$$
We also define extensions of $(\psi, \tilde{G}, \tilde{\mathcal{A}})$ into $\mathcal{N}_\mathrm{ext}$ as follows:
\begin{equation*}
\begin{aligned}
	\left(\psi_\mathrm{ext},\mathfrak{G}_\mathrm{ext},\mathfrak{A}_\mathrm{ext}\right)(\mathbf{x})
	&:=\left\{\begin{aligned}
	\left(\psi,\tilde{G},\tilde{\mathcal{A}}\right)(-x_1,x_2,x_3)\quad&\mbox{for }-1\le x_1<0,\\
	\left(\psi,\tilde{G},\tilde{\mathcal{A}}\right)(x_1,x_2,x_3)\quad&\mbox{for }0\le x_1\le L,\\
	\left(\psi,\tilde{G},\tilde{\mathcal{A}}\right)(2L-x_1,x_2,x_3)\quad&\mbox{for }L< x_1\le L+1.
	\end{aligned}\right.\\
	\end{aligned}
\end{equation*}
Then $\mathfrak{G}_\mathrm{ext}\in C^{\alpha}(\overline{\mathcal{N}_\mathrm{ext}})$ and
\begin{equation*}
\|\mathfrak{G}_\mathrm{ext}\|_{\alpha,\mathcal{N}_\mathrm{ext}}\le C\|\tilde{G}\|_{\alpha,\mathcal{N}_{L,f_\ast}^+}.
\end{equation*}
By the compatibility conditions of $(\mathcal{W}_{\ast}, f_{\ast})$ given in \eqref{Ent-Ang-set} and \eqref{F-set},
$$\nabla\mathfrak{A}_\mathrm{ext}\cdot\boldsymbol{\tau}_{f_\ast}\equiv0\quad\mbox{on}\quad\Gamma_\mathrm{ext}\cap\{x_1=0,L\}.$$
From this and the definition of $\mathfrak{A}_\mathrm{ext}$, we have the estimate
\begin{equation*}
\|\mathfrak{A}_\mathrm{ext}\|_{1,\alpha,\Gamma_\mathrm{ext}}\le C\|\tilde{\mathcal{A}}\|_{1,\alpha,\Gamma_\mathrm{cd}^{L,f_\ast}}.
\end{equation*}

Consider a connected subdomain $\mathcal{N}_l$ of $\mathcal{N}_\mathrm{ext}$ such that
\begin{equation*}
\mathcal{N}_\mathrm{ext}\cap\left\{-\frac{1}{2}\le x_1\le \frac{1}{2}\right\}\subset\mathcal{N}_l\subset\mathcal{N}_\mathrm{ext}\cap\left\{-1\le x_1\le 1\right\}
\end{equation*}
and the boundary $\partial\mathcal{N}_l$ is smooth.
By the standard elliptic theory, the boundary value problem
\begin{equation}\label{3D-MP}
\left\{\begin{aligned}
-\left(\partial_{x_1 x_1}+\frac{1}{r}\partial_r(r\partial_r)-\frac{1}{r^2}\right)\Psi=\mathfrak{G}_\mathrm{ext}\quad&\mbox{in}\quad\mathcal{N}_{l},\\
\Psi=0\quad&\mbox{on}\quad \partial\mathcal{N}_l\cap\{r=R^e(x_1)\},\\
\frac{1}{r}\nabla(r\Psi)\cdot\mathbf{n}_{f_{\ast}^e}=\mathfrak{A}_\mathrm{ext}\quad&\mbox{on}\quad\partial\mathcal{N}_{l}\cap\{r=f^e_{\ast}(x_1)\},\\
\Psi=\psi_\mathrm{ext}\quad&\mbox{on}\quad\partial\mathcal{N}_{l}\setminus\{r=R^e(x_1),r=f_{\ast}^e(x_1)\},
\end{aligned}\right.
\end{equation}
for
\begin{equation*}
\mathbf{n}_{f_{\ast}^e}:=\frac{-(f_{\ast}^e)'(x_1)\mathbf{e}_{x_1}+\mathbf{e}_r}{\sqrt{1+|(f_{\ast}^e)'(x_1)|^2}},
\end{equation*}
has a unique solution $\Psi\in C^{2,\alpha}(\overline{\mathcal{N}_{l}})$ that satisfies
$$\|\Psi\|_{2,\alpha,\mathcal{N}_\mathrm{ext}\cap\left\{-\frac{1}{2}\le x_1\le \frac{1}{2}\right\}}\le C\left(\|\mathfrak{G}_\mathrm{ext}\|_{\alpha,\mathcal{N}_{l}}+\|\mathfrak{A}_\mathrm{ext}\|_{1,\alpha,\partial\mathcal{N}_{l}\cap\{r=f_\ast^e(x_1)\}}+\|\psi\|_{C^0(\overline{\mathcal{N}_{L,f_\ast}^+})}\right).$$
By the definitions of $(\mathfrak{G}_\mathrm{ext}, \mathfrak{A}_\mathrm{ext},\psi_\mathrm{ext})$ and the uniqueness of a solution to \eqref{3D-MP}, we have
$\Psi(x_1,x_2,x_3)=\Psi(-x_1,x_2,x_3)$ and $\partial_{x_1}\Psi(0,x_2,x_3)=0$.
The uniqueness of a solution to \eqref{3D-psi-BC} yields that $\Psi=\psi$ in $\mathcal{N}_l\cap\{x_1\ge 0\}.$
By combining \eqref{psi-estimate-C0} and the $C^{2,\alpha}$-estimate of $\Psi$ given right above, we obtain that
\begin{equation}\label{left-psi}
\|\psi\|_{2,\alpha,\mathcal{N}_{L,f_{\ast}}^+\cap\left\{0\le x_1\le \frac{1}{2}\right\}}\le C\left(\|\tilde{G}\|_{\alpha,\mathcal{N}_{L,f_\ast}^+}+\|\tilde{\mathcal{A}}\|_{1,\alpha,\Gamma_\mathrm{cd}^{L,f_\ast}}\right).
\end{equation}
Similarly, one obtains
\begin{equation}\label{right-psi}
\|\psi\|_{2,\alpha,\mathcal{N}_{L,f_{\ast}}^+\cap\left\{L-\frac{1}{2}\le x_1\le L\right\}}\le C\left(\|\tilde{G}\|_{\alpha,\mathcal{N}_{L,f_\ast}^+}+\|\tilde{\mathcal{A}}\|_{1,\alpha,\Gamma_\mathrm{cd}^{L,f_\ast}}\right).
\end{equation}
It follows from \eqref{left-psi}--\eqref{right-psi} that
\begin{equation}\label{psi-est-2}
\|\psi\|_{2,\alpha,\mathcal{N}_{L,f_\ast}^+}\le C\left(\|\tilde{G}\|_{\alpha,\mathcal{N}_{L,f_\ast}^+}+\|\tilde{\mathcal{A}}\|_{1,\alpha,\Gamma_\mathrm{cd}^{L,f_\ast}}\right).
\end{equation}

Finally, we verify that the solution obtained above is axisymmetric by a rotation-invariance argument. For any $\vartheta\in[0,2\pi)$, define a function $\psi^{\vartheta}$ by
$$\psi^{\vartheta}(\mathbf{x}):=\psi(x_1,x_2\cos\vartheta-x_3\sin\vartheta,x_2\sin\vartheta+x_3\cos\vartheta).$$
Then, we have $\psi^{\vartheta}=\psi$ on $\Gamma_{\mathrm{w}}^L\cup\Gamma_{\mathrm{cd}}^{L,f_\ast}$ and $\partial_{x_1} \psi^{\vartheta}=\partial_{x_1}\psi$ on $\Gamma_{\mathrm{en}}^+\cup\Gamma_{\mathrm{ex}}^{L,f_\ast}$.
It can be directly checked that $$-\left(\partial_{x_1 x_1}+\frac{1}{r}\partial_r(r\partial_r)-\frac{1}{r^2}\right)\psi^\vartheta=-\left(\partial_{x_1 x_1}+\frac{1}{r}\partial_r(r\partial_r)-\frac{1}{r^2}\right)\psi$$ holds in $\mathcal{N}_{L,f_\ast}^+$. This implies that $\psi^{\vartheta}$ is a solution to \eqref{3D-psi-BC}. By the uniqueness of the solution to \eqref{3D-psi-BC}, we conclude that
$\psi=\psi^{\vartheta}.$
Therefore, $\psi$ is axisymmetric.

\textbf{3.}
For $\xi, \zeta\in\mathbb{R}$, $\mathbf{s}=(s_1,s_2,s_3)\in\mathbb{R}^3$, and $\mathbf{v}=(v_1,v_2,v_3)\in\mathbb{R}^3$, define $\tilde{\varrho}$ and $\mathbf{A}=(A_1,A_2,A_3)$ by
\begin{equation*}
\tilde{\varrho}(\xi,\zeta,\mathbf{s},\mathbf{v}):=\varrho(\xi,\zeta,\mathbf{s}+\mathbf{v}),\quad
A_j(\xi,\zeta,\mathbf{s},\mathbf{v}):=\tilde{\varrho}(\xi,\zeta,\mathbf{s},\mathbf{v})s_j\quad\mbox{for}\quad j=1,2,3,
\end{equation*}
where $\varrho$ is defined by \eqref{def-H-G}.
Then the equation
$$\operatorname{div}\left(\varrho(S, B,\mathbf{q}(r,\psi,D\psi,D\varphi,\Lambda))\mathbf{q}(r,\psi,D\psi,D\varphi,\Lambda)\right)=0$$
can be rewritten as
\begin{equation}\label{re-conti}
\operatorname{div}\left(\mathbf{A}(S,B,D\varphi,\mathbf{t}(r,\psi,D\psi,\Lambda))\right)=-\operatorname{div}\left(\tilde{\varrho}(S,B,D\varphi,\mathbf{t}(r,\psi,D\psi,\Lambda))\mathbf{t}(r,\psi,D\psi,\Lambda)\right).
\end{equation}
For $\varphi_0$ given by \eqref{def-varphi0-B}, denote $\mathbf{V}_0:=(S_0^+,B_0^+,D\varphi_0,\beta_0\mathbf{e}_\theta)$ and set
\begin{equation}\label{aij-def}
a_{ij}:=\partial_{s_j}A_i(\mathbf{V}_0)\quad\mbox{for}\quad i,j=1,2,3.
\end{equation}
Then $$[a_{ij}]_{i,j=1}^3=\begin{bmatrix}
    \rho_0^+\left(1-\dfrac{u_0^2}{c_0^2}\right) & \rho_0^+\dfrac{u_0\beta_0\sin\theta}{c_0^2} & -\rho_0^+\dfrac{u_0\beta_0\cos\theta}{c_0^2} \\ 0 & \rho_0^+ & 0 \\ 0 & 0 & \rho_0^+
\end{bmatrix}.$$

The pseudo-subsonic condition for the background state implies that $$\delta_0(\beta_0):=\inf_{\mathcal{N}_{L,f_\ast}^+}\left(1-\frac{u_0^2}{c_0^2}-\frac{u_0^2\beta_0^2}{4c_0^4}\right)>0.$$

\begin{lemma}\label{lem-pseudosubsonic}
Let $[a_{ij}]_{i,j=1}^3$ be defined by \eqref{aij-def}, and suppose that $\delta_0(\beta^\natural)>0$ for a positive constant $\beta^\natural>0$.
If $|\beta_0|\le\beta^\natural$, then the coefficient matrix $[a_{ij}]_{i,j=1}^3$ satisfies the uniform ellipticity condition in $\mathcal{N}_{L,f_\ast}^+$. That is, there exist constants $\lambda,\Lambda>0$, depending only on the data, such that $$\lambda |\boldsymbol{\xi}|^2\le \sum_{i,j=1}^3 a_{ij}\xi_i\xi_j\le \Lambda |\boldsymbol{\xi}|^2\quad\mbox{for all }\boldsymbol{\xi}\in\mathbb R^3.$$
\end{lemma}

\begin{proof}
For $\boldsymbol{\xi}=(\xi_1,\xi_2,\xi_3)\in\mathbb{R}^3$, define $$Q(\boldsymbol{\xi}):=\sum_{i,j=1}^3a_{ij}\xi_i\xi_j.$$
By the definition of $[a_{ij}]_{i,j=1}^3$, we have
$$Q(\boldsymbol{\xi})=\rho_0^+\left[\left(1-\frac{u_0^2}{c_0^2}\right)\xi_1^2+\frac{u_0\beta_0}{c_0^2}\xi_1(\sin\theta\ \xi_2-\cos\theta\ \xi_3)+\xi_2^2+\xi_3^2\right].$$ Set $$\eta:=\sin\theta\ \xi_2-\cos\theta\ \xi_3, \quad\zeta:=\cos\theta\ \xi_2+\sin\theta\ \xi_3.$$ Then $\eta^2+\zeta^2=\xi_2^2+\xi_3^2$. Therefore, $$Q(\boldsymbol{\xi})=\rho_0^+\left[\left(1-\frac{u_0^2}{c_0^2}\right)\xi_1^2+\frac{u_0\beta_0}{c_0^2}\xi_1\eta+\eta^2+\zeta^2\right].$$ Equivalently, $$Q(\boldsymbol{\xi})=\rho_0^+\left[\begin{pmatrix}\xi_1&\eta\end{pmatrix}\mathbb{M}\begin{pmatrix}\xi_1\\\eta\end{pmatrix}+\zeta^2\right],$$ where $$\mathbb{M}:=\begin{pmatrix}1-\dfrac{u_0^2}{c_0^2}&\dfrac{u_0\beta_0}{2c_0^2}\\\dfrac{u_0\beta_0}{2c_0^2}&1\end{pmatrix}.$$
We have $$\det\mathbb{M}=1-\frac{u_0^2}{c_0^2}-\frac{u_0^2\beta_0^2}{4c_0^4}\ge \delta_0(\beta^\natural)>0.$$ Moreover, $$1-\frac{u_0^2}{c_0^2}=\det\mathbb{M}+\frac{u_0^2\beta_0^2}{4c_0^4}>0.$$
Hence, by Sylvester's criterion, $\mathbb{M}$ is positive definite.

Let $0<\mu_1\le\mu_2$ be the eigenvalues of $\mathbb{M}$. Since $$\mu_1\mu_2=\det\mathbb{M}$$ and $$\mu_2\le \operatorname{tr}\mathbb{M}=2-\frac{u_0^2}{c_0^2}<2,$$ we obtain $$\mu_1=\frac{\det\mathbb{M}}{\mu_2}\ge \frac{\det\mathbb{M}}{\operatorname{tr}\mathbb{M}}\ge\frac{\delta_0(\beta^\natural)}{2}.$$
Consequently, $$\begin{pmatrix}\xi_1&\eta\end{pmatrix}\mathbb{M}\begin{pmatrix}\xi_1\\\eta\end{pmatrix}\ge\frac{\delta_0(\beta^\natural)}{2}(\xi_1^2+\eta^2).$$
Since $0<\delta_0(\beta^\natural)\le 1$, it follows that \begin{align*}
Q(\boldsymbol{\xi})&\ge\rho_0^+\left[\frac{\delta_0(\beta^\natural)}{2}(\xi_1^2+\eta^2)+\zeta^2\right]\\&\ge \frac{\rho_0^+\delta_0(\beta^\natural)}{2}(\xi_1^2+\eta^2+\zeta^2)\\&=\frac{\rho_0^+\delta_0(\beta^\natural)}{2}|\boldsymbol{\xi}|^2.
\end{align*}
On the other hand, since $\mu_2<2$, \begin{align*}
Q(\boldsymbol{\xi})&\le \rho_0^+[2(\xi_1^2+\eta^2)+\zeta^2]\\&\le 2\rho_0^+(\xi_1^2+\eta^2+\zeta^2)\\&=2\rho_0^+|\boldsymbol{\xi}|^2.
\end{align*}
This proves the uniform ellipticity of $[a_{ij}]_{i,j=1}^3$.
\end{proof}

Set $\phi:=\varphi-\varphi_0$.
Then \eqref{re-conti} can be rewritten as
\begin{equation*}
\mathcal{L}(\phi)=\operatorname{div}\mathbf{F}(S-S_0^+,B-B_0^+,D\phi,\mathbf{t}(r,\psi,D\psi,\Lambda)-\beta_0\mathbf{e}_\theta),
\end{equation*}
where $\mathcal{L}$ and $\mathbf{F}=(F_1,F_2,F_3)$ are defined as follows:
\begin{equation}\label{def-F}
\left.\begin{aligned}
\mathcal{L}(\phi):=&\sum_{i=1}^3 \partial_i\left(\sum_{j=1}^3 a_{ij}\partial_{j}\phi\right)=\sum_{i=1}^3a_{ii}\partial_{ii}\phi+a_{12}\partial_1\partial_2\phi+a_{13}\partial_1\partial_3\phi,\\
F_i(Q):=&-\tilde{\varrho}(\mathbf{V}_0+Q)v_i-\int_0^1 D_{(\xi,\zeta,\mathbf{v})}A_i(\mathbf{V}_0+tQ)dt\cdot(\xi,\zeta,\mathbf{v})\\
&-\mathbf{s}\cdot\int_0^1\left(D_\mathbf{s}A_i(\mathbf{V}_0+tQ)-D_\mathbf{s}A_i(\mathbf{V}_0)\right)dt,
\end{aligned}
\right.
\end{equation}
with $Q=(\xi,\zeta, \mathbf{s},\mathbf{v})\in\mathbb{R}\times\mathbb{R}\times(\mathbb{R}^3)^2$.
Here, $\partial_{x_i}$ is abbreviated as $\partial_i$.

By the boundary conditions for $\varphi$ given in \eqref{Fixed-BVP} and the definition of $\varphi_0$, the boundary conditions for $\phi$ on $\partial\mathcal{N}_{L,f_{\ast}}^+\setminus\Gamma_\mathrm{cd}^{L,f_{\ast}}$ become
\begin{equation*}
\phi=\varphi_\mathrm{en }\quad\mbox{on}\quad \Gamma_\mathrm{en}^+,\quad\nabla\phi\cdot\mathbf{n}_R=\frac{u_0R'(x_1)}{\sqrt{1+|R'(x_1)|^2}}\quad\mbox{on}\quad\Gamma_\mathrm{w}^{L}\quad\mbox{and}\quad
\phi=0\quad\mbox{on}\quad \Gamma_\mathrm{ex}^{L,f_{\ast}}.
\end{equation*}
On $\Gamma_\mathrm{cd}^{L,f_{\ast}}$, the boundary condition for $\varphi$ given in \eqref{Fixed-BVP} implies that $\phi$ should be a constant along $\Gamma_\mathrm{cd}^{L,f_{\ast}}$.
Since we seek $\phi$ continuous up to the boundary, and since $\varphi_\mathrm{en}(\frac{1}{2})=0$ by the definition \eqref{def-varphi-en}, we prescribe the boundary condition for $\phi$ on $\Gamma_\mathrm{cd}^{L,f_{\ast}}$ as
\begin{equation*}
\phi=0\quad\mbox{on}\quad \Gamma_\mathrm{cd}^{L,f_\ast}.
\end{equation*}

For a fixed $(\tilde{\phi},\tilde{\psi})\in\mathcal{K}^{f_\ast}_{\delta_3\sigma,\delta_4\sigma}$, let $\psi\in C^{2,\alpha}(\overline{\mathcal{N}_{L,f_\ast}^+})$ be the unique solution to the linear boundary value problem \eqref{3D-psi-BC} associated with $(\tilde{\phi},\tilde{\psi})\in\mathcal{K}^{f_\ast}_{\delta_3\sigma,\delta_4\sigma}$. For such $\psi$, we set
\begin{equation}\label{FF-def}
\mathfrak{F}:=\mathbf{F}(S_\ast-S_0^+,B_\ast-B_0^+,D{\tilde{\phi}},\mathbf{t}(r,{\psi},D{\psi},\Lambda_{\ast})-\beta_0\mathbf{e}_\theta),
\end{equation}
where $\mathbf{F}$ is given by \eqref{def-F}. We consider the following linear boundary value problem
\begin{equation}\label{lin-phi}
\left\{\begin{aligned}
\mathcal{L}(\phi)=\operatorname{div}\mathfrak{F}\quad&\mbox{in}\quad \mathcal{N}_{L,f_\ast}^+,\\
\phi=\varphi_\mathrm{en}\quad&\mbox{on}\quad\Gamma_\mathrm{en}^+,\\
\phi=0\quad&\mbox{on}\quad \Gamma_\mathrm{cd}^{L,f_\ast}\cup\Gamma^{L,f_{\ast}}_\mathrm{ex},\\
\nabla\phi\cdot\mathbf{n}_R=\frac{u_0R'(x_1)}{\sqrt{1+|R'(x_1)|^2}}=:\mathfrak{H}\quad&\mbox{on}\quad\Gamma_\mathrm{w}^{L}.
\end{aligned}\right.
\end{equation}
In the next step, we prove the well-posedness of \eqref{lin-phi}.

\textbf{4.}
\textbf{Claim:} For each $(\tilde{\phi},\tilde{\psi})\in\mathcal{K}^{f_\ast}_{\delta_3\sigma,\delta_4\sigma}$, the linear boundary value problem
\eqref{lin-phi} associated with $(\tilde{\phi},\tilde{\psi})$ has a unique solution $\phi\in C^{2,\alpha}(\overline{\mathcal{N}_{L,f_\ast}^+})$, and the solution satisfies
\begin{equation}\label{3D-phi-est}
\|\phi\|_{k,\alpha,\mathcal{N}_{L,f_\ast}^+}
\le C\left(\|{\mathfrak F}\|_{k-1,\alpha,\mathcal{N}_{L,f_\ast}^+}+\|\mathfrak{H}\|_{k-1,\alpha,\Gamma_\mathrm{w}^L}+\|{\varphi_\mathrm{en}}\|_{k,\alpha,\Gamma_\mathrm{en}^+}\right)\quad\mbox{for }k=1,2.
\end{equation}
Moreover, the solution $\phi$ is axisymmetric, and it satisfies
$$\partial_{x_1 x_1}\phi\equiv 0\quad\mbox{on}\quad
\Gamma_\mathrm{en}^\epsilon\cup\Gamma^{L,f_{\ast}}_\mathrm{ex}.$$

\begin{proof}[Verification of Claim.]
To homogenize the boundary condition on $\Gamma_\mathrm{en}^+$, we introduce a lifting function $\varphi_\mathrm{en}^*$ defined by
\begin{equation}\label{def-var-ast}
{\varphi_\mathrm{en}^{\ast}}(\mathbf{x}):=\eta(x_1)\varphi_\mathrm{en}\left(\frac{r+R(x_1)-2f_\ast(x_1)}{2R(x_1)-2f_{\ast}(x_1)}\right)\quad\mbox{for}\quad\mathbf{x}\in\mathcal{N}_{L,f_\ast}^+,
\end{equation}
where $\varphi_\mathrm{en}$ is given by \eqref{def-varphi-en} and $\eta$ is a $C^{\infty}$-function satisfying
\begin{equation}\label{eta-def}
\eta=1\quad\mbox{for } x_1<\frac{L}{10},\quad \eta=0\quad\mbox{for }x_1>\frac{9L}{10},\quad |\eta'(x_1)|\le2,\quad|\eta''(x_1)|\le 2.
\end{equation}
Set $\phi_\mathrm{hom}:=\phi-\varphi_\mathrm{en}^{\ast}$. Then the linear boundary value problem \eqref{lin-phi} can be rewritten as
\begin{equation}\label{3D-hom-eq}
\left\{\begin{aligned}
\mathcal{L}(\phi_\mathrm{hom})
=\mathfrak{F}^{\ast}\quad&\mbox{in}\quad\mathcal{N}_{L,f_\ast}^+,\\
\phi_\mathrm{hom}=0\quad&\mbox{on}\quad\partial\mathcal{N}_{L,f_\ast}^+\setminus \Gamma_\mathrm{w}^L,\\
\nabla\phi_\mathrm{hom}\cdot\mathbf{n}_R=\mathfrak{h}\quad&\mbox{on}\quad\Gamma_\mathrm{w}^{L},
\end{aligned}\right.
\end{equation}
for $\mathfrak{F}^{\ast}, \mathfrak{h}$ defined by
\begin{equation}\label{def-F-ast}\begin{aligned}
\mathfrak{F}^{\ast}&:=\operatorname{div}\mathfrak{F}-\sum_{i=1}^3\partial_i\left(\sum_{j=1}^3a_{ij}\partial_{j}\varphi_\mathrm{en}^{\ast}\right),\\
\mathfrak{h}&:=\frac{R'(x_1)\left[u_0+\eta'(x_1)\varphi_\mathrm{en}(1)\right]}{\sqrt{1+|R'(x_1)|^2}},
\end{aligned}\end{equation}
where $a_{ij}$ $(i=1,2,3)$ are given by \eqref{aij-def}.
By the standard elliptic theory, the linear boundary value problem \eqref{3D-hom-eq} has a unique solution $\phi_\mathrm{hom}\in C^{1,\alpha}(\overline{\mathcal{N}_{L,f_\ast}^+})\cap C^{2,\alpha}(\mathcal{N}_{L,f_{\ast}}^+)$.

To obtain a uniform $C^0$-estimate of $\phi_\mathrm{hom}$ for all $L$, we define a function $\mathfrak{M}$ by
$$\mathfrak{M}(\mathbf{x}):=25\left(\frac{\|\mathfrak{F}^\ast\|_{\alpha,\mathcal{N}_{L,f_\ast}}}{{\rho_0^+}}+\|\mathfrak{h}\|_{\alpha,\Gamma_\mathrm{w}^L}\right)\left(2-e^{-4r}\right).$$
Since $\rho_0^+>\nu$ in $\mathcal{N}_{L,f_{\ast}}^+$ for some $\nu>0$, $\mathfrak{M}$ is well-defined.
A direct computation yields
\begin{equation*}
\left\{\begin{aligned}
\mathcal{L}(\mathfrak{M}\pm\phi_\mathrm{hom})\le0\quad&\mbox{in}\quad\mathcal{N}_{L,f_\ast}^+,\\
\mathfrak{M}\pm\phi_\mathrm{hom}=\mathfrak{M}\ge0\quad&\mbox{on}\quad\partial\mathcal{N}_{L,f_\ast}^+\setminus\Gamma_\mathrm{w}^L,\\
\nabla(\mathfrak{M}\pm\phi_\mathrm{hom})\cdot\mathbf{n}_R\ge 0\quad&\mbox{on}\quad\Gamma_\mathrm{w}^L.
\end{aligned}\right.
\end{equation*}
Since $\mathcal{L}$ is uniformly elliptic, the comparison principle implies
$-\mathfrak{M}\le\phi_\mathrm{hom}\le\mathfrak{M}$ in $\mathcal{N}_{L,f_{\ast}}^+,$ from which it follows that
$$\|\phi_\mathrm{hom}\|_{0,\mathcal{N}_{L,f_\ast}^+}\le C\left(\|\mathfrak{F}^{\ast}\|_{\alpha,\mathcal{N}_{L,f_\ast}^+}+\|\mathfrak{h}\|_{\alpha,\Gamma_\mathrm{w}^L}\right).$$
Then we obtain the estimate
$$\|\phi_\mathrm{hom}\|_{1,\alpha,\mathcal{N}_{L,f_\ast}^+}\le C\left(\|\mathfrak{F}^{\ast}\|_{\alpha,\mathcal{N}_{L,f_\ast}^+}+\|\mathfrak{h}\|_{\alpha,\Gamma_\mathrm{w}^L}\right).$$
To obtain a $C^{2,\alpha}$-estimate of $\phi_\mathrm{hom}$ up to the boundary, we use the method of reflection.
By the compatibility conditions of $(S_\ast,B_\ast,\Lambda_\ast,\tilde{\phi})$ given in \eqref{Ent-Ang-set} and \eqref{ell-set}, and $\partial_{x_1}\psi\equiv0$ on $\Gamma_\mathrm{en}^\epsilon\cup\Gamma^{L,f_{\ast}}_\mathrm{ex}$ given from \eqref{3D-psi-BC}, we have
\begin{equation}\label{F-0-ex}
\operatorname{div}\mathfrak{F}=\operatorname{div}\mathbf{F}(S_{\ast}-S_0^+,B_\ast-B_0^+,D\tilde{\phi},\mathbf{t}(r,{\psi},D{\psi},\Lambda_{\ast})-\beta_0\mathbf{e}_\theta)\equiv0\quad\mbox{on }\Gamma_\mathrm{en}^\epsilon\cup\Gamma^{L,f_{\ast}}_\mathrm{ex}.
\end{equation}
From the definition of $\varphi_\mathrm{en}^{\ast}$ given in \eqref{def-var-ast}, the compatibility conditions of $f_{\ast}$ given in \eqref{F-set}, and the definition of $\eta$ given in \eqref{eta-def}, it can be directly checked that
\begin{equation}\label{partial-varphi-en}
\partial_{ii}\varphi_\mathrm{en}^{\ast}\equiv 0\quad\mbox{on}\quad \Gamma_\mathrm{en}^\epsilon\cup\Gamma^{L,f_{\ast}}_\mathrm{ex},\quad i=1,2,3.
\end{equation}
It follows from \eqref{F-0-ex}--\eqref{partial-varphi-en} and the definition of $\mathfrak{F}^{\ast}, \mathfrak{h}$ given in \eqref{def-F-ast} that
\begin{equation*}
\mathfrak{F}^{\ast}\equiv 0\quad\mbox{on}\quad \Gamma_\mathrm{en}^\epsilon\cup\Gamma^{L,f_{\ast}}_\mathrm{ex}, \qquad \mathfrak{h}\equiv 0 \quad \mbox{on} \quad \Gamma_\mathrm{w}^L\cap(\Gamma_\mathrm{en}^\epsilon\cup\Gamma_\mathrm{ex}^{L,f_\ast}).
\end{equation*}
Then we can apply the method of reflection to obtain the estimate
$$\|\phi_\mathrm{hom}\|_{2,\alpha,\mathcal{N}_{L,f_\ast}^+}\le C\left(\|\mathfrak{F}^{\ast}\|_{\alpha,\mathcal{N}_{L,f_\ast}^+}+\|\mathfrak{H}\|_{1,\alpha,\Gamma_\mathrm{w}^L}\right),$$
and this implies that the linear boundary value problem \eqref{lin-phi} has a unique solution $\phi=\phi_\mathrm{hom}+\varphi_\mathrm{en}^{\ast}\in C^{2,\alpha}(\overline{\mathcal{N}_{L,f_{\ast}}^+})$ that satisfies
\begin{equation*}
\|\phi\|_{k,\alpha,\mathcal{N}_{L,f_{\ast}}^+}\le C\left(\|\mathfrak{F}\|_{k-1,\alpha,\mathcal{N}_{L,f_\ast}^+}+\|\mathfrak{H}\|_{1,\alpha,{\Gamma}_\mathrm{w}^L}+\|\varphi_\mathrm{en}\|_{k,\alpha,\Gamma_\mathrm{en}^+}\right)\quad\mbox{for }k=1,2.
\end{equation*}

Finally, we verify that the solution obtained above is axisymmetric by a rotation-invariance argument. For any $\vartheta\in[0,2\pi)$, define a function $\phi_\mathrm{hom}^{\vartheta}$ by
$$\phi_\mathrm{hom}^{\vartheta}(\mathbf{x}):=\phi_\mathrm{hom}(x_1,x_2\cos\vartheta-x_3\sin\vartheta,x_2\sin\vartheta+x_3\cos\vartheta).$$
Then, we have $\phi_\mathrm{hom}^{\vartheta}=\phi_\mathrm{hom}$ on $\partial \mathcal{N}_{L,f_\ast}^+$.
Using \eqref{aij-def}, it can be directly checked that $\mathcal{L}(\phi_\mathrm{hom}^{\vartheta})=\mathcal{L}(\phi_\mathrm{hom})$ holds in $\mathcal{N}_{L,f_\ast}^+$. This implies that $\phi_\mathrm{hom}^{\vartheta}$ is a solution to \eqref{3D-hom-eq}. By the uniqueness of the solution to \eqref{3D-hom-eq}, we conclude that
$\phi_\mathrm{hom}=\phi_\mathrm{hom}^{\vartheta}.$
Therefore $\phi_\mathrm{hom}$ is axisymmetric, and this implies that $\phi$ is axisymmetric.

Since $\phi_\mathrm{hom}\equiv 0$ and $\sum_{i=1}^3 a_{ii}\partial_{ii}\varphi_\mathrm{en}^{\ast}\equiv 0$ on $\Gamma_\mathrm{en}^\epsilon\cup\Gamma^{L,f_{\ast}}_\mathrm{ex}$, we have
\begin{equation}\label{phi-0-ex}
\partial_{ii}\phi\equiv0 \quad\mbox{on}\quad\Gamma_\mathrm{en}^\epsilon\cup\Gamma^{L,f_{\ast}}_\mathrm{ex}\quad\mbox{for}\quad i=2,3.
\end{equation}
It follows from \eqref{F-0-ex} and \eqref{phi-0-ex} that
$\mathcal{L}(\phi)=a_{11}\partial_{x_1 x_1}\phi\equiv 0$ on $\Gamma_\mathrm{en}^\epsilon\cup\Gamma^{L,f_{\ast}}_\mathrm{ex}.$
Since $a_{11}>0$, we conclude that
$\partial_{x_1 x_1}\phi\equiv 0$ on $\Gamma_\mathrm{en}^\epsilon\cup\Gamma^{L,f_{\ast}}_\mathrm{ex}.$
The verification of claim is completed.
\end{proof}

\textbf{5.}
For fixed $(\mathcal{W}_{\ast},f_{\ast})\in\mathcal{P}_{\delta_1\sigma}\times\mathcal{F}_{\delta_2\sigma}(\frac{1}{2})$,
define an iteration mapping $\mathcal{I}^{f_{\ast},\mathcal{W}_{\ast}}:\mathcal{K}^{f_\ast}_{\delta_3\sigma,\delta_4\sigma}\rightarrow [C^{2,\alpha}(\overline{\mathcal{N}_{L,f_\ast}^+})]^2$ by
\begin{equation*}
\mathcal{I}^{f_\ast,\mathcal{W}_{\ast}}(\tilde{\phi},\tilde{\psi})=(\phi,\psi),
\end{equation*}
where $(\phi,\psi)$ is the solution to \eqref{3D-psi-BC} and \eqref{lin-phi} associated with $(\tilde{\phi},\tilde{\psi})$.
We now show that, for suitable choices of $\delta_3$, $\delta_4$, and sufficiently small $\sigma$, the map $\mathcal{I}^{f_*,\mathcal{W}_*}$ sends $\mathcal{K}^{f_\ast}_{\delta_3\sigma,\delta_4\sigma}$ into itself and is a contraction.

A direct computation shows that there exists a constant $\epsilon_1\in(0,\frac{1}{8})$ depending only on the data such that if
\begin{equation}\label{epsilon1}
\delta_1\sigma+\delta_2\sigma+\delta_3\sigma+\delta_4\sigma\le \epsilon_1,
\end{equation}
then we have
\begin{equation}\label{est-lem}
\left.
\begin{aligned}
&\|{\mathfrak F}\|_{1,\alpha,\mathcal{N}_{L,f_\ast}^+}\le C\left(\delta_1\sigma+(\delta_3\sigma)^2+\delta_4\sigma\right),\\
&\|\mathfrak{H}\|_{1,\alpha,\Gamma_\mathrm{w}^L}\le C\sigma,\\
&\|\tilde{\mathcal{A}}\|_{1,\alpha,\Gamma_\mathrm{cd}^{L,f_\ast}}\le C\left(\delta_1\sigma+\beta_0^2\delta_2\sigma+(\delta_2\sigma)^2\right),\\
&\|\tilde{G}\|_{\alpha,\mathcal{N}_{L,f_\ast}^+}\le C\left(\delta_1\sigma+\beta_0^2\delta_3\sigma+\beta_0^2\delta_4\sigma\right),\\
\end{aligned}\right.
\end{equation}
where ${\mathfrak F}$, $\tilde{\mathcal{A}}$, and $\tilde{G}$ are given by \eqref{FF-def}, and \eqref{def-GB}.
It follows from \eqref{psi-est-2}, \eqref{3D-phi-est}, and \eqref{est-lem} that
\begin{equation}\label{pp-est-M3}
\left.\begin{aligned}
&\|\phi\|_{2,\alpha,\mathcal{N}_{L,f_\ast}^+}\le C_1^{\flat}\left(\delta_1\sigma+(\delta_3\sigma)^2+\delta_4\sigma+\sigma\right),\\
&\|\psi\|_{2,\alpha,\mathcal{N}_{L,f_{\ast}}^+}\le C_1^{\flat}\left(\delta_1\sigma+(\delta_2\sigma)^2+\beta_0^2\delta_2\sigma+\beta_0^2\delta_3\sigma+\beta_0^2\delta_4\sigma\right),
\end{aligned}\right.
\end{equation}
for a constant $C_1^{\flat}>0$ depending on the data but not on $L$.
We choose $\delta_3$, $\delta_4$, $\beta^{\ast\ast}$ and $\sigma_6^{\ast}$ as
\begin{equation}\label{3D-sigma8}
\begin{aligned}
&\delta_3=5C_1^{\flat}(1+\delta_1+\delta_4),\quad \delta_4=5C_1^{\flat}\delta_1,\quad \beta^{\ast\ast}=\min\left\{\beta^\natural,\sqrt{\frac{1}{5C_1^\flat}},\sqrt{\frac{\delta_4}{5C_1^\flat \delta_2}},\sqrt{\frac{\delta_4}{5C_1^\flat \delta_3}}\right\},\\
&\mbox{and}\quad \sigma_6^{\ast}=\min\left\{\frac{\epsilon_1}{\delta_1+\delta_2+\delta_3+\delta_4},\frac{\delta_4}{5C_1^{\flat}\delta_2^2},\frac{1}{5C_1^{\flat}\delta_3},\frac{1}{32},\frac{1}{32\delta_2}\right\},
\end{aligned}
\end{equation}
where $\epsilon_1$ is given in \eqref{epsilon1},
so that \eqref{pp-est-M3} implies that
$(\phi,\psi)\in\mathcal{K}^{f_\ast}_{\delta_3\sigma,\delta_4\sigma}$ for $|\beta_0|\le\beta^{\ast\ast}, \sigma\le\sigma_6^{\ast}.$
Under such choices of $(\delta_3, \delta_4, \sigma_6^{\ast})$, the iteration mapping
$\mathcal{I}^{f_{\ast},\mathcal{W}_{\ast}}$ maps $\mathcal{K}^{f_\ast}_{\delta_3\sigma,\delta_4\sigma}$ into itself if $|\beta_0|\le\beta^{\ast\ast}, \sigma\le\sigma_6^{\ast}$.
Furthermore, $(\phi,\psi)$ satisfies the estimate
\begin{equation*}
\|\phi\|_{2,\alpha,\mathcal{N}_{L,f_{\ast}}^+}+\|\psi\|_{2,\alpha,\mathcal{N}_{L,f_{\ast}}^+}\le (\delta_3+\delta_4)\sigma\le C(\delta_1+1)\sigma.
\end{equation*}

Now we show that $\mathcal{I}^{f_\ast,\mathcal{W}_{\ast}}$ is a contraction mapping if $\sigma$ is a small constant depending only on the data and $(\delta_1,\delta_2)$.

For each $j=1,2$, let
\begin{equation*}\left\{
\begin{aligned}
&(\phi^{(j)},\psi^{(j)}):=\mathcal{I}^{f_\ast,\mathcal{W}_\ast}(\tilde{\phi}^{(j)},\tilde{\psi}^{(j)})\quad\mbox{for }(\tilde{\phi}^{(j)},\tilde{\psi}^{(j)})\in\mathcal{K}^{f_\ast}_{\delta_3\sigma,\delta_4\sigma},\\
&\mathbf{F}_{\ast}:=\mathbf{F}(S_{\ast}-S_0^+,B_{\ast}-B_0^+,D\tilde{\phi}^{(1)},\mathbf{t}(r,\psi^{(1)},D\psi^{(1)},\Lambda_{\ast})-\beta_0\mathbf{e}_\theta)\\
&\qquad -\mathbf{F}(S_{\ast}-S_0^+,B_{\ast}-B_0^+,D\tilde{\phi}^{(2)},\mathbf{t}(r,\psi^{(2)},D\psi^{(2)},\Lambda_{\ast})-\beta_0\mathbf{e}_\theta),\\
&G_{\ast}:=G(\mathcal{W}_{\ast},\partial_r\mathcal{W}_{\ast},\mathbf{t}(r,\tilde{\psi}^{(1)},D\tilde{\psi}^{(1)},\Lambda_{\ast}),D\tilde{\phi}^{(1)}+D\varphi_0)\\
&\qquad -G(\mathcal{W}_{\ast},\partial_r\mathcal{W}_{\ast},\mathbf{t}(r,\tilde{\psi}^{(2)},D\tilde{\psi}^{(2)},\Lambda_{\ast}),D\tilde{\phi}^{(2)}+D\varphi_0),
\end{aligned}\right.
\end{equation*}
where $\mathbf{F}$ and $G$ are given by \eqref{def-F} and \eqref{def-H-G}, respectively.
A direct computation shows that there exists a constant $\epsilon_2\in(0,\epsilon_1]$ depending only on the data such that if $$\delta_1\sigma+\delta_2\sigma+\delta_3\sigma+\delta_4\sigma\le\epsilon_2,$$ then we have
\begin{equation}\label{diff-est}
\left.\begin{aligned}
&\|\mathbf{F}_{\ast}\|_{1,\alpha,\mathcal{N}_{L,f_\ast}^+}\le C\|\psi^{(1)}-\psi^{(2)}\|_{2,\alpha,\mathcal{N}_{L,f_{\ast}}^+}+C(\delta_1+1)\sigma\|\tilde{\phi}^{(1)}-\tilde{\phi}^{(2)}\|_{2,\alpha,\mathcal{N}_{L,f_\ast}^+},\\
&\|G_{\ast}\|_{\alpha,\mathcal{N}_{L,f_\ast}^+}\le C\left[\beta_0^2+(\delta_1+1)\sigma\right]\left(\|\tilde{\psi}^{(1)}-\tilde{\psi}^{(2)}\|_{2,\alpha,\mathcal{N}_{L,f_{\ast}}^+}+\|\tilde{\phi}^{(1)}-\tilde{\phi}^{(2)}\|_{2,\alpha,\mathcal{N}_{L,f_\ast}^+}\right).
\end{aligned}\right.
\end{equation}
Then it follows from \eqref{psi-est-2}, \eqref{3D-phi-est}, and \eqref{diff-est} that
\begin{equation*}
\begin{aligned}
\|\phi^{(1)}&-\phi^{(2)}\|_{2,\alpha,\mathcal{N}_{L,f_\ast}^+}+\|\psi^{(1)}-\psi^{(2)}\|_{2,\alpha,\mathcal{N}_{L,f_{\ast}}^+}\\
&\le C_2^{\flat}\left[\beta_0^2+(\delta_1+1)\sigma\right]\left(\|\tilde{\psi}^{(1)}-\tilde{\psi}^{(2)}\|_{2,\alpha,\mathcal{N}_{L,f_\ast}^+}+\|\tilde{\phi}^{(1)}-\tilde{\phi}^{(2)}\|_{2,\alpha,\mathcal{N}_{L,f_\ast}^+}\right)
\end{aligned}
\end{equation*}
for a constant $C_2^{\flat}>0$ depending only on the data but not on $L$.
Choose $\sigma_6$ and $\beta^\ast$ as
\begin{equation}\label{Sigma6}
\sigma_6=\min\left\{\sigma_6^{\ast},\frac{1}{2C_2^{\flat}(\delta_1+1)},\frac{\epsilon_2}{\delta_1+\delta_2+\delta_3+\delta_4}\right\},\quad\beta^\ast=\min\left\{\beta^{\ast\ast},\sqrt{\frac{1}{3C_2^\flat}}\right\}
\end{equation}
with $\sigma_6^{\ast}$ and $\beta^{\ast\ast}$ defined in \eqref{3D-sigma8}. Thus if $|\beta_0|\le\beta^\ast, \sigma\le \sigma_6$, then the mapping $\mathcal{I}^{f_{\ast},\mathcal{W}_\ast}$ is a contraction mapping so that $\mathcal{I}^{f_{\ast},\mathcal{W}_{\ast}}$ has a unique fixed point in $\mathcal{K}^{f_\ast}_{\delta_3\sigma,\delta_4\sigma}$.
This proves the existence and uniqueness of a solution to \eqref{Fixed-BVP}.
The proof of Lemma~\ref{Pro-fix-S} is completed.
\end{proof}

Next, we prove the existence and uniqueness of a solution to Problem~\ref{Prob4-Fix-S}, which is a free boundary problem.
\begin{proof}[Proof of Lemma~\ref{Lem-S-free}.]
\textbf{1.}
Now we choose $\delta_2$ from \eqref{F-set}, and adjust $\sigma$ to find a solution of Problem~\ref{Prob4-Fix-S} by the method of iteration.

Given $f_{\ast}\in\mathcal{F}_{\delta_2\sigma}(\frac{1}{2})$ and $(S_{\ast},B_{\ast},\Lambda_{\ast})\in\mathcal{P}_{\delta_1\sigma}$, let $(\varphi,\psi)\in [C^{2,\alpha}(\overline{\mathcal{N}_{L,f_{\ast}}^+})]^2$ be the unique solution to the boundary value problem \eqref{Fixed-BVP}. Note that $(\varphi,\psi)$ satisfies the estimate \eqref{fix-est} given in Lemma~\ref{Pro-fix-S}.
For simplicity, we set
\begin{equation*}
\begin{aligned}
&\rho^{\ast}:=\varrho(S_{\ast},B_{\ast},\mathbf{q}(r,\psi,D\psi,D\varphi,\Lambda_{\ast})),\\
&\mathbf{u}^{\ast}:=\left(\partial_{x_1}\varphi+\frac{1}{r}\partial_r(r\psi)\right)\mathbf{e}_{x_1}+\left(\partial_r\varphi-\partial_{x_1}\psi\right)\mathbf{e}_r,
\end{aligned}
\end{equation*}
where $\varrho$ is given in \eqref{def-H-G}.
From the first equation in \eqref{Fixed-BVP}, we have
\begin{equation}\label{conti-u}
\partial_{x_1}(r\rho^{\ast}\mathbf{u}^{\ast}\cdot\mathbf{e}_{x_1})+\partial_r(r\rho^{\ast}\mathbf{u}^{\ast}\cdot\mathbf{e}_r)=0.
\end{equation}
As in the proof of Lemma~\ref{Pro-fix-S}, there exists a constant $\epsilon_3\in(0,1]$ depending only on the data such that if
\begin{equation*}\label{epsilon3}
\delta_1\sigma+\delta_2\sigma+\sigma\le\epsilon_3,
\end{equation*}
then we have
\begin{equation}\label{rhou-est-1}
\|\rho^{\ast}\mathbf{u}^{\ast}-\rho_0^+u_0\mathbf{e}_{x_1}\|_{1,\alpha,\mathcal{N}_{L,f_\ast}^+}\le C_{\star}(\delta_1+1)\sigma,
\end{equation}
where the constant $C_{\star}>0$ depends only on the data but not on $L$.
If $\sigma\in(0,\sigma_6]$ satisfies
\begin{equation*}
\sigma\le \frac{\rho_0^+u_0}{32C_{\star}(\delta_1+1)},
\end{equation*}
then we obtain from \eqref{rhou-est-1} that
\begin{equation}\label{rhou-est-2}
\|\rho^{\ast}\mathbf{u}^{\ast}-\rho_0^+u_0\mathbf{e}_{x_1}\|_{1,\alpha,\mathcal{N}_{L,f_\ast}^+}\le \frac{\rho_0^+u_0}{32}.
\end{equation}

For each $x_1\in[0,L]$, we choose $f(x_1)\in\mathbb{R}^+$ to satisfy
\begin{equation}\label{3D-f-est1}
\int_{f(x_1)}^{f_\ast(x_1)} t \rho_0^+u_0 dt=\int_{1/2}^1 t\rho^{\ast}\mathbf{u}^{\ast}\cdot\mathbf{e}_{x_1}(0,t)dt-\int_{f_{\ast}(x_1)}^{R(x_1)} t\rho^{\ast}\mathbf{u}^{\ast}\cdot\mathbf{e}_{x_1}(x_1,t)dt.
\end{equation}
If $f\equiv f_{\ast}$, then \eqref{3D-f-est1} yields that
\begin{equation}\label{fix-f-eq}
\int_{1/2}^1 t\rho^{\ast}\mathbf{u}^{\ast}\cdot\mathbf{e}_{x_1}(0,t)dt=\int_{f(x_1)}^{R(x_1)} t\rho^{\ast}\mathbf{u}^{\ast}\cdot\mathbf{e}_{x_1}(x_1,t)dt.
\end{equation}
Differentiating \eqref{fix-f-eq} with respect to $x_1$, and using the equation \eqref{conti-u} and the fact that $$R'(x_1)\mathbf{u}^\ast\cdot\mathbf{e}_{x_1}(x_1,R(x_1))=\mathbf{u}^\ast\cdot\mathbf{e}_r(x_1,R(x_1)),$$
we have
\begin{equation*}
f'(x_1)=\frac{\mathbf{u}^{\ast}\cdot\mathbf{e}_r}{\mathbf{u}^{\ast}\cdot\mathbf{e}_{x_1}}(x_1,f(x_1))=\frac{\partial_r\varphi-\partial_{x_1}\psi}{\partial_{x_1}\varphi+\frac{1}{r}\partial_r(r\psi)}(x_1,f(x_1)).
\end{equation*}
Also, we have $f(0)=\frac{1}{2}$.
Thus $f$ satisfies the free boundary condition \eqref{g-free-cut} for $0<x_1<L$.

Since $\rho_0^+u_0>0$, \eqref{3D-f-est1} is equivalent to
\begin{equation}\label{3D-f-est2}
f^2(x_1)=f_\ast^2(x_1)+\frac{2}{\rho_0^+u_0}\int_{f_\ast(x_1)}^{R(x_1)} t\rho^{\ast}\mathbf{u}^{\ast}\cdot\mathbf{e}_{x_1}(x_1,t)dt-\frac{2}{\rho_0^+u_0}\int_{1/2}^{1} t\rho^{\ast}\mathbf{u}^{\ast}\cdot\mathbf{e}_{x_1}(0,t)dt.
\end{equation}
By \eqref{rhou-est-1} and \eqref{rhou-est-2},
\begin{equation}
\label{3D-f-est4}
\mbox{RHS of \eqref{3D-f-est2}}\ge
\frac{1}{32}>0\quad\mbox{if}\quad\sigma\le\min\left\{\frac{\epsilon_3}{\delta_1+\delta_2+1}, \frac{\rho_0^+u_0}{32 C_{\star}(\delta_1+1)}\right\}=:\sigma_5'.
\end{equation}
Then the function $f:[0,L]\rightarrow\mathbb{R}^+$ given by
\begin{equation}\label{3D-f-est3}
f(x_1):=\sqrt{f_\ast^2(x_1)+\frac{2}{\rho_0^+u_0}\int_{f_\ast(x_1)}^{R(x_1)} t\rho^{\ast}\mathbf{u}^{\ast}\cdot\mathbf{e}_{x_1}(x_1,t)dt-\frac{2}{\rho_0^+u_0}\int_{1/2}^1 t\rho^{\ast}\mathbf{u}^{\ast}\cdot\mathbf{e}_{x_1}(0,t)dt}
\end{equation}
is well defined and satisfies \eqref{3D-f-est1}.
Moreover, $f(0)=\frac{1}{2}$, $f'(0)=f'(L)=0$.
A direct computation yields the estimate
\begin{equation}\label{3D-f-est7}
\left\|f-\frac{1}{2}\right\|_{2,\alpha,(0,L)}\le C_{\star\star}(\delta_1+1)\sigma
\end{equation}
for a constant $C_{\star\star}>0$ depending only on the data but not on $L$.

We define an iteration mapping $\mathcal{I}^{\mathcal{W}_{\ast}}:\mathcal{F}_{\delta_2\sigma}(\frac{1}{2})\rightarrow C^{2,\alpha}([0,L])$ by
$$\mathcal{I}^{\mathcal{W}_{\ast}}(f_{\ast})=f$$
for $f$ given by \eqref{3D-f-est3}.
Choose $\delta_2$ and $\sigma_5^{\ast}$ as
\begin{equation}\label{Sigma5star}
\delta_2=C_{\star\star}(\delta_1+1)\quad\mbox{and}\quad\sigma_5^{\ast}=\min\left\{\sigma_6,\sigma_5'\right\}
\end{equation}
for $\sigma_6$ and $\sigma_5'$ defined by \eqref{Sigma6} and \eqref{3D-f-est4}, respectively.
Under such choices of $(\delta_2,\sigma_5^{\ast})$, the iteration mapping $\mathcal{I}^{\mathcal{W}_{\ast}}$ maps $\mathcal{F}_{\delta_2\sigma}(\frac{1}{2})$ into itself if $\sigma\le\sigma_5^{\ast}$.

\textbf{2.}
The iteration set $\mathcal{F}_{\delta_2\sigma}(\frac{1}{2})$ given by \eqref{F-set} is a convex and compact subset of $C^{2,\alpha/2}([0,L])$.
For each fixed $\mathcal{W}_{\ast}\in \mathcal{P}_{\delta_1\sigma}$, the iteration map $\mathcal{I}^{\mathcal{W}_{\ast}}$ maps $\mathcal{F}_{\delta_2\sigma}(\frac{1}{2})$ into itself where $\delta_2$ is chosen by \eqref{Sigma5star}, and $\sigma\le \sigma_5^{\ast}$ for $\sigma_5^{\ast}$ from \eqref{Sigma5star}.

Suppose that a sequence $\{f_{\ast}^{(k)}\}_{k=1}^{\infty}\subset\mathcal{F}_{\delta_2\sigma}(\frac{1}{2})$ converges in $C^{2,\alpha/2}([0,L])$ to $f_{\ast}^{(\infty)}\in\mathcal{F}_{\delta_2\sigma}(\frac{1}{2})$.
For each $k\in\mathbb{N}\cup\{\infty\}$, set
\begin{equation*}
f^{(k)}:=\mathcal{I}^{\mathcal{W}_{\ast}}(f_{\ast}^{(k)}).
\end{equation*}
Let $\mathcal{U}^{(k)}:=(\varphi^{(k)},\psi^{(k)})\in [C^{2,\alpha}(\overline{\mathcal{N}_{L,f_{\ast}^{(k)}}^+})]^2$ be the unique solution of \eqref{Fixed-BVP} associated with $f_{\ast}=f_{\ast}^{(k)}$.
Define a transformation $T^{(k)}:\overline{\mathcal{N}_{L,f_{\ast}^{(\infty)}}^+}\rightarrow \overline{\mathcal{N}_{L,f_{\ast}^{(k)}}^+}$ by
\begin{equation*}
T^{(k)}(x_1,r,\theta)=\left(x_1,\frac{R(x_1)-f_\ast^{(k)}(x_1)}{R(x_1)-f_\ast^{(\infty)}(x_1)}\left(r-R(x_1)\right)+R(x_1),\theta\right).\end{equation*}
Then $\{\mathcal{U}^{(k)}\circ T^{(k)}\}_{k=1}^{\infty}$ is sequentially compact in $[C^{2,\alpha/2}(\overline{\mathcal{N}_{L,f_\ast^{(\infty)}}^+})]^2$ and the limit of each convergent subsequence of $\{\mathcal{U}^{(k)}\circ T^{(k)}\}_{k=1}^{\infty}$ in $[C^{2,\alpha/2}(\overline{\mathcal{N}_{L,f_\ast^{(\infty)}}^+})]^2$ solves \eqref{Fixed-BVP} associated with $f_{\ast}=f_{\ast}^{(\infty)}$.
Due to the uniqueness of \eqref{Fixed-BVP}, the entire sequence $\{\mathcal{U}^{(k)}\circ T^{(k)}\}_{k=1}^{\infty}$ converges in $[C^{2,\alpha/2}(\overline{\mathcal{N}_{L,f_\ast^{(\infty)}}^+})]^2$.
It follows from \eqref{3D-f-est3}--\eqref{3D-f-est7} that $f^{(k)}$ converges to $f^{(\infty)}$ in $C^{2,\alpha/2}([0,L])$.
This implies that $\mathcal{I}^{\mathcal{W}_{\ast}}(f_{\ast}^{(k)})$ converges to $\mathcal{I}^{\mathcal{W}_{\ast}}(f_{\ast}^{(\infty)})$ in $C^{2,\alpha/2}([0,L])$.
Thus $\mathcal{I}^{\mathcal{W}_{\ast}}$ is a continuous map in $C^{2,\alpha/2}([0,L])$.
Applying the Schauder fixed point theorem yields that $\mathcal{I}^{\mathcal{W}_{\ast}}$ has a fixed point $f\in\mathcal{F}_{\delta_2\sigma}(\frac{1}{2})$.
For such $f$, let $(\varphi,\psi)\in [C^{2,\alpha}(\overline{\mathcal{N}_{L,f}^+})]^2$ be the unique solution to the fixed boundary problem \eqref{Fixed-BVP} associated with $f_{\ast}=f$.
Then $(f,\varphi,\psi)$ is a solution to Problem~\ref{Prob4-Fix-S}.
It follows from \eqref{fix-est} and \eqref{3D-f-est7} that
\begin{equation*}
\left\|f-\frac{1}{2}\right\|_{2,\alpha,(0,L)}+\|\varphi-\varphi_0\|_{2,\alpha,\mathcal{N}_{L,f}^+}+\|\psi\|_{2,\alpha,\mathcal{N}_{L,f}^+}\le C\left(\delta_1+1\right)\sigma.
\end{equation*}

\textbf{3.} Finally, it remains to prove the uniqueness of a solution to Problem~\ref{Prob4-Fix-S}.
For a fixed $\mathcal{W}_{\ast}\in\mathcal{P}_{\delta_1\sigma}$, let $(f^{(1)},\varphi^{(1)},\psi^{(1)})$ and $(f^{(2)},\varphi^{(2)},\psi^{(2)})$ be two solutions to Problem~\ref{Prob4-Fix-S}, and suppose that each solution satisfies the estimate given in \eqref{3D-pps-est} of Lemma~\ref{Lem-S-free}.
Define a transformation $\mathfrak{T}:\overline{\mathcal{N}_{L,f^{(1)}}^+}\rightarrow \overline{\mathcal{N}_{L,f^{(2)}}^+}$ by
\begin{equation}\label{TT}
\mathfrak{T}(x_1,r,\theta)=\left(x_1,\frac{R(x_1)-f^{(2)}(x_1)}{R(x_1)-f^{(1)}(x_1)}\left(r-R(x_1)\right)+R(x_1),\theta\right).
\end{equation}
Since $f^{(j)}\ge \frac{3}{8}>0$ $(j=1,2)$, the transformation $\mathfrak{T}$ is invertible and
$$\mathfrak{T}^{-1}(x_1,r,\theta)=\left(x_1,\frac{R(x_1)-f^{(1)}(x_1)}{R(x_1)-f^{(2)}(x_1)}\left(r-R(x_1)\right)+R(x_1),\theta\right).$$
Set
\begin{equation*}
\left\{\begin{aligned}
&\widetilde{\phi}:=\varphi^{(1)}-\left(\varphi^{(2)}\circ\mathfrak{T}\right),\quad\widetilde{\psi}:=\psi^{(1)}-\left(\psi^{(2)}\circ\mathfrak{T}\right),\quad\widetilde{Z}:=r\psi^{(1)}-\left((r\psi^{(2)})\circ\mathfrak{T}\right),\\
& \widetilde{f}:=f^{(1)}-f^{(2)},\quad \widetilde{\Xi}:=\Xi_{\ast}-\left(\Xi_{\ast}\circ\mathfrak{T}\right),
\end{aligned}\right.
\end{equation*} where $$\Xi_\ast:=\left(S_\ast,B_\ast,\frac{\Lambda_\ast^2}{2}\right).$$
We first rewrite the nonlinear boundary value problem \eqref{S-Free-BP} for $(\varphi^{(2)},\psi^{(2)})$ in $\mathcal{N}_{L,f^{(2)}}^+$ as a nonlinear boundary value problem for $(\varphi^{(2)}\circ\mathfrak{T},\psi^{(2)}\circ\mathfrak{T})$ in $\mathcal{N}_{L,f^{(1)}}^+$, and subtract the resultant equations and boundary conditions from the nonlinear boundary value problem \eqref{S-Free-BP} for $(\varphi^{(1)},\psi^{(1)})$ in $\mathcal{N}_{L,f^{(1)}}^+$.
Then we get a nonlinear boundary value problem for $(\widetilde{\phi},\widetilde{\psi})$ in $\mathcal{N}_{L,f^{(1)}}^+$.
In the equation for $\widetilde Z$, we keep the terms involving derivatives of $\widetilde{\Xi}$ in divergence form and apply $C^{1,\alpha}$ estimates for divergence form elliptic equations with the transformed conormal boundary condition.
By adapting the proof of Lemma~\ref{Pro-fix-S} and using
\begin{equation*}
\|\widetilde{\Xi}\|_{\alpha,\mathcal{N}_{L,f^{(1)}}^+}\le C(\delta_1\sigma+\beta_0^2)\|\widetilde{f}\|_{1,\alpha,(0,L)},
\end{equation*}
we obtain \begin{equation*}
\|\widetilde{Z}\|_{1,\alpha,\mathcal{N}_{L,f^{(1)}}^+}\le C[(\delta_1+1)\sigma+\beta_0^2]\left(\|\widetilde{\phi}\|_{1,\alpha,\mathcal{N}_{L,f^{(1)}}^+}+\|\widetilde{Z}\|_{1,\alpha,\mathcal{N}_{L,f^{(1)}}^+}+\|\widetilde{f}\|_{1,\alpha,(0,L)}\right).
\end{equation*} Therefore we have
\begin{equation*}
\begin{aligned}
\|\widetilde{\phi}\|_{1,\alpha,\mathcal{N}_{L,f^{(1)}}^+}+\|\widetilde{Z}\|_{1,\alpha,\mathcal{N}_{L,f^{(1)}}^+}
\le& C_1^{\ast}[(\delta_1+1)\sigma+\beta_0^2]\left(\|\widetilde{\phi}\|_{1,\alpha,\mathcal{N}_{L,f^{(1)}}^+}+\|\widetilde{Z}\|_{1,\alpha,\mathcal{N}_{L,f^{(1)}}^+}\right)\\
&+C[(\delta_1+1)\sigma+\beta_0^2]\|\widetilde{f}\|_{1,\alpha,(0,L)}
\end{aligned}
\end{equation*}
for a constant $C_1^{\ast}>0$ depending only on the data but not on $L$.
If it holds that
$$\beta_0^2\le \frac{1}{4C_1^{\ast}},\quad\sigma\le \frac{1}{4C_1^{\ast}(\delta_1+1)},$$
then we obtain from the previous estimate that
\begin{equation}\label{RR-est}
\|\widetilde{\phi}\|_{1,\alpha,\mathcal{N}_{L,f^{(1)}}^+}+\|\widetilde{\psi}\|_{1,\alpha,\mathcal{N}_{L,f^{(1)}}^+}\le C[(\delta_1+1)\sigma+\beta_0^2]\|\widetilde{f}\|_{1,\alpha,(0,L)}.
\end{equation}
By using the free boundary condition \eqref{g-free-cut}, we can express $(\widetilde{f})'$ in terms of $(\widetilde{\phi},\widetilde{\psi},\mathfrak{T},D\mathfrak{T})$.
Then we apply \eqref{RR-est} to obtain the estimate
\begin{equation}\label{3D-g12_EST}
\|(\widetilde{f})'\|_{\alpha,(0,L)}\le C[(\delta_1+1)\sigma+\beta_0^2]\|\widetilde{f}\|_{1,\alpha,(0,L)}.
\end{equation}
To complete the estimate of $\|\widetilde{f}\|_{1,\alpha,(0,L)}=\|\widetilde{f}\|_{0,(0,L)}+\|(\widetilde{f})'\|_{\alpha,(0,L)}$, we now estimate $\|\widetilde{f}\|_{0,(0,L)}$.
Define $\rho^{(1)}$, $u_{x_1}^{(1)}$, $\rho^{(2)}$, and $u_{x_1}^{(2)}$ by
\begin{equation*}
\begin{aligned}
&\rho^{(k)}:=\varrho(S_{\ast},B_{\ast},\mathbf{q}(r,\psi^{(k)},D\psi^{(k)},D\varphi^{(k)},\Lambda_{\ast})),\\
& u_{x_1}^{(k)}:=\partial_{x_1}\varphi^{(k)}+\frac{1}{r}\partial_r(r\psi^{(k)})\quad\mbox{for}\quad k=1,2,
\end{aligned}
\end{equation*}
where $\varrho$ is given by \eqref{def-H-G}.
By using \eqref{fix-f-eq}, we get
\begin{equation}\label{f12-rhou}
\begin{aligned}
\int_{1/2}^1&r\left(\rho^{(1)}u_{x_1}^{(1)}-\rho^{(2)}u_{x_1}^{(2)}\right)(0,r)dr\\
&=\int_{f^{(1)}(x_1)}^{R(x_1)} r\rho^{(1)}u_{x_1}^{(1)}(x_1,r)dr-\int_{f^{(2)}(x_1)}^{R(x_1)} r\rho^{(2)}u_{x_1}^{(2)}(x_1,r)dr.
\end{aligned}
\end{equation}
Fix $x_0\in[0,L]$.
Without loss of generality, we may assume that
$$f^{(1)}(x_0)\le f^{(2)}(x_0).$$
Then \eqref{f12-rhou} can be rewritten as
\begin{equation*}
\begin{aligned}
\int_{1/2}^1&r\left(\rho^{(1)}u_{x_1}^{(1)}-\rho^{(2)}u_{x_1}^{(2)}\right)(0,r)dr\\
&=\int_{f^{(2)}(x_0)}^{R(x_0)} r\left(\rho^{(1)}u_{x_1}^{(1)}-\rho^{(2)}u_{x_1}^{(2)}\right)(x_0,r)dr+\int_{f^{(1)}(x_0)}^{f^{(2)}(x_0)} r\rho^{(1)}u_{x_1}^{(1)}(x_0,r)dr.
\end{aligned}
\end{equation*}
By applying \eqref{RR-est}, we have
\begin{equation*}
0\le f^{(2)}(x_0)-f^{(1)}(x_0)\le C[(\delta_1+1)\sigma+\beta_0^2]\|\widetilde{f}\|_{1,\alpha,(0,L)}.
\end{equation*}
Combining this with \eqref{3D-g12_EST}, we finally get
\begin{equation}\label{final-f}
\|\widetilde{f}\|_{1,\alpha,(0,L)}\le C_2^{\ast}[(\delta_1+1)\sigma+\beta_0^2]\|\widetilde{f}\|_{1,\alpha,(0,L)},
\end{equation}
where the constant $C_2^{\ast}>0$ depends only on the data but not on $L$.
We choose $\sigma_5$ and $\beta^\star$ as
\begin{equation}\label{Sigma5}
\sigma_5=\min\left\{\sigma_5^{\ast},\frac{1}{4C_1^{\ast}(\delta_1+1)},\frac{1}{4C_2^{\ast}(\delta_1+1)}\right\}\quad\text{and}\quad\beta^\star=\min\left\{\beta^\ast,\sqrt{\frac{1}{4C_1^\ast(\delta_1+1)}},\sqrt{\frac{1}{4C_2^{\ast}(\delta_1+1)}}\right\}
\end{equation}
for $\sigma_5^{\ast}$ defined in \eqref{Sigma5star},
so that \eqref{final-f} implies that $f^{(1)}=f^{(2)}$ for $\sigma\le\sigma_5$.
By Lemma~\ref{Pro-fix-S}, $(\varphi^{(1)},\psi^{(1)})=(\varphi^{(2)},\psi^{(2)}).$
The proof of Lemma~\ref{Lem-S-free} is completed.
\end{proof}

\subsection{Proof of Proposition~\ref{3D-Prop4.1}}
\label{subsection_4_3}
The proof of Proposition~\ref{3D-Prop4.1} is divided into four steps.

\textbf{1.}
Fix $\mathcal{W}_{\ast}=(S_*,B_*,\Lambda_*)\in\mathcal{P}_{\delta_1\sigma}$.
By Lemma~\ref{Lem-S-free}, if $|\beta_0|\le\beta^\star, \sigma\le \sigma_5$, then Problem~\ref{Prob4-Fix-S} has a unique solution $(f,\varphi,\psi)\in C^{2,\alpha}([0,L])\times [C^{2,\alpha}(\overline{\mathcal{N}_{L,f}^+})]^2$ satisfying \eqref{3D-pps-est}.
For convenience, we denote $$\begin{aligned}\mathcal{W}_\mathrm{en}:=(S_\mathrm{en},B_0^+,r\beta_\mathrm{en}), \qquad \mathcal{W}_0^+:=(S_0^+,B_0^+,\beta_0 r), \\\Omega_{L,f}^+:=\{(x_1,r)\in\mathbb{R}^2:0<x_1<L,\,f(x_1)<r<R(x_1)\}.\end{aligned}$$

\begin{lemma} \label{Pro-trans}
Under the same assumptions on $(S_\mathrm{en},\beta_\mathrm{en},u_r^\mathrm{en})$ as in Proposition~\ref{3D-Prop4.1},
there exist a positive constant $\beta^\sharp\in(0,\beta^\star]$ and a small constant $\sigma_4^{\ast\ast}\in(0,\sigma_5]$ depending only on the data such that if
\begin{itemize}[label=--]
\item $|\beta_0|\le\beta^\sharp$;
\item $\sigma\le\sigma_4^{\ast\ast}$,
\end{itemize}
then the initial value problem \eqref{Ite-3D2} has a unique solution $\mathcal{W}=(S, B, \Lambda)$ satisfying
\begin{equation*}
\|\mathcal{W}-\mathcal{W}_0^+\|_{1,\alpha,\mathcal{N}_{L,f}^+}\le C^{\ast}\|\mathcal{W}_\mathrm{en}-\mathcal{W}_0^+\|_{1,\alpha,\Gamma_\mathrm{en}^+}+C^{\ast}(|\beta_0|+\beta_0^2)(\delta_1+1)\sigma
\end{equation*}
for a constant $C^{\ast}>0$ depending only on the data but not on $L$.
\end{lemma}

\begin{proof}[Proof of Lemma~\ref{Pro-trans}.]
Define a function $w:\overline{\Omega_{L,f}^+}\rightarrow\mathbb{R}$ by
\begin{equation}\label{def-w}
w(x_1,r):=\int_{R(x_1)}^r s\mathbf{M}\cdot\mathbf{e}_{x_1}(x_1,s)ds\quad\mbox{for}\quad (x_1,r)\in\overline{\Omega_{L,f}^+}
\end{equation}
for
$$\mathbf{M}=\varrho(S_{\ast},B_\ast,\nabla\varphi+\mathbf{t}(r,\psi,D\psi,\Lambda_{\ast}))\left(\nabla\varphi+\frac{1}{r}\partial_r(r\psi)\mathbf{e}_{x_1}-(\partial_{x_1}\psi)\mathbf{e}_r\right),$$
where $\mathbf{t}$ and $\varrho$ are given by \eqref{def-T} and \eqref{def-H-G}, respectively.
For such $w$, we consider an invertible function $\mathcal{G}:[1/2,1]\rightarrow [w(0,1/2),w(0,1)]$ satisfying
\begin{equation}\label{def-G0}
\mathcal{G}(r)=w(0,r),
\end{equation}
and define a function $\mathcal{R}_0:\overline{\Omega_{L,f}^+}\rightarrow[1/2,1]$ by
\begin{equation}\label{3D-R0}
\mathcal{R}_0(x_1,r):=\mathcal{G}^{-1}\circ w(x_1,r).
\end{equation}
Note that \begin{equation*}
\|\mathcal{R}_0\|_{1,\alpha,\mathcal{N}_{L,f}^+}\le C\|\mathbf{M}\|_{\alpha,\mathcal{N}_{L,f}^+}.
\end{equation*}
By adjusting the proof of \cite[Proposition 3.5]{bae20183}, we can obtain a unique solution $\mathcal{W}$ of \eqref{Ite-3D2} represented in
\begin{equation}\label{def-W}
\mathcal{W}(x_1,r)=\mathcal{W}_\mathrm{en}(\mathcal{R}_0(x_1,r)),
\end{equation}
and the estimate
\begin{equation*}\begin{aligned}
\|\mathcal{W}-\mathcal{W}_0^+\|_{1,\alpha,\mathcal{N}_{L,f}^+}&\le\|\mathcal{W}_\mathrm{en}\circ \mathcal{R}_0-\mathcal{W}_0^+\circ\mathcal{R}_0\|_{1,\alpha,\mathcal{N}_{L,f}^+}+\|\mathcal{W}_0^+\circ\mathcal{R}_0-\mathcal{W}_0^+\|_{1,\alpha,\mathcal{N}_{L,f}^+}
\\&\le C_1^{\ast}\|\mathcal{W}_\mathrm{en}-\mathcal{W}_0^+\|_{1,\alpha,\Gamma_\mathrm{en}^+}+\|\mathcal{W}_0^+\circ\mathcal{R}_0-\mathcal{W}_0^+\|_{1,\alpha,\mathcal{N}_{L,f}^+},
\end{aligned}\end{equation*}
where the constant $C_1^{\ast}>0$ depends only on the data but not on $L$.
A direct computation shows that $$\|\mathcal{W}_0^+\circ\mathcal{R}_0-\mathcal{W}_0^+\|_{1,\alpha,\mathcal{N}_{L,f}^+}\le C_2^{\ast}\|D\mathcal{W}_0^+\|_{1,\alpha,\mathcal{N}_{L,f}^+}\|\mathcal{R}_0(x_1,r)-r\|_{1,\alpha,\mathcal{N}_{L,f}^+},$$ where the constant $C_2^{\ast}>0$ depends only on the data but not on $L$.
It can be verified that $$\|D\mathcal{W}_0^+\|_{1,\alpha,\mathcal{N}_{L,f}^+}\le C_3^\ast(|\beta_0|+\beta_0^2),\qquad\|\mathcal{R}_0(x_1,r)-r\|_{1,\alpha,\mathcal{N}_{L,f}^+}\le C_4^{\ast}(\delta_1+1)\sigma,$$ where the constants $C_3^{\ast},C_4^\ast>0$ depend only on the data but not on $L$.
Hence $$\|(S,B,\Lambda)-(S_0^+,B_0^+,\beta_0r)\|_{1,\alpha,\mathcal{N}_{L,f}^+}\le C^*\|(S_\mathrm{en},r\beta_\mathrm{en})-(S_0^+,\beta_0r)\|_{1,\alpha,\Gamma_\mathrm{en}^+}+C^{\ast}(|\beta_0|+\beta_0^2)(\delta_1+1)\sigma,$$ where the constant $C^\ast>0$ depends only on the data but not on $L$.
The proof of Lemma~\ref{Pro-trans} is completed.
\end{proof}

\begin{remark}
The second term in the estimate of Lemma~\ref{Pro-trans},
$$ C^*(|\beta_0|+\beta_0^2)(\delta_1+1)\sigma, $$
is a direct consequence of the nonzero swirl in the background state. Indeed,
$$ \|D\mathcal W_0^+\|_{1,\alpha,\mathcal{N}_{L,f}^+}\le C(|\beta_0|+\beta_0^2), $$
where
$$\mathcal W_0^+=(S_0^+,B_0^+,\beta_0r).$$

Thus, the perturbation of the streamline label $\mathcal R_0-r$ interacts with the nonvanishing radial derivatives of the background transported quantities and produces the additional term above. When $\beta_0=0$, the background transported quantities are constant and this term disappears.

It is important that $\beta_0$ is a background parameter independent of the perturbation size $\sigma$. Therefore, the factor $|\beta_0|+\beta_0^2$ cannot be made small merely by reducing $\sigma$. This is the reason that a fixed upper bound on $|\beta_0|$ is required in the subsequent iteration argument. Notice, however, that no condition of the form $|\beta_0|\le C\sigma$ is imposed.
\end{remark}

\textbf{2.}
To define an iteration map on the fixed set $\mathcal{P}_{\delta_1\sigma}$, we must extend the updated triple $\mathcal{W}=(S,B,\Lambda)$, which is defined on the variable domain $\mathcal{N}_{L,f}^+$, to the fixed domain $\mathcal{N}_{L,1/4}^+$.

For the unique solution $\mathcal{W}$ of the initial-value problem \eqref{Ite-3D2}, define an extension of $\mathcal{W}$ into $\mathcal{N}_{L,1/4}^+$ as follows:
\begin{equation}\label{ext-W-def}
\mathcal{E}_f(\mathcal{W})(x_1,r):=\left\{\begin{aligned}
\mathcal{W}(x_1,r)\quad\mbox{for }& r\ge f(x_1),\\
\mathcal{W}_0^+(r)+\sum_{i=1}^3c_i\left[\mathcal{W}(x_1,r_i)-\mathcal{W}_0^+(r_i)\right]\quad\mbox{for }& \frac{1}{4}< r<f(x_1),
\end{aligned}\right.
\end{equation}
for $r_i$ defined by
$$r_i(x_1,r)=f(x_1)+\frac{f(x_1)-r}{i}\quad\text{for }i=1, 2, 3.$$ Here, $c_1=6$, $c_2=-32$, $c_3=27$, which are constants determined by the system of equations $$\sum_{i=1}^3c_i\left(-\frac{1}{i}\right)^m=1, \quad m=0, 1, 2.$$
The coefficients in the definition of $\mathcal{E}_f(\mathcal{W})$ are chosen so that the extension preserves $C^{1,\alpha}$-regularity across the reference interface.
Since $R(x_1)-f(x_1)\ge \frac{7}{16}$ on $[0, L]$, $\mathcal{E}_f$ is well defined by \eqref{ext-W-def}, and it satisfies
\begin{equation}\label{W-f-est}
\|\mathcal{E}_f(\mathcal{W})-\mathcal{W}_0^+\|_{1,\alpha,\mathcal{N}_{L,1/4}^+}\le C\|\mathcal{W}-\mathcal{W}_0^+\|_{1,\alpha,\mathcal{N}_{L,f}^+}.
\end{equation}
We define an iteration mapping $\mathcal{J}:\mathcal{P}_{\delta_1\sigma}\rightarrow \left[C^{1,\alpha/2}(\overline{\mathcal{N}_{L,1/4}^+})\right]^3$ by
\begin{equation*}
\mathcal{J}(\mathcal{W}_{\ast})=\mathcal{E}_f(\mathcal{W}).
\end{equation*}
By \eqref{W-f-est} and Lemma~\ref{Pro-trans}, we have the estimate
\begin{equation*}
\|\mathcal{E}_f(\mathcal{W})-\mathcal{W}_0^+\|_{1,\alpha,\mathcal{N}_{L,1/4}^+}\le C\|\mathcal{W}-\mathcal{W}_0^+\|_{1,\alpha,\mathcal{N}_{L,f}^+}\le C_1^\star\sigma+C_2^{\star}(|\beta_0|+\beta_0^2)(\delta_1+1)\sigma
\end{equation*}
for a constant $C_1^{\star}>0$ depending only on the data but not on $L$.

\textbf{3.}
In this step, we finally choose $(\delta_1,\sigma_4)$ so that $\mathcal{J}$ has a unique fixed point in $\mathcal{P}_{\delta_1\sigma}$.

A direct computation shows that there exists a constant $\epsilon_4>0$ depending only on the data such that if
$$(\delta_1+1)\sigma\le \epsilon_4,$$
 then
\begin{equation*}
\|\varrho(S_{\ast},B_{\ast},\mathbf{q}(r,\psi,D\psi,D\varphi,\Lambda_{\ast}))\mathbf{q}(r,\psi,D\psi,D\varphi,\Lambda_{\ast})\cdot\mathbf{e}_{x_1}-\rho_0^+u_0\|_{0,\mathcal{N}_{L,f}^+}\le C_3^{\star}(\delta_1+1)\sigma
\end{equation*}
for a constant $C_3^{\star}>0$ depending only on the data but not on $L$.
If it holds that
$$\sigma\le\frac{\rho_0^+u_0}{2C_3^{\star}(\delta_1+1)} ,$$
then we obtain from the previous estimate that
\begin{equation}\label{rhou-positive}
\|\varrho(S_{\ast},B_{\ast},\mathbf{q}(r,\psi,D\psi,D\varphi,\Lambda_{\ast}))\mathbf{q}(r,\psi,D\psi,D\varphi,\Lambda_{\ast})\cdot\mathbf{e}_{x_1}-\rho_0^+u_0\|_{0,\mathcal{N}_{L,f}^+}\le \frac{\rho_0^+u_0}{2}.
\end{equation}
Also, by the boundary conditions in \eqref{S-Free-BP} for $(\varphi,\psi)$ and the definition of $\varphi_\mathrm{en}$ given in \eqref{def-varphi-en}, we have
\begin{equation}\label{partial-x-S}
\partial_r\varphi-\partial_{x_1}\psi\equiv0\quad\mbox{on}\quad \Gamma_\mathrm{en}^\epsilon\cup\Gamma_\mathrm{ex}^{L,f}.
\end{equation}
It follows from \eqref{Ite-3D2} and \eqref{rhou-positive}--\eqref{partial-x-S} that
$$\left(\partial_{x_1}\mathcal{E}_f(S),\partial_{x_1}\mathcal{E}_f(B),\partial_{x_1}\mathcal{E}_f(\Lambda)\right)\equiv \mathbf{0}\quad \mbox{on}\quad \Gamma_\mathrm{en}^\epsilon\cup \Gamma_\mathrm{ex}^{L,1/4}.$$

Choose $\delta_1, \beta_2^\ast$ and $\sigma_4^{\ast}$ as
\begin{equation}\label{Sigma4star}
\begin{aligned}
\delta_1&=2C_1^{\star}, \\\beta_2^\ast&=\min\left\{\beta^\sharp,\frac{C_1^\star}{2C_2^\star(2C_1^\star+1)},\sqrt{\frac{C_1^\star}{2C_2^\star(2C_1^\star+1)}}\right\}\\\sigma_4^{\ast}&=\min\left\{\sigma_5,\sigma_4^{\ast\ast},\frac{\epsilon_4}{\delta_1+1}, \frac{\rho_0^+u_0}{2C_3^{\star}(\delta_1+1)}\right\}
\end{aligned}
\end{equation}
with $\sigma_5$ defined in \eqref{Sigma5} and $\sigma_4^{\ast\ast}$ given in Lemma~\ref{Pro-trans}.
Under such choices of $(\delta_1,\beta_2^\ast,\sigma_4^{\ast})$, the mapping $\mathcal{J}$ maps $\mathcal{P}_{\delta_1\sigma}$ into itself whenever $|\beta_0|\le\beta_2^\ast, \sigma\le\sigma_4^{\ast}$. We emphasize that the restriction on the background swirl is independent of the perturbation size parameter $\sigma$.

The iteration set $\mathcal{P}_{\delta_1\sigma}$ given by \eqref{Ent-Ang-set} is a convex and compact subset of $[C^{1,\alpha/2}(\overline{\mathcal{N}_{L,1/4}^+})]^3$.
Suppose that a sequence $\{\mathcal{W}_{\ast}^{(k)}\}_{k=1}^{\infty}:=\{(S_{\ast}^{(k)},B_{\ast}^{(k)},\Lambda_{\ast}^{(k)})\}_{k=1}^{\infty}\subset\mathcal{P}_{\delta_1\sigma}$ converges in $C^{1,\alpha/2}(\overline{\mathcal{N}_{L,1/4}^+})$ to $\mathcal{W}_{\ast}^{(\infty)}:=(S_{\ast}^{(\infty)},B_{\ast}^{(\infty)},\Lambda_{\ast}^{(\infty)})\in\mathcal{P}_{\delta_1\sigma}$.
For each $k\in\mathbb{N}\cup\{\infty\}$, set
\begin{equation*}
\mathcal{W}^{(k)}:=\mathcal{J}(\mathcal{W}_{\ast}^{(k)}).
\end{equation*}
Let $(f^{(k)},\varphi^{(k)},\psi^{(k)})\in C^{2,\alpha}([0,L])\times [C^{2,\alpha}(\overline{\mathcal{N}_{L,f^{(k)}}^+})]^2$ be the unique solution of Problem~\ref{Prob4-Fix-S} associated with $\mathcal{W}_{\ast}=\mathcal{W}_{\ast}^{(k)}$.
Due to the uniqueness of the solution to the Problem~\ref{Prob4-Fix-S}, $\{f^{(k)}\}_{k=1}^{\infty}$ converges in $C^{2,\alpha/2}([0,L])$.
Denote its limit by $f^{(\infty)}$ and the unique solution of \eqref{Fixed-BVP} associated with $(f_{\ast},\mathcal{W}_{\ast})=(f^{(\infty)},\mathcal{W}_{\ast}^{(\infty)})$ by $(\varphi^{(\infty)},\psi^{(\infty)})$.
Define a transformation $T^{(k)}:\overline{\mathcal{N}_{L,f^{(\infty)}}^+}\rightarrow\overline{\mathcal{N}_{L,f^{(k)}}^+}$ by
\begin{equation*}
T^{(k)}(x_1,r,\theta)=\left(x_1,\frac{R(x_1)-f^{(k)}(x_1)}{R(x_1)-f^{(\infty)}(x_1)}\left(r-R(x_1)\right)+R(x_1),\theta\right),
\end{equation*}
and set
\begin{equation*}
\begin{aligned}
\mathbf{M}^{(k)}:=\varrho\left(S_{\ast}^{(k)},B_{\ast}^{(k)},\nabla\varphi^{(k)}+\mathbf{t}(r,\psi^{(k)},D\psi^{(k)},\Lambda_{\ast}^{(k)})\right)\left(\nabla\varphi^{(k)}+\frac{1}{r}\partial_r(r\psi^{(k)})\mathbf{e}_{x_1}-(\partial_{x_1}\psi^{(k)})\mathbf{e}_r\right),
\end{aligned}
\end{equation*}
where $\mathbf{t}$ and $\varrho$ are given by \eqref{def-T} and \eqref{def-H-G}, respectively.
Then $\mathbf{M}^{(k)}\circ T^{(k)}$ converges to $\mathbf{M}^{(\infty)}$ in $C^{1,\alpha/2}(\overline{\mathcal{N}_{L,f^{(\infty)}}^+}).$
By Lemma~\ref{Pro-trans}, $\mathcal{W}^{(k)}$ converges to $\mathcal{W}^{(\infty)}$ in $C^{1,\alpha/2}(\overline{\mathcal{N}_{L,1/4}^+})$.
This implies that $\mathcal{J}(\mathcal{W}_{\ast}^{(k)})$ converges to $\mathcal{J}(\mathcal{W}_{\ast}^{(\infty)})$ in $C^{1,\alpha/2}(\overline{\mathcal{N}_{L,1/4}^+})$.
Thus, $\mathcal{J}$ is a continuous map in $[C^{1,\alpha/2}(\overline{\mathcal{N}_{L,1/4}^+})]^3$.
Applying the Schauder fixed point theorem yields that $\mathcal{J}$ has a fixed point $\mathcal{W}=\mathcal{E}_f(S, B, \Lambda)\in\mathcal{P}_{\delta_1\sigma}$.
For such $\mathcal{W}$, let $(f,\varphi,\psi)$ be the unique solution of Problem~\ref{Prob4-Fix-S}, and let us set $(S, B, \Lambda):=\left.\mathcal{E}_f(S, B, \Lambda)\right|_{\mathcal{N}_{L,f}^+}$.
Then $(f,S, B, \Lambda,\varphi,\psi)$ solves Problem~\ref{Prob3-Cut} provided that $\sigma\le\sigma_4^{\ast}$.

Finally, we prove the uniqueness of a fixed point of $\mathcal{J}$.
Let $(f^{(1)},\mathcal{W}^{(1)},\varphi^{(1)},\psi^{(1)})$ and $(f^{(2)},\mathcal{W}^{(2)},\varphi^{(2)},\psi^{(2)})$ be two solutions to Problem~\ref{Prob3-Cut}, 

and suppose that each solution satisfies the estimates given in \eqref{3D-Prop-est} of Proposition~\ref{3D-Prop4.1}.
For a transformation $\mathfrak{T}:\overline{\mathcal{N}_{L,f^{(1)}}^+}\rightarrow\overline{\mathcal{N}_{L,f^{(2)}}^+}$ defined in \eqref{TT}, set
\begin{equation*}
\left\{\begin{aligned}
&\widetilde{\phi}:=\varphi^{(1)}-\left(\varphi^{(2)}\circ\mathfrak{T}\right),\quad \widetilde{\psi}:=\psi^{(1)}-\left(\psi^{(2)}\circ\mathfrak{T}\right),\\
&\widetilde{\mathcal{W}}:=\mathcal{W}^{(1)}-\left(\mathcal{W}^{(2)}\circ\mathfrak{T}\right),\quad \widetilde{f}:=f^{(1)}-f^{(2)}.
\end{aligned}\right.
\end{equation*}
A direct computation shows that there exist constants $\beta_2'>0$ and $\sigma_4'>0$ depending only on the data but not on $L$ so that if $|\beta_0|\le \beta_2'$ and $\sigma\le \sigma_4'$, then
\begin{equation}\label{W-est}
\begin{aligned}
\|\widetilde{\mathcal{W}}\|_{\alpha,\mathcal{N}_{L,f^{(1)}}^+}&+\|\partial_r\widetilde{\mathcal{W}}\|_{\alpha,\mathcal{N}_{L,f^{(1)}}^+}\\
&\le C(|\beta_0|+\beta_0^2+\sigma)\left(\|\widetilde{\phi}\|_{1,\alpha,\mathcal{N}_{L,f^{(1)}}^+}+\|\widetilde{\psi}\|_{1,\alpha,\mathcal{N}_{L,f^{(1)}}^+}+\|\widetilde{f}\|_{1,\alpha,(0,L)}\right)\\
&\le C(|\beta_0|+\beta_0^2+\sigma)\|\widetilde{f}\|_{1,\alpha,(0,L)}.
\end{aligned}
\end{equation}
Adapting the proof of Lemma~\ref{Lem-S-free} and using the estimate \eqref{W-est}, we have
\begin{equation}\label{3D-g}
\|\widetilde{f}\|_{1,\alpha,(0,L)}\le C_5^{\star}(|\beta_0|+\beta_0^2+\sigma)\|\widetilde{f}\|_{1,\alpha,(0,L)}
\end{equation}
for a constant $C_5^{\star}>0$ depending only on the data but not on $L$.
We choose $\beta_2$ and $\sigma_4$ as
\begin{equation*}\label{Sigma4}
\beta_2=\min\left\{\beta_2^\ast,\beta_2',\frac{1}{4C_5^\star},\sqrt{\frac{1}{4C_5^\star}}\right\},\quad\sigma_4=\min\left\{\sigma_4^{\ast},\sigma_4',\frac{1}{4C_5^{\star}}\right\}
\end{equation*}
with $\sigma_4^{\ast}$ defined in \eqref{Sigma4star}, so that we obtain from \eqref{3D-g} that $f^{(1)}=f^{(2)}$ for $\sigma\le\sigma_4$.
Then, by \eqref{W-est}, we have $\mathcal{W}^{(1)}=\mathcal{W}^{(2)}.$
Therefore
$$(f^{(1)},\mathcal{W}^{(1)},\varphi^{(1)},\psi^{(1)})=(f^{(2)},\mathcal{W}^{(2)},\varphi^{(2)},\psi^{(2)})$$
by Lemma~\ref{Lem-S-free}.
The proof of Proposition~\ref{3D-Prop4.1} is completed.
\qed

\section{Free boundary problem in the infinitely long nozzle $\mathcal{N}$}\label{3D-sec-ex}
\subsection{Proof of Theorem~\ref{3D-Thm-HD}}\label{5-1}
Let $\beta_2$ and $\sigma_4$ be the constants obtained in Proposition~\ref{3D-Prop4.1}. We set $$\beta_1:=\beta_2,\qquad\sigma_3:=\sigma_4.$$ Suppose that $|\beta_0|\le\beta_2$ and $\sigma\le \sigma_4$.
Then Proposition~\ref{3D-Prop4.1} yields a solution of Problem~\ref{Prob3-Cut} for every $L>L_0+10$.
For each $m\in\mathbb{N}$ with $m>L_0$, let $(f^{(m)},S^{(m)},B^{(m)},\Lambda^{(m)},\varphi^{(m)},\psi^{(m)})$ be a solution of Problem~\ref{Prob3-Cut} in $\mathcal{N}_{m+20}:=\mathcal{N}\cap\{0<x_1<m+20\}$, and suppose that the solution satisfies the estimates \eqref{3D-Prop-est} given in Proposition~\ref{3D-Prop4.1}.
Then, using the Arzel\`a-Ascoli theorem and a diagonal procedure, we can extract a subsequence, still written as $\{(f^{(m)},S^{(m)},B^{(m)},\Lambda^{(m)},\varphi^{(m)},\psi^{(m)})\}_{m\in\mathbb{N}}$ so that the subsequence converges to functions $(f^{\ast},S^{\ast},B^{\ast},\Lambda^{\ast},\varphi^{\ast},\psi^{\ast})$ in the following sense: for any $L>L_0$,
\begin{itemize}
\item[(i)] $f^{(m)}$ converges to $f^{\ast}$ in $C^2$ in $[0,L]$.
\item[(ii)] $(S^{(m)}\circ{T}^{(m)},B^{(m)}\circ T^{(m)},\Lambda^{(m)}\circ{T}^{(m)})$ converges to $(S^{\ast},B^{\ast},\Lambda^{\ast})$ in $C^1$ in $\overline{\mathcal{N}_{L,f^{\ast}}^+}$, where ${T}^{(m)}:\overline{\mathcal{N}_{m+20,f^{\ast}}^+}\rightarrow\overline{\mathcal{N}_{m+20,f^{(m)}}^+}$ is defined by
\begin{equation*}
T^{(m)}(x_1,r,\theta)=\left(x_1,\frac{R(x_1)-f^{(m)}(x_1)}{R(x_1)-f^\ast(x_1)}\left(r-R(x_1)\right)+R(x_1),\theta\right).
\end{equation*}
\item[(iii)] $\left(\varphi^{(m)}\circ{T}^{(m)},\psi^{(m)}\circ{T}^{(m)}\right)$ converges to $\left(\varphi^{\ast},\psi^{\ast}\right)$ in $C^2$ in $\overline{\mathcal{N}_{L,f^{\ast}}^+}$.
\end{itemize}
After changing variables and passing to the limit $m\rightarrow\infty$, one can show that
$(f^{\ast},S^{\ast},B^{\ast},\Lambda^{\ast},\varphi^{\ast},\psi^{\ast})$ is a solution to the free boundary problem \eqref{3D-H} with boundary conditions \eqref{g-free-cond} and \eqref{3D-BC-C}.
Furthermore, it follows from the $C^2$-convergence of $\{(f^{(m)},\varphi^{(m)},\psi^{(m)})\}_{m\in\mathbb{N}}$, $C^1$-convergence of $\{(S^{(m)},B^{(m)},\Lambda^{(m)})\}_{m\in\mathbb{N}}$, and the estimates \eqref{3D-Prop-est} given in Proposition~\ref{3D-Prop4.1} that $(f^{\ast},S^{\ast},B^{\ast},\Lambda^{\ast},\varphi^{\ast},\psi^{\ast})$ satisfy the estimates \eqref{Thm-HD-est} for a constant $C>0$ depending only on the data. Thus Theorem 3.1 holds with $\beta_1=\beta_2$ and $\sigma_3=\sigma_4$.
\qed

\subsection{Proof of Theorem~\ref{3D-MainThm}(a)}\label{5-2}
Let $\beta_1$ and $\sigma_3$ be from Theorem~\ref{3D-Thm-HD}, and suppose that $|\beta_0|\le\beta_1$ and $\sigma\le \sigma_3$. By Theorem~\ref{3D-Thm-HD}, the free boundary problem \eqref{3D-H} with \eqref{g-free-cond} and \eqref{3D-BC-C} has a solution $(g_D,S, B, \Lambda,\varphi,\psi)$ that satisfies the estimates \eqref{Thm-HD-est}.
For such a solution, we define $(\rho, \mathbf{u}, p)$ by
\begin{equation*}
\left.\begin{aligned}
&\mathbf{u}:=\left(\partial_{x_1}\varphi+\frac{1}{r}\partial_r(r\psi)\right)\mathbf{e}_{x_1}+(\partial_r\varphi-\partial_{x_1}\psi)\mathbf{e}_r+\frac{\Lambda}{r}\mathbf{e}_{\theta},\\
&\rho:=\varrho\left(S,B, \mathbf{u}\right),\quad p:=S\rho^{\gamma}\quad\mbox{in}\quad \overline{\mathcal{N}_{g_D}^+},
\end{aligned}\right.
\end{equation*}
where $\varrho$ is given by \eqref{def-H-G}.
It follows from the estimates \eqref{Thm-HD-est} given in Theorem~\ref{3D-Thm-HD} that $(g_D,\rho, \mathbf{u},p)$ satisfies the estimate \eqref{Thm2.1-uniq-est}.
Then, one can choose a positive constant $\bar\beta\in (0,\beta_1]$ and a small constant $\sigma_1\in(0,\sigma_3]$ depending only on the data such that if $|\beta_0|\le\bar\beta$ and $\sigma\le\sigma_1$, then
$(g_D, \rho, \mathbf{u}, p)$ satisfy $\rho\ge\frac{1}{2}\rho_0^+>0$ and
$$\inf_{\mathcal{N}_{g_D}^+}\left(1-\frac{u_{x_1}^2+u_r^2}{c^2}-\frac{(u_{x_1}^2+u_r^2)u_\theta^2}{4c^4}\right)>0$$
in $\overline{\mathcal{N}_{g_D}^+}$, and hence solves Problem~\ref{3D-Problem2}.
Here, $c_0^+$ is given by $c_0^+=\sqrt{\gamma p_0^+/\rho_0^+}$.
Thus Theorem~\ref{3D-MainThm}(a) holds with $\bar\beta\in(0,\beta_1]$ and $\sigma_1\in(0,\sigma_3]$.
\qed

\subsection{Proof of Theorem~\ref{3D-MainThm}(b)}\label{sec-far}

Let $\bar\beta$ and $\sigma_1$ be the constants fixed in the proof of Theorem~\ref{3D-MainThm}(a). Thus, for $$|\beta_0|\le\bar\beta,\qquad\sigma\le\sigma_1,$$ there exists a solution $(g_D,\rho, \mathbf{u},p)$ with $\mathbf{u}=u_{x_1}\mathbf{e}_{x_1}+u_r\mathbf{e}_r+u_{\theta}\mathbf{e}_{\theta}$ of Problem~\ref{3D-Problem2} satisfying the estimate \eqref{Thm2.1-uniq-est}.

Set
\begin{equation*}
\begin{aligned}
\Omega_{g_D}^+&:=\left\{(x_1,r)\in\mathbb{R}^2: x_1> 0, g_D(x_1)<r<R(x_1)\right\},\\
\Gamma_\mathrm{en}^{g_D}&:=\partial\Omega_{g_D}^+\cap\{x_1=0\},\quad
\Gamma_\mathrm{cd}^{g_D}:=\partial\Omega_{g_D}^+\cap\{r=g_D(x_1)\}.
\end{aligned}
\end{equation*}
The equation
$\partial_{x_1}(\rho u_{x_1})+\partial_r(\rho u_r)+\frac{\rho u_r}{r}=0$ in $\Omega_{g_D}^+$, stated in \eqref{3D-ang}, can be rewritten as
\begin{equation*}\label{conti-eq}
\partial_{x_1}(r\rho u_{x_1})+\partial_r(r\rho u_r)=0\quad\mbox{in}\quad\Omega_{g_D}^+.
\end{equation*}
Using this equation, one can directly verify that the function $\mathfrak{h}$ given by
\begin{equation*}
\mathfrak{h}(x_1,r):=\int_{R(x_1)}^rt\rho u_{x_1}(x_1,t)dt\quad\mbox{for}\quad (x_1,r)\in\overline{\Omega_{g_D}^+}
\end{equation*}
satisfies
\begin{equation}\label{3D-st-f}
\partial_{x_1}\mathfrak{h}=-r\rho u_r,\quad\partial_r\mathfrak{h}=r\rho u_{x_1}.
\end{equation}
Set
$$\omega:=\partial_{x_1}\mathfrak{h}.$$

\textbf{Claim:} $$\int_L^{L+1}\int_{g_D(x_1)}^{R(x_1)}|\nabla\omega|^2drdx_1\rightarrow0\quad\mbox{as}\quad L\rightarrow\infty.$$

Assume that the claim is true. Since $\omega\in C^{1,\alpha}(\overline{\mathcal{N}^+_{g_D}})$, we have
\begin{equation}\label{3D-omega-lim}
\|\nabla\omega(x_1,\cdot)\|_{C^0(\overline{\Omega_{g_D}^+\cap\{x_1> L\}})}\rightarrow 0\quad\mbox{as}\quad L\rightarrow\infty.
\end{equation}
By \eqref{3D-omega-lim} and the compatibility condition $\omega\equiv 0$ on $\{(x_1,r):x_1>L_0,r=R(x_1)\}$, we have
\begin{equation}\label{omega0}
\|\omega\|_{C^0(\overline{\Omega_{g_D}^+\cap\{x_1> L\}})}\rightarrow 0\quad\mbox{as}\quad L\rightarrow\infty.
\end{equation}
Since $\rho>\rho_0^+/2$ in $\Omega_{g_D}^+$ and $\omega=\partial_{x_1}\mathfrak{h}=-r\rho u_r$, \eqref{omega0} implies that
\begin{equation}\label{ur0}
\|ru_r\|_{C^0(\overline{\Omega_{g_D}^+\cap\{x_1> L\}})}\rightarrow 0\quad\mbox{as}\quad L\rightarrow\infty.
\end{equation}
By \eqref{3D-omega-lim} and \eqref{ur0}, we have
\begin{equation*}
\| ru_r\|_{C^1(\overline{\Omega_{g_D}^+\cap\{x_1> L\}})}\rightarrow 0\quad\mbox{as}\quad L\rightarrow\infty,
\end{equation*}
from which
\begin{eqnarray}
&\nonumber&\|g_D'\|_{C^1({\{x_1\ge L\}})}\rightarrow 0,\\
&\label{nabla-ur}&\|u_r\|_{C^1(\overline{\mathcal{N}_{g_D}^+\cap\{x_1> L\}})}\rightarrow 0\quad\mbox{as}\quad L\rightarrow\infty.
\end{eqnarray}
It follows from the equation in \eqref{3D-ang} and \eqref{nabla-ur} that
\begin{equation*}
\left\|\partial_r p-\frac{\rho u_{\theta}^2}{r}\right\|_{C^0(\overline{\mathcal{N}_{g_D}^+\cap\{x_1> L\}})}\rightarrow 0\quad\mbox{as}\quad L\rightarrow\infty.
\end{equation*}
Thus, once the claim is established, the proof of Theorem~\ref{3D-MainThm}(b) is complete. It remains only to prove the claim.

\begin{proof}[Verification of Claim.]
By \eqref{def-w}--\eqref{def-W}, the entropy $S(=p/\rho^{\gamma})$, Bernoulli function $B(=\frac{1}{2}|\mathbf{u}|^2+\frac{\gamma}{\gamma-1}S\rho^{\gamma-1})$ and angular momentum density $\Lambda(=ru_{\theta})$ are represented as
\begin{equation}\label{S-Lambda}
\begin{aligned}
S(x_1,r)=S_\mathrm{en}\circ\mathcal{G}^{-1}(\mathfrak{h}(x_1,r))&=:{S}(\mathfrak{h}(x_1,r)),\\
B(x_1,r)=B_0^+\circ\mathcal{G}^{-1}(\mathfrak{h}(x_1,r))&=:{B}(\mathfrak{h}(x_1,r))\quad\mbox{for } (x_1,r)\in\overline{\Omega_{g_D}^+},\\
\Lambda(x_1,r)=\Lambda_\mathrm{en}\circ\mathcal{G}^{-1}(\mathfrak{h}(x_1,r))&=:{\Lambda}(\mathfrak{h}(x_1,r))
\end{aligned}
\end{equation}
where $\mathcal{G}$ is given by \eqref{def-G0} associated with $w=\mathfrak{h}$ and $\Lambda_\mathrm{en}(r):=r\beta_\mathrm{en}(r)$ for $r\in[1/2,1]$.
Since $S_\mathrm{en}$, $B_0^+$, $\Lambda_\mathrm{en}$, and $\mathcal{G}^{-1}$
are differentiable, $S$, $B$ and $\Lambda$ are differentiable functions of $\mathfrak{h}$.
Set
$$\mathfrak{S}(\mathfrak{h}):=\frac{\gamma}{\gamma-1}{S}(\mathfrak{h}).$$
Then, by the definition of the Bernoulli function \eqref{Ber-inv}, we have
\begin{equation}\label{3D-Ber}
B(\mathfrak{h})r^2\rho^2=\frac{1}{2}\left(|\nabla\mathfrak{h}|^2+\Lambda(\mathfrak{h})^2\rho^2\right)+r^2\mathfrak{S}(\mathfrak{h})\rho^{\gamma+1}\quad\mbox{in}\quad \overline{\Omega_{g_D}^+},
\end{equation}
for $$\nabla:=(\partial_{x_1},\partial_r).$$
By differentiating the equation \eqref{3D-Ber} with respect to $x_1$ and $r$, we have
\begin{equation}\label{3D-par-rho}
\begin{aligned}
&\partial_{x_1}\rho=-\frac{(\partial_{x_1}\mathfrak{h})(\partial_{x_1 x_1}\mathfrak{h}+\Lambda\Lambda'\rho^2+r^2\mathfrak{S}'\rho^{\gamma+1}-r^2B'\rho^2)+(\partial_r\mathfrak{h})(\partial_{r{x_1}}\mathfrak{h})}{r^2(\gamma+1)\mathfrak{S}\rho^{\gamma}-2r^2B\rho+\Lambda^2\rho},\\
&\partial_r\rho=-\frac{(\partial_{x_1}\mathfrak{h})(\partial_{x_1 r}\mathfrak{h}-\partial_{x_1}\mathfrak{h}/r)+(\partial_r\mathfrak{h})(\partial_{rr}\mathfrak{h}+\Lambda\Lambda'\rho^2+r^2\mathfrak{S}'\rho^{\gamma+1}-r^2B'\rho^2-\partial_r\mathfrak{h}/r)-\Lambda^2\rho^2/r}{r^2(\gamma+1)\mathfrak{S}\rho^{\gamma}-2r^2B\rho+\Lambda^2\rho},
\end{aligned}
\end{equation}
where $'$ denotes the derivative with respect to $\mathfrak{h}$.
Using \eqref{3D-st-f}--\eqref{3D-par-rho}, the equation
\begin{equation}\label{ES-Far-eq}
\rho(u_{x_1}\partial_{x_1}+u_r\partial_r)u_r-\frac{\rho u_{\theta}^2}{r}+\partial_rp=0\quad\mbox{in}\quad\Omega_{g_D}^+
\end{equation}
in \eqref{3D-ang} can be rewritten as
\begin{equation}\label{3D-ES-st}
-\left(\frac{\partial_r\mathfrak{h}}{r}\right)\nabla\cdot\left(\frac{\nabla\mathfrak{h}}{r\rho}\right)-\frac{(\partial_r\mathfrak{h})\mathfrak{S}'\rho^{\gamma}}{\gamma}-\frac{(\partial_r\mathfrak{h})\Lambda\Lambda'\rho}{r^2}+(\partial_r\mathfrak{h})\rho B'=0\quad\mbox{in}\quad\Omega_{g_D}^+.
\end{equation}
We multiply \eqref{3D-ES-st} by $r/(\partial_r\mathfrak{h})$ to get
\begin{equation}\label{3D-Stream}
\nabla\cdot\left(\frac{\nabla\mathfrak{h}}{r\rho}\right)=-\frac{r}{\gamma}\mathfrak{S}'\rho^{\gamma}-\frac{\Lambda\Lambda'\rho}{r}+r\rho B'\quad\mbox{in}\quad\Omega_{g_D}^+.
\end{equation}
Differentiate \eqref{3D-Stream} with respect to $x_1$ to get the following equation for $\omega$:
\begin{equation}\label{3D-St-Diff}
\partial_i\left(\frac{\mathfrak{q}_{ij}}{r\rho^2}\partial_j\omega\right)+\partial_i\left(\frac{\mathfrak{q}_1\partial_i\mathfrak{h}}{r\rho^2}\omega\right)
=\mathfrak{q}_2\omega+\mathfrak{q}_3(\partial_i\mathfrak{h})(\partial_i\omega)\quad\mbox{in}\quad \Omega_{g_D}^+,
\end{equation}
where
\begin{equation}\label{def-q12}
\begin{aligned}
&\mathfrak{O}:=r^2(\gamma+1)\mathfrak{S}\rho^{\gamma}-2r^2B\rho+\Lambda^2\rho,\\
&\mathfrak{q}_{ij}:=\rho\delta_{ij}+\frac{(\partial_i\mathfrak{h})(\partial_j\mathfrak{h})}{\mathfrak{O}},\\
&\mathfrak{q}_1:=\frac{\Lambda\Lambda'\rho^2+r^2\mathfrak{S}'\rho^{\gamma+1}-r^2 B'\rho^2}{\mathfrak{O}},\\
&\mathfrak{q}_2
:=-\frac{r}{\gamma}\mathfrak{S}''\rho^{\gamma}-\frac{(\Lambda')^2\rho}{r}-\frac{\Lambda\Lambda''\rho}{r}+r\rho B''+\frac{(\Lambda\Lambda'\rho^2+r^2\mathfrak{S}'\rho^{\gamma+1}-r^2 B' \rho^2)^2}{r\rho^2\mathfrak{O}},\\
&\mathfrak{q}_3:=\frac{1}{\mathfrak{O}}\left(r\mathfrak{S}'\rho^{\gamma-1}+\frac{\Lambda\Lambda'}{r}-rB'\right).
\end{aligned}
\end{equation}

Note that $\mathbf{u}$ is represented by \eqref{3D-u}, for $(\varphi, \psi, S,B,\Lambda)$ solving the equations \eqref{3D-H}. Similarly to \eqref{3D-psi-BC}, we rewrite the second equation in \eqref{3D-H} as
\begin{equation*}
-\left(\partial_{x_1 x_1}+\frac{1}{r}\partial_r(r\partial_r)-\frac{1}{r^2}\right)\psi=\frac{1}{\mathbf{u}\cdot\mathbf{e}_{x_1}}\left(-\partial_r B+\frac{\varrho^{\gamma-1}(S,B,\mathbf{u})}{\gamma-1}\partial_rS+\frac{\Lambda}{r^2}\partial_r\Lambda\right)
\quad\mbox{in}\quad\Omega_{g_D}^+.
\end{equation*}
By Theorem~\ref{3D-MainThm}(a) and Lemma~\ref{Pro-trans}, the right-hand side of this equation is $C^{1,\alpha}$ in $\Omega_{g_D}^+$. Therefore we have $\psi\in C^{3,\alpha}(\Omega_{g_D}^+)$. Next, we regard the first equation in \eqref{3D-H} as a second order quasilinear equation for $\varphi$. By Theorem~\ref{3D-MainThm}(a), this equation is uniformly elliptic. Since $\varphi$ is $C^{2,\alpha}$ in $\Omega_{g_D}^{+}$, and $\psi\in C^{3,\alpha}(\Omega_{g_D}^+)$, we obtain that $\varphi$ is $C^{3,\alpha}$ in $\Omega_{g_D}^+$. 
This implies that $\mathfrak{h}\in C^{3,\alpha}(\Omega_{g_D}^+)$, thus the equation \eqref{3D-St-Diff} is well-defined.

By the boundary conditions \eqref{Prob2-BC-ent}, $\omega$ satisfies
\begin{equation}\label{omega-BC-0}
\omega=-r\rho u_r^\mathrm{en}\quad\mbox{on}\quad\Gamma_\mathrm{en}^{g_D},\quad \omega=-R'(x_1)\partial_r\mathfrak{h}\quad\mbox{on}\quad\Gamma_\mathrm{w}^+.
\end{equation}

Next, we compute a conormal boundary condition for \eqref{3D-St-Diff} on $\Gamma_\mathrm{cd}^{g_D}$.

We consider the expression
\begin{equation}\label{3D-fxr2}
(\partial_{x_1}\mathfrak{h})^2+(\partial_r\mathfrak{h})^2=\mathcal{C}_1(g_D(x_1))^2-\mathcal{C}_2\quad\mbox{on}\quad\Gamma_\mathrm{cd}^{g_D}=\partial\Omega_{g_D}^+\cap\{r=g_D(x_1)\}
\end{equation}
for
\begin{equation}\label{def-C12}
\mathcal{C}_1:=\frac{(\partial_{x_1}\mathfrak{h})^2+(\partial_r\mathfrak{h})^2+\Lambda^2\rho^2}{r^2}(x_1,g_D(x_1)),\quad\mathcal{C}_2:=\Lambda^2 \rho^2(x_1,g_D(x_1)).\end{equation}
Since we have
\begin{equation}\label{S-cont}
S=S_\mathrm{en}\left(\frac 12\right),\quad B=B_0^+\left(\frac 12\right),\quad\Lambda=\Lambda_\mathrm{en}\left(\frac 12\right),\quad p=p_0^-\quad\mbox{on}\quad\Gamma_\mathrm{cd}^{g_D},
\end{equation}
we obtain
\begin{equation}\label{rho-cont}
\rho=\left(\frac{p_0^-}{S_\mathrm{en}(\frac 12)}\right)^{1/\gamma},
\end{equation}
from which it follows that $\mathcal{C}_2$ in \eqref{def-C12} is given by
\begin{equation*}
\mathcal{C}_2=\Lambda_\mathrm{en}^2(\frac 12)(p_0^-)^{2/\gamma}S_\mathrm{en}^{-2/\gamma}(\frac 12).
\end{equation*}
A direct computation using \eqref{Ber-inv}, \eqref{3D-st-f}, and \eqref{S-cont}--\eqref{rho-cont} yields that
\begin{equation*}
\begin{aligned}
\mathcal{C}_1
&=2\left(B_0^+(\frac{1}{2})-\frac{\gamma}{\gamma-1}(p_0^-)^{1-1/\gamma}S_\mathrm{en}^{1/\gamma}(\frac 12)\right)(p_0^-)^{2/\gamma}S_\mathrm{en}^{-2/\gamma}(\frac 12).
\end{aligned}
\end{equation*}
Differentiating \eqref{3D-fxr2} in the tangential direction along $\Gamma_\mathrm{cd}^{g_D}$,
we have
\begin{equation*}
(\partial_{x_1}\mathfrak{h})\left(\partial_{x_1 x_1}\mathfrak{h}+g_D'(x_1)\partial_{x_1 r}\mathfrak{h}\right)+(\partial_r\mathfrak{h})\left(\partial_{r{x_1}}\mathfrak{h}+g_D'(x_1)\partial_{rr}\mathfrak{h}\right)=\mathcal{C}_1g_D(x_1)g_D'(x_1)\quad\mbox{on}\quad\Gamma_\mathrm{cd}^{g_D}.
\end{equation*}
We solve this expression for $\partial_{x_1 r}\mathfrak{h}$ to get
\begin{equation}\label{3D-FXR}
\partial_{x_1 r}\mathfrak{h}
=-\left(\frac{\mathcal{C}_1g_D(x_1)}{\partial_r\mathfrak{h}}+\partial_{x_1 x_1}\mathfrak{h}-\partial_{rr}\mathfrak{h}\right)\frac{\omega}{(\partial_{x_1}\mathfrak{h})g_D'(x_1)+(\partial_r\mathfrak{h})}\quad\mbox{on}\quad\Gamma_\mathrm{cd}^{g_D}.
\end{equation}
Substituting the expression for $\mathcal{C}_1$ in \eqref{def-C12} into \eqref{3D-FXR}, we have
\begin{equation}\label{q22term}
\partial_r\omega=\partial_{x_1 r}\mathfrak{h}
=\left(\frac{-\partial_{x_1 x_1}\mathfrak{h}+g_D(x_1)\mathfrak{E}}{(\partial_{x_1}\mathfrak{h})g_D'(x_1)+(\partial_r\mathfrak{h})}\right)\omega\quad\mbox{on}\quad\Gamma_\mathrm{cd}^{g_D}
\end{equation}
for
\begin{equation*}
\mathfrak{E}:=\frac{1}{\partial_r\mathfrak{h}}\left\{-\left(\frac{\partial_{x_1}\mathfrak{h}}{r}\right)^2-\left(\frac{\Lambda\rho}{r}\right)^2\right\}+\partial_r\left(\frac{\partial_r\mathfrak{h}}{r}\right).
\end{equation*}
By the definition of $\mathfrak{q}_{21}$ in \eqref{def-q12}, we also have
\begin{equation}\label{q21term}
\mathfrak{q}_{21}=\frac{(\partial_{x_1}\mathfrak{h})(\partial_r\mathfrak{h})}{\mathfrak{O}}=\frac{(\partial_r\mathfrak{h})\omega}{\mathfrak{O}}.
\end{equation}
Finally, a direct computation using \eqref{q22term}--\eqref{q21term} yields the following conormal boundary condition for \eqref{3D-St-Diff} on $\Gamma_\mathrm{cd}^{g_D}$:
\begin{equation}\label{3D-omega-BC}
\left(\frac{\mathfrak{q}_{1j}}{r\rho^2}\partial_j\omega,\frac{\mathfrak{q}_{2j}}{r\rho^2}\partial_j\omega\right)\cdot\mathbf{n}_{g_D}=\widetilde{\mu}\omega\quad\mbox{on}\quad\Gamma_\mathrm{cd}^{g_D}
\end{equation}
for $\widetilde{\mu}$ defined by
\begin{equation*}\label{3D-omega-GammaD}
\begin{aligned}
\widetilde{\mu}:=&\frac{\mathfrak{q}_{11}\partial_{x_1 x_1}\mathfrak{h}+\mathfrak{q}_{12}\partial_{x_1 r}\mathfrak{h}}{r\rho^2(\partial_r\mathfrak{h})\sqrt{1+|g_D'|^2}}+\frac{1}{r\rho^2\sqrt{1+|g_D'|^2}}\left(\frac{(\partial_r\mathfrak{h})(\partial_{x_1 x_1}\mathfrak{h})}{\mathfrak{O}}\right)\\
&+\frac{\mathfrak{q}_{22}}{r\rho^2\sqrt{1+|g_D'|^2}}
\left(\frac{-\partial_{x_1 x_1}\mathfrak{h}+g_D\mathfrak{E}}{(\partial_{x_1}\mathfrak{h})g_D'+(\partial_r\mathfrak{h})}\right),
\end{aligned}
\end{equation*}
where we represent $\mathbf{n}_{g_D}$ as
\begin{equation*}
\mathbf{n}_{g_D}=\frac{1}{\sqrt{1+|g_D'(x_1)|^2}}\left(\frac{\omega}{\partial_r\mathfrak{h}},1\right).
\end{equation*}

Fix a constant $L>L_0+1$ and let $\eta$ be a $C^{\infty}$ function satisfying
$$\eta=1\quad\mbox{for}\quad L_0+1<x_1<L,\quad \eta=0\quad\mbox{for}\quad x_1\notin[L_0,L+1],\quad\mbox{and}\quad|\eta'(x_1)|\le 2.$$
Multiply \eqref{3D-St-Diff} by $\eta^2\omega$ and integrate over the domain $\Omega_{g_D}^+$ to get
\begin{equation}\label{3D-sum-omega}
\iint_{\Omega_{g_D}^+}\frac{\eta^2|\nabla\omega|^2}{r\rho} drdx_1=\sum_{i=1}^6 I_i+\sum_{i=1}^2 B_i
\end{equation}
for
\begin{equation*}
\begin{aligned}
&I_1:=-\iint_{\Omega_{g_D}^+}\frac{|\nabla\mathfrak{h}\cdot\nabla\omega|^2\eta^2}{r\rho^2\mathfrak{O}}drdx_1,\\
&I_2:=-2\iint_{\Omega_{g_D}^+}\left(\frac{\mathfrak{q}_{ij}}{r\rho^2}\partial_j\omega\right)\eta(\partial_i\eta)\omega drdx_1,\\
&I_3:=-2\iint_{\Omega_{g_D}^+}\frac{1}{\mathfrak{O}}\left(r\mathfrak{S}'\rho^{\gamma-1}+\frac{\Lambda\Lambda'}{r}-rB'\right)(\nabla\mathfrak{h}\cdot\nabla\eta)\eta\omega^2drdx_1,\\
&I_4:=-2\iint_{\Omega_{g_D}^+}\frac{1}{\mathfrak{O}}\left(r\mathfrak{S}'\rho^{\gamma-1}+\frac{\Lambda\Lambda'}{r}-rB'\right)(\partial_i\mathfrak{h})(\partial_i\omega)\eta^2\omega drdx_1,\\
&I_5:=\iint_{\Omega_{g_D}^+}\left(\frac{r}{\gamma}\mathfrak{S}''\rho^{\gamma}+\frac{(\Lambda')^2\rho}{r}+\frac{\Lambda\Lambda''\rho}{r}-rB''\rho\right)\eta^2\omega^2 drdx_1,\\
&I_6:=-\iint_{\Omega_{g_D}^+}\frac{(r^2\mathfrak{S}'\rho^{\gamma+1}+\Lambda\Lambda'\rho^2-r^2B'\rho^2)^2}{r\rho^2\mathfrak{O}}\eta^2\omega^2 drdx_1,\\
&B_1:=\int_{\Gamma_\mathrm{w}^{g_D}\cup\Gamma_\mathrm{cd}^{g_D}\cup\Gamma_\mathrm{en}^{g_D}}\left(\frac{\mathfrak{q}_{ij}}{r\rho^2}\partial_j\omega\right)\eta^2\omega\cdot\mathbf{n}_\mathrm{out}ds,\\
&B_2:=\int_{\Gamma_\mathrm{w}^{g_D}\cup\Gamma_\mathrm{cd}^{g_D}\cup\Gamma_\mathrm{en}^{g_D}}\frac{\partial_i\mathfrak{h}}{\mathfrak{O}}\left(r\mathfrak{S}'\rho^{\gamma-1}+\frac{\Lambda\Lambda'}{r}-rB'\right)\eta^2\omega^2 \cdot\mathbf{n}_\mathrm{out}ds.
\end{aligned}
\end{equation*}
We will show that
\begin{equation}\label{3D-omega-claim}
\left\{\begin{aligned}
&I_1+I_4+I_6\le 0,\\
&|I_2|\le C\left(\int_{L_0}^{L_0+1}\int_{g_D(x_1)}^{R(x_1)}\left(1+\frac{1}{r^2}\right)|\nabla\omega|^2drdx_1+\int_L^{L+1}\int_{g_D(x_1)}^{R(x_1)}\left(1+\frac{1}{r^2}\right)|\nabla\omega|^2drdx_1\right),\\
&|I_3|\le C\sigma\left(\int_{L_0}^{L_0+1}\int_{g_D(x_1)}^{R(x_1)}\frac{|\nabla\omega|^2}{r}drdx_1+\int_L^{L+1}\int_{g_D(x_1)}^{R(x_1)}\frac{|\nabla\omega|^2}{r}drdx_1\right),\\
&|I_5|\le C(\beta_0^2+\sigma)\int_{L_0}^{L+1}\int_{g_D(x_1)}^{R(x_1)}\frac{|\nabla\omega|^2}{r}drdx_1,\\
&|B_1|\le C(\beta_0^2+\sigma)\int_{L_0}^{L+1}\int_{g_D(x_1)}^{R(x_1)}{\frac{|\nabla\omega|^2}{r}}drdx_1,\\
&|B_2|\le C(\beta_0^2+\sigma)\int_{L_0}^{L+1}\int_{g_D(x_1)}^{R(x_1)}{|\nabla\omega|^2}drdx_1,
\end{aligned}\right.
\end{equation}
where $C>0$ is a constant depending only on the data.
From now on, the constant $C$ depends only on the data, which may vary from line to line.

First, by H\"older's inequality, we have
\begin{equation*}
\begin{aligned}
I_4&\le 2\left(\iint_{\Omega_{g_D}^+}\frac{|\nabla\mathfrak{h}\cdot\nabla\omega|^2\eta^2}{r\rho^2\mathfrak{O}}drdx_1\right)^{1/2}
\left(\iint_{\Omega_{g_D}^+}\frac{\left(r^2\mathfrak{S}'\rho^{\gamma+1}+{\Lambda\Lambda'\rho^2}-r^2B'\rho^2\right)^2}{r\rho^2\mathfrak{O}}\eta^2\omega^2drdx_1\right)^{1/2}\\
&=2\sqrt{|I_1||I_6|},
\end{aligned}
\end{equation*}
from which we obtain that
$$I_1+I_4+I_6\le -|I_1|+2\sqrt{|I_1||I_6|}-|I_6|\le 0.$$

Before we prove the remaining estimates in \eqref{3D-omega-claim}, we compute estimates for $(\rho, \mathfrak{O}, \mathfrak{S}',\mathfrak{S}'',B',B'',\Lambda',\Lambda'')$.
A straightforward computation using the estimate \eqref{Thm2.1-uniq-est} given in Theorem~\ref{3D-MainThm}(a) shows that there exists a constant $\sigma_{\star}\in(0,\sigma_1]$ depending only on the data such that if $\sigma\le \sigma_{\star}$, then we have
\begin{equation}\label{3D-far-Lem}
|\rho-\rho_0^+|\le \frac{\rho_0^+}{2}\quad\mbox{and}\quad|\mathfrak{V}_0-\mathfrak{V}|\le \frac{\mathfrak{V}_0}{2}\quad\mbox{in}\quad\overline{\Omega_{g_D}^+}
\end{equation}
for
\begin{equation*}
\mathfrak{V}_0:=c_0^2-u_0^2=\frac{\gamma p_0^+}{\rho_0^+}-u_0^2,\quad \mathfrak{V}:=c^2-(\mathbf{u}\cdot\mathbf{e}_{x_1})^2-(\mathbf{u}\cdot\mathbf{e}_r)^2.
\end{equation*}
By \eqref{3D-far-Lem}, it holds that
\begin{equation}\label{3D-O-lower}
\begin{aligned}
\mathfrak{O}&=r^2(\gamma+1)\mathfrak{S}\rho^{\gamma}-2r^2B\rho+\Lambda^2\rho
={r^2}{\rho}\left(\frac{\gamma p}{\rho}-|\mathbf{u}|^2\right)+\Lambda^2\rho\\
&=r^2\rho\left(c^2-|\mathbf{u}|^2+\left(\frac{\Lambda}{r}\right)^2\right)
=r^2\rho\mathfrak{V}
\ge \frac{r^2\rho_0^+ \mathfrak{V}_0}{4}.
\end{aligned}
\end{equation}
Using the equations in \eqref{3D-ang} and the definition of $\mathfrak{h}$, it can be checked that
\begin{equation}\label{rg-rel}
\int_{\mathcal{G}^{-1}(\mathfrak{h}(x_1,r))}^1s\rho u_{x_1}(0,s)ds
=\int_r^{R(x_1)} s\rho u_{x_1}(x_1,s)ds \quad\text{in $\Omega_{g_D}^+$},
\end{equation}
where $\mathcal{G}$ is given in \eqref{S-Lambda}.
One can also check that there exists a constant $\sigma_{\star\star}\in(0,\sigma_{\star}]$ depending only on the data such that if $\sigma\le\sigma_{\star\star}$, then
\begin{equation}\label{rhou-G}
|\rho u_{x_1}-\rho_0^+u_0|\le \frac{\rho_0^+u_0}{2},
\end{equation}
and it follows from \eqref{rg-rel}--\eqref{rhou-G} that
\begin{equation*}\label{G-inv-est}
0< \frac{1}{\sqrt{3}}\le \frac{\mathcal{G}^{-1}(\mathfrak{h}(x_1,r))}{r}\le \sqrt{3}\quad \text{in $\Omega_{g_D}^+$},
\end{equation*}
then we get
\begin{equation}\label{S-Lam-est}
\begin{aligned}
|\mathfrak{S}'(\mathfrak{h})|+|\mathfrak{S}''(\mathfrak{h})|+|B'(\mathfrak{h})|+|B''(\mathfrak{h})|&\le C(\beta_0^2+\sigma),\\|\Lambda'(\mathfrak{h})|+|\Lambda''(\mathfrak{h})|&\le C(|\beta_0|+\sigma)
\end{aligned}
\end{equation}
in $\Omega_{g_D}^+$.
We are now ready to estimate $I_2$.
Since $\frac{\omega(\partial_j\omega)}{r}\le C\left(\omega^2+\frac{|\nabla\omega|^2}{r^2}\right)$ and $\rho\ge\frac{\rho_0^+}{2}$ in $\Omega_{g_D}^+$, we have
\begin{equation}\label{3D-I2-est}
|I_2|\le C\left(\int_{L_0}^{L_0+1}\int_{g_D(x_1)}^{R(x_1)}\left(\omega^2+\frac{|\nabla\omega|^2}{r^2}\right)drdx_1+\int_L^{L+1}\int_{g_D(x_1)}^{R(x_1)}\left(\omega^2+\frac{|\nabla\omega|^2}{r^2}\right)drdx_1\right).
\end{equation}
By the boundary condition $\omega\equiv 0$ on $\{(x_1,r):x_1>L_0,r=R(x_1)\}$ stated in \eqref{omega-BC-0}, we have
\begin{equation*}\label{3D-omega-Poin}
\omega(x_1,t)=\int_{R(x_1)}^{t}\partial_r\omega(x_1,r)dr\quad\mbox{for}\quad (x_1,t)\in\overline{\Omega_{g_D}^+\cap\{x_1\ge L_0\}}.
\end{equation*}
By H\"older's inequality, we have the following estimates:
\begin{equation}\label{3D-omega22}
\left\{\begin{aligned}
&\omega^2(x_1,t)\le C (t^2-R(x_1)^2)\int^{t}_{R(x_1)}\frac{(\partial_r\omega)^2(x_1,r)}{r}dr,\\
&\omega^2(x_1,t)\le (t-R(x_1))\int^t_{R(x_1)}(\partial_r\omega)^2(x_1,r) dr\le R(x_1)\int_{g_D(x_1)}^{R(x_1)}|\nabla\omega|^2 dr\quad\mbox{for }(x_1,t)\in\overline{\Omega_{g_D}^+\cap\{x_1\ge L_0\}}.
\end{aligned}\right.
\end{equation}
Substituting the second estimate of \eqref{3D-omega22} into \eqref{3D-I2-est} yields
$$|I_2|\le C\left(\int_{L_0}^{L_0+1}\int_{g_D(x_1)}^{R(x_1)}\left(1+\frac{1}{r^2}\right)|\nabla\omega|^2drdx_1+\int_L^{L+1}\int_{g_D(x_1)}^{R(x_1)}\left(1+\frac{1}{r^2}\right)|\nabla\omega|^2drdx_1\right).$$

It follows from \eqref{3D-far-Lem}--\eqref{S-Lam-est} that
\begin{align}
\label{I3-est}|I_3|&\le C\sigma\left(\int_{L_0}^{L_0+1}\int_{g_D(x_1)}^{R(x_1)}\frac{\omega^2}{r^2}drdx_1+\int_L^{L+1}\int_{g_D(x_1)}^{R(x_1)}\frac{\omega^2}{r^2}drdx_1\right),\\
\label{B2-est}|B_2|&\le C(\beta_0^2+\sigma)\int_{L_0}^{L+1}\omega^2(x_1,g_D(x_1))dx_1.
\end{align}
Substituting the first estimate of \eqref{3D-omega22} into \eqref{I3-est} gives
\begin{equation*}
|I_3|\le C\sigma\left(\int_{L_0}^{L_0+1}\int_{g_D(x_1)}^{R(x_1)}\frac{|\nabla\omega|^2}{r^2}drdx_1+\int_L^{L+1}\int_{g_D(x_1)}^{R(x_1)}\frac{|\nabla\omega|^2}{r^2}drdx_1\right).
\end{equation*}
Similarly, substituting the second estimate of \eqref{3D-omega22} into \eqref{B2-est} gives
\begin{equation*}
|B_2|\le C(\beta_0^2+\sigma)\int_{L_0}^{L+1}\int_{g_D(x_1)}^{R(x_1)}|\nabla\omega|^2drdx_1.
\end{equation*}
It follows from \eqref{3D-omega-BC} and \eqref{3D-far-Lem}--\eqref{S-Lam-est} that
\begin{equation}\label{3D_B1}
\begin{aligned}
|B_1|\le C\int_{L_0}^{L+1}|\widetilde{\mu}|\omega^2(x_1,g_D(x_1))dx_1
\le C(\beta_0^2+\sigma)\int_{L_0}^{L+1}\omega^2(x_1,g_D(x_1))dx_1.
\end{aligned}
\end{equation}
Substituting the first estimate of \eqref{3D-omega22} into \eqref{3D_B1} gives
\begin{equation*}
|B_1|\le C(\beta_0^2+\sigma)\int_{L_0}^{L+1}\int_{g_D(x_1)}^{R(x_1)}{\frac{|\nabla\omega|^2}{r}}drdx_1.
\end{equation*}
Also, we obtain from \eqref{3D-far-Lem}, \eqref{S-Lam-est}, and the first estimate of \eqref{3D-omega22} that
$$|I_5|\le C(\beta_0^2+\sigma)\int_{L_0}^{L+1}\int_{g_D(x_1)}^{R(x_1)}\frac{\omega^2}{r^2}drdx_1\le C(\beta_0^2+\sigma)\int_{L_0}^{L+1}\int_{g_D(x_1)}^{R(x_1)}\frac{|\nabla\omega|^2}{r}drdx_1.$$

Now the estimates in \eqref{3D-omega-claim} are all verified.

From \eqref{3D-sum-omega}--\eqref{3D-omega-claim}, we have
\begin{equation*}
\begin{aligned}
\int_{L_0+1}^L&\int_{g_D(x_1)}^{R(x_1)}\frac{|\nabla\omega|^2}{r} drdx_1\\
&\le C^{(\sharp)}(\beta_0^2+\sigma)\int_{L_0}^{L+1}\int_{g_D(x_1)}^{R(x_1)}\frac{|\nabla\omega|^2}{r} drdx_1\\&+C\left(\int_L^{L+1}\int_{g_D(x_1)}^{R(x_1)}\left(1+\frac{1}{r^2}\right)|\nabla\omega|^2drdx_1+\int_{L_0}^{L_0+1}\int_{g_D(x_1)}^{R(x_1)}\left(1+\frac{1}{r^2}\right)|\nabla\omega|^2drdx_1\right),
\end{aligned}
\end{equation*}
where the constant $C^{(\sharp)}>0$ depends only on the data.
If it holds that
$$\beta_0^2\le\frac{1}{4C^{(\sharp)}},\quad{\sigma}\le\frac{1}{2C^{(\sharp)}}, $$
then we obtain from the previous estimate that
\begin{equation*}\label{3D-urr}
\begin{aligned}
 \int_{L_0+1}^L\int_{g_D(x_1)}^{R(x_1)}\frac{|\nabla\omega|^2}{r} drdx_1
&\le C\left(\int_L^{L+1}\int_{g_D(x_1)}^{R(x_1)}\left(1+\frac{1}{r^2}\right)|\nabla\omega|^2drdx_1+\int_{L_0}^{L_0+1}\int_{g_D(x_1)}^{R(x_1)}\left(1+\frac{1}{r^2}\right)|\nabla\omega|^2drdx_1\right).
\end{aligned}
\end{equation*}
Since $|\nabla\omega|\le C$ and $\frac{|\nabla\omega|^2}{r^2}\le C$ in $\overline{\Omega_{g_D}^+}$ by \eqref{Thm2.1-uniq-est}, we have
\begin{equation*}
 \int_{L_0+1}^L\int_{g_D(x_1)}^{R(x_1)}\frac{|\nabla\omega|^2}{r} drdx_1\le C.
\end{equation*}
Since $0<g_D(x_1)<R(x_1)<2$, we have
\begin{equation*}
\int_{L_0+1}^L\int_{g_D(x_1)}^{R(x_1)}|\nabla\omega|^2drdx_1\le 2\int_{L_0+1}^L\int_{g_D(x_1)}^{R(x_1)}\frac{|\nabla\omega|^2}{r} drdx_1\le C
\end{equation*}
for some constant $C>0$ independent of $L$.
Passing to the limit $L\rightarrow\infty$ yields
\begin{equation*}
\int_{L_0+1}^\infty\int_{g_D(x_1)}^{R(x_1)}|\nabla\omega|^2drdx_1\le C.
\end{equation*}

Hence
$$\int_L^{L+1}\int_{g_D(x_1)}^{R(x_1)}|\nabla\omega|^2drdx_1\rightarrow0\quad\mbox{as}\quad L\rightarrow\infty.$$
This proves the claim. The proof of Theorem~\ref{3D-MainThm}(b)
is completed by choosing $\bar\beta_2$ and $\sigma_2$ as
$$\bar\beta_2:=\min\left\{\bar\beta,\,\sqrt{\frac{1}{4C^{(\sharp)}}}\right\},\quad\sigma_2:=\min\left\{\sigma_1,\,\sigma_{\star\star},\,\frac{1}{4C^{(\sharp)}}\right\}.$$
\end{proof}

\vspace{.25in}
\noindent
\textbf{Acknowledgments.}
The research of Myoungjean Bae and Jong-Seo Yoon was supported in part by the Ministry of Education of the Republic of Korea and the National Research Foundation of Korea (NRF-RS-2025-00553734).

\end{document}